\documentclass[reqno]{amsart}

\usepackage[utf8]{inputenc}
\usepackage{amsmath}
\usepackage{amssymb}
\usepackage{amsthm}
\usepackage{verbatim}
\usepackage{dsfont}
\usepackage[all]{xy}
\usepackage[usenames]{color}
\usepackage{amsfonts}
\usepackage{latexsym}
\usepackage{graphicx}
\usepackage[usenames,dvipsnames]{xcolor}
\usepackage{tikz}
\usepackage{tikz-cd}
\usepackage{todonotes}
\usepackage{hyperref}

\usepackage[capitalize]{cleveref}
\crefname{theorem}{Theorem}{Theorems}
\crefname{fact}{Fact}{Facts}
\crefname{note}{Note}{Notes}
\crefname{lemma}{Lemma}{Lemmas}
\crefname{alg}{Algorithm}{Algorithms}
\crefname{remark}{Remark}{Remarks}
\crefname{example}{Example}{Examples}
\crefname{prop}{Proposition}{Propositions}
\crefname{conj}{Conjecture}{Conjectures}
\crefname{cor}{Corollary}{Corollaries}
\crefname{defn}{Definition}{Definitions}
\crefname{equation}{\!\!}{\!\!} %Remove spacing around phantom equation name

\usepackage[left=1in, top=1in, right=1in, bottom=1in]{geometry}

\usetikzlibrary{decorations.markings}
\usetikzlibrary{decorations.pathreplacing}
\usetikzlibrary{arrows,shapes,positioning}
\tikzstyle directed=[postaction={decorate,decoration={markings, mark=at position #1 with {\arrow[scale=1]{>}}}}]
\tikzstyle rdirected=[postaction={decorate,decoration={markings, mark=at position #1 with {\arrow[scale=1]{<}}}}]
\usepackage[all]{xy}
\usepackage{ytableau}
\newcommand{\webs}{\mathfrak{q}(n)\text{-}\mathbf{Web}_{\up}}
\newcommand{\qwebs}{\mathfrak{q}\text{-}\mathbf{Web}_{\up}}

\newcommand{\websupdown}{\mathfrak{q}(n)\text{-}\mathbf{Web}_{\up\down}}
\newcommand{\qwebsupdown}{\mathfrak{q}\text{-}\mathbf{Web}_{\up\down}}

\newcommand{\frakh}{\mathfrak{h}}

\newcommand{\beq}{\begin{equation}}
\newcommand{\eeq}{\end{equation}}
\newcommand{\bea}{\begin{eqnarray*}}
\newcommand{\eea}{\end{eqnarray*}}
\newcommand{\bi}{\begin{itemize}}
\newcommand{\ei}{\end{itemize}}
\newcommand{\be}{\begin{enumerate}}
\newcommand{\ee}{\end{enumerate}}
\newcommand{\bc}{\begin{center}}
\newcommand{\ec}{\end{center}}
\newcommand{\bt}{\begin{tikzpicture}}
\newcommand{\et}{\end{tikzpicture}}
\newcommand{\nfrac}[2]{\genfrac{}{}{0pt}{}{#1}{#2}}
\newcommand{\bma}{\begin{bmatrix}}
\newcommand{\ema}{\end{bmatrix}}

\newcommand{\Vt}{V}
\renewcommand{\star}{\circledast}
\newcommand{\fqt}{\mathfrak{q}}
\newcommand{\s}{\mathcal{S}}
\newcommand{\mods}{\mathfrak{q}(n)\text{-}\mathbf{Mod}_{\mathcal{S}}}
\newcommand{\sort}{\operatorname{sort}}
\newcommand{\modsupdown}{\mathfrak{q}(n)\text{-}\mathbf{Mod}_{\mathcal{S},\mathcal{S}^*}}
\newcommand{\A}{\mathbb{A}}
\newcommand{\OBC}{\mathcal{OBC}}
\newcommand{\AOBC}{\mathcal{AOBC}}
\newcommand{\ul}{\underline}

\newcommand{\lcap}{
\begin{tikzpicture}[baseline = 3pt, scale=0.5, color=\clr]
        \draw[-,thick] (1,0) to[out=up, in=right] (0.53,0.5) to[out=left, in=right] (0.47,0.5);
        \draw[->,thick] (0.49,0.5) to[out=left,in=up] (0,0);
\end{tikzpicture}
}
\newcommand{\lcup}{
\begin{tikzpicture}[baseline = 6pt, scale=0.5, color=\clr]
        \draw[-,thick] (1,1) to[out=down, in=right] (0.53,0.5) to[out=left, in=right] (0.47,0.5);
        \draw[->,thick] (0.49,0.5) to[out=left,in=down] (0,1);
\end{tikzpicture}
}

\newcommand{\swap}{
\begin{tikzpicture}[baseline = 3pt, scale=0.5, color=\clr]
        \draw[->,thick] (0,0) to[out=up, in=down] (1,1);
        \draw[->,thick] (1,0) to[out=up, in=down] (0,1);
\end{tikzpicture}
}

\newcommand{\cldot}{
\begin{tikzpicture}[baseline = 3pt, scale=0.5, color=\clr]
        \draw[->,thick] (0,0) to[out=up, in=down] (0,1);
        \draw (0,0.4)\wdot;
\end{tikzpicture}
}

\newcommand{\clr}{rgb:black,1;blue,4;red,1}
\newcommand{\wdot}{ node[circle, draw, color=\clr, fill=white, thick, inner sep=0pt, minimum width=4pt]{}}

\newcommand{\modt}{\mathfrak{q}(n)\text{-}\mathbf{Mod}_{V, V^{*}}}

\makeatletter

\newtheorem{theorem}{Theorem}[subsection]

\newtheorem{lemma}[theorem]{Lemma}
\newtheorem{proposition}[theorem]{Proposition}
\newtheorem{corollary}[theorem]{Corollary} 
\newtheorem{definition}[theorem]{Definition}

\newtheorem{remark}[theorem]{Remark}

\@addtoreset{equation}{section}

\newcommand{\End}{\operatorname{End}}
 
\newcommand{\Hom}{\operatorname{Hom}}

\newcommand{\id}{\text{id}}
\newcommand{\gl}{\mathfrak{gl}}

\newcommand{\Z}{\mathbb{Z}}

\newcommand{\C}{\k}

\newcommand{\cat}{\mathcal}

\newcommand{\arup}[1]{\stackrel{#1}{\longrightarrow}}
\newcommand{\larup}[1]{\stackrel{#1}{\longleftarrow}}
\newcommand{\onto}{\twoheadrightarrow}

\newcommand{\up}{\uparrow}
\newcommand{\down}{\downarrow}
\newcommand{\ob}[1]{\mathsf{#1}}

\newcommand{\unit}{\mathds{1}}

\newcommand{\col}{\operatorname{col}}
\newcommand{\q}{\mathfrak{q}}
\newcommand{\0}{\bar{0}}
\renewcommand{\1}{\bar{1}}

\newcommand{\ibar}{{\bar{\imath}}}
\newcommand{\jbar}{{\bar{\jmath}}}

\renewcommand{\k}{\mathbb{C}}

\newcommand{\ev}{\operatorname{ev}}
\newcommand{\coev}{\operatorname{coev}}
\newcommand{\Ser}{\operatorname{Ser}}

\newcommand{\fg}{\mathfrak{g}}
\newcommand{\fq}{\mathfrak{q}}
\newcommand{\fh}{\mathfrak{h}}

\newcommand{\fb}{\mathfrak{b}}
\newcommand{\Id}{\operatorname{Id}}

\newcommand{\svec}{\mathfrak{svec}}
\newcommand{\usvec}{{\underline{\mathfrak{svec}}}}

\newcommand{\gsmod}{\fg\operatorname{-smod}}

\newcommand{\p}[1]{|#1|}

\newcommand{\U}{U}

\newcommand{\frakS}{\Sigma}

\newcommand{\SP}{\mathcal{SP}}
\newcommand{\Udot}{\dot{\U}}
\newcommand{\bUdot}{\dot{\mathbf{U}}}
\newcommand{\etilde}{e}
\newcommand{\ftilde}{f}
\newcommand{\htilde}{h}
\newcommand{\uddoublearrow}{\vert}

\begin{document}

\title{Webs of Type Q}

\author{Gordon C. Brown }
\address{Fort Worth, TX}
\email{gcbrown@utexas.edu}
\author{Jonathan R. Kujawa}
\address{Department of Mathematics \\
          University of Oklahoma \\
          Norman, OK 73019}
\thanks{Research of the first author was partially supported by NSA grant H98230-11-1-0127.  Research of the second author was partially supported by NSA grant H98230-11-1-0127 and a Simons Collaboration Grant for Mathematicians.}\
\email{kujawa@math.ou.edu}

\date{\today}
%\keywords{TBD}

\begin{abstract}
We introduce web supercategories of type Q.  We describe the structure of these categories and show they have a symmetric braiding.  The main result of the paper shows these diagrammatically defined monoidal supercategories provide combinatorial models for the monoidal supercategories generated by the supersymmetric tensor powers of the natural supermodule for the Lie superalgebra of type Q. 
\end{abstract}

\maketitle  

%\tableofcontents

\section{Introduction} 

\subsection{Background}  The enveloping algebra of a complex Lie algebra, $U(\fg)$, is a Hopf algebra.  Consequently one can consider the tensor product and duals of $U(\fg)$-modules.  In particular, the category of finite-dimensional $U(\fg)$-modules is naturally a monoidal $\C$-linear category.  In trying to understand the representation theory of $U(\fg)$ it is natural to ask if one can describe this category (or well-chosen subcategories) as a monoidal category.  Kuperberg did this for the rank two Lie algebras by giving a diagrammatic presentation for the monoidal subcategory generated by the fundamental representations  \cite{Ku}.  The rank one case follows from work of Rumer, Teller, and Weyl \cite{RTW}.  For $\fg=\mathfrak{sl}(n)$ a diagrammatic presentation for the monoidal category generated by the fundamental representations was conjectured for $n=4$ by Kim in \cite{Ki} and for general $n$ by Morrison \cite{Mo}.  

Recently, Cautis-Kamnitzer-Morrison gave a complete combinatorial description of the monoidal category generated by the fundamental representations of $U(\mathfrak{sl}(n))$  \cite{CKM}; that is, the full subcategory of all modules of the form $\Lambda^{k_{1}}(V_{n}) \otimes \dotsb \otimes \Lambda^{k_{t}}(V_{n})$ for $t, k_{1}, \dotsc , k_{t} \in \Z_{\geq 0}$, where $V_{n}$ is the natural module for $\mathfrak{sl}(n)$.  Perhaps of greater significance is the method of proof used in \emph{loc cit}.  They show that a diagrammatic description of the category follows from a skew Howe duality between $\mathfrak{sl}(m)$ and $\mathfrak{sl}(n)$. More generally,  centralizing actions are now understood to naturally lead to presentations of monoidal categories. This philosophy has since been applied in a number of settings.  For example, Tubbenhauer-Vaz-Wedrich give a presentation of the monoidal category of $\gl (n)$-modules generated by the exterior and symmetric powers of the natural module \cite{TVW}. See \cite{RT, ST, QS} for further examples.

\subsection{Webs of Type Q}  The setting of this paper is $\C$-linear monoidal \emph{super}categories.  A supercategory is a category enriched over the category of superspaces and satisfying graded versions of the usual axioms for a $\C$-linear monoidal category.  For example, the interchange law now requires a sign according to the parity of the morphisms (see \cref{super-interchange} for a diagrammatic version).  

In \cref{D:UpwardWebs} we introduce the monoidal supercategory of (unoriented) webs of type $Q$, $\qwebs$.  The objects are words from the set 
\[
\left\{ \uparrow_{k} \mid k\in \Z_{\geq 0} \right\}.  
\]  The generating morphisms are certain diagrams which we call dots, merges, and splits.

\iffalse
\begin{equation*}
\xy
(0,0)*{
\begin{tikzpicture}[color=\clr, scale=1]
	\draw[color=\clr, thick, directed=1] (1,0) to (1,1);
	\node at (1,-0.15) {\scriptsize $k$};
	\node at (1,1.15) {\scriptsize $k$};
	\draw (1,0.5) \wdot;
\end{tikzpicture}
};
\endxy \ ,\quad\quad
\xy
(0,0)*{
\begin{tikzpicture}[color=\clr, scale=.35]
	\draw [color=\clr,  thick, directed=1] (0, .75) to (0,2);
	\draw [color=\clr,  thick, directed=.65] (1,-1) to [out=90,in=330] (0,.75);
	\draw [color=\clr,  thick, directed=.65] (-1,-1) to [out=90,in=210] (0,.75);
	\node at (0, 2.5) {\scriptsize $k\! +\! l$};
	\node at (-1,-1.5) {\scriptsize $k$};
	\node at (1,-1.5) {\scriptsize $l$};
\end{tikzpicture}
};
\endxy \ ,\quad\quad
\xy
(0,0)*{
\begin{tikzpicture}[color=\clr, scale=.35]
	\draw [color=\clr,  thick, directed=.65] (0,-0.5) to (0,.75);
	\draw [color=\clr,  thick, directed=1] (0,.75) to [out=30,in=270] (1,2.5);
	\draw [color=\clr,  thick, directed=1] (0,.75) to [out=150,in=270] (-1,2.5); 
	\node at (0, -1) {\scriptsize $k\! +\! l$};
	\node at (-1,3) {\scriptsize $k$};
	\node at (1,3) {\scriptsize $l$};
\end{tikzpicture}
}; .
\endxy
\end{equation*} 
\fi

 The $\Z_{2}$-grading is given by declaring dots to be of parity $\1$ and the merges and splits to be of parity $\0$.  As customary for diagrammatic categories, composition is given by vertical concatenation and tensor product is given by horizontal concatenation.  A diagram obtained by a finite sequence of these operations is called a \emph{web} and a general morphism is a linear combination of webs.  Note, our convention is to read diagrams from bottom to top.   For example, the following web is a morphism from $\uparrow_{4}\uparrow_{9}\uparrow_{6}\uparrow_{7}$ to $\uparrow_{6}\uparrow_{5}\uparrow_{1}\uparrow_{4}\uparrow_{8}\uparrow_{2}$:
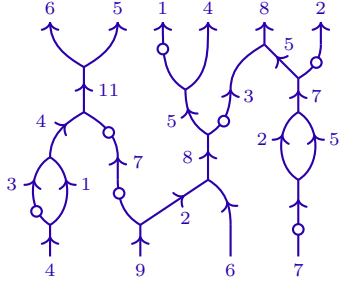
\[
\xy
(0,0)*{
\bt[color=\clr, scale=1.2]
	\node at (2,0) {\scriptsize $4$};
	\node at (3,0) {\scriptsize $9$};
	\node at (4,0) {\scriptsize $6$};
	\node at (4.75,0) {\scriptsize $7$};
	\node at (2,2.9) {\scriptsize $6$};
	\node at (2.75,2.9) {\scriptsize $5$};
	\node at (3.25,2.9) {\scriptsize $1$};
	\node at (3.75,2.9) {\scriptsize $4$};
	\node at (4.375,2.9) {\scriptsize $8$};
	\node at (5,2.9) {\scriptsize $2$};
	\draw [thick, directed=0.65] (2,0.15) to (2,0.5);
	\draw [thick, directed=0.65] (3,0.15) to (3,0.5);
	\draw [thick, ] (4,0.15) to (4,0.5);
	\draw [thick, directed=0.75] (4.75,0.15) to (4.75,1);
	\draw [thick, directed=0.65] (2,0.5) [out=30, in=330] to (2,1.25);
	\draw [thick, directed=0.65] (2,0.5) [out=150, in=210] to (2,1.25);
	\draw [thick, directed=0.65] (2,1.25) [out=90, in=210] to (2.375,1.75);
	\draw [thick, ] (3,0.5) [out=150, in=270] to (2.75,1);
	\draw [thick, directed=0.65] (3,0.5) to (3.75,1);
	\draw [thick, ] (2.75,1) to (2.75,1.25);
	\draw [thick, directed=1] (3.5,2) [out=30, in=270] to (3.75,2.75);
	\draw [thick, directed=1] (3.5,2) [out=150, in=270] to (3.25,2.75);
	\draw [thick, directed=0.01] (2.75,1.25) [out=90, in=330] to (2.375,1.75);
	\draw [thick, directed=0.65] (2.375,1.75) to (2.375,2.25);
	\draw [thick, directed=1] (2.375,2.25) [out=30, in=270] to (2.75,2.75);
	\draw [thick, directed=1] (2.375,2.25) [out=150, in=270] to (2,2.75);
	\draw [thick, directed=0.45] (4,0.5) [out=90, in=330] to (3.75,1);
	\draw [thick, directed=0.65] (3.75,1) to (3.75,1.5);
	\draw [thick, directed=0.65] (3.75,1.5) [out=150, in=270] to (3.5,2);
	\draw [thick, ] (3.75,1.5) [out=30, in=270] to (4,2);
	\draw [thick, directed=0.01] (4,2) [out=90, in=210] to (4.375,2.5);
	\draw [thick, directed=0.65] (4.75,1) [out=30, in=330] to (4.75,1.75);
	\draw [thick, directed=0.65] (4.75,1) [out=150, in=210] to (4.75,1.75);
	\draw [thick, directed=0.65] (4.75,1.75) to (4.75,2.125);
	\draw [thick, directed=0.65] (4.75,2.125) to (4.375,2.5);
	\draw [thick, ] (4.75,2.125) [out=30, in=270] to (5,2.5);
	\draw [thick, directed=1] (5,2.5) to (5,2.75);
	\draw [thick, directed=1] (4.375,2.5) to (4.375,2.75);
	\node at (1.575,0.95) {\scriptsize $3$}; 
	\node at (2.4,0.95) {\scriptsize $1$}; 
	\node at (2.975,1.2) {\scriptsize $7$}; 
	\node at (2.65,2) {\scriptsize $11$}; 
	\node at (1.9,1.65) {\scriptsize $4$}; 
	\node at (3.5,0.575) {\scriptsize $2$}; 
	\node at (3.525,1.25) {\scriptsize $8$}; 
	\node at (3.35,1.65) {\scriptsize $5$}; 
	\node at (4.2,1.925) {\scriptsize $3$}; 
	\node at (4.95,1.925) {\scriptsize $7$}; 
	\node at (5.15,1.45) {\scriptsize $5$}; 
	\node at (4.35,1.45) {\scriptsize $2$}; 
	\node at (4.65,2.5) {\scriptsize $5$}; 
	\draw (1.85,0.65) \wdot;
	\draw (2.765,0.85) \wdot;
	\draw (2.65,1.525) \wdot;
	\draw (4.75,0.45) \wdot;
	\draw (4.95,2.3) \wdot;
	\draw (3.925,1.65) \wdot;
	\draw (3.25,2.45) \wdot;
\et
};
\endxy
\]
Morphisms are subject to an explicit list of diagrammatic relations.  See \cref{D:UpwardWebs} for details. 

In \cref{D:orientedwebs} we introduce the oriented version, $\qwebsupdown $.  It is the monoidal supercategory whose objects are all words from the set 
\[
\left\{\up_{k}, \down_{k} \mid k \in \Z_{\geq 0} \right\}.
\]  The generating morphisms of $\websupdown$ are the dots, merges, and splits from before along with two new morphisms called cups and caps.
\iffalse
\[ 
\xy
(0,0)*{
\bt[scale=.35, color=\clr]
	\draw [ thick, looseness=2, directed=0.99] (1,2.5) to [out=270,in=270] (-1,2.5);
	\node at (-1,3) {\scriptsize $k$};
	\node at (1,3) {\scriptsize $k$};
\et
};
\endxy \ ,\quad\quad
\xy
(0,0)*{
\bt[scale=.35, color=\clr]
	\draw [ thick, looseness=2, directed=0.99] (1,2.5) to [out=90,in=90] (-1,2.5);
	\node at (-1,2) {\scriptsize $k$};
	\node at (1,2) {\scriptsize $k$};
\et
}; .
\endxy 
\]
\fi
 Cups and caps are declared to have parity $\0$.  Again morphisms are linear combinations of webs and are subject to an explicit list of diagrammatic relations.

In \cref{S:Upward Webs of Type Q,S:OrientedWebs} we describe the structure of these categories.  Let $\uparrow_{1}^{k}$ denote the $k$-fold tensor product $\uparrow_{1}\uparrow_{1}\dotsb \uparrow_{1}$. An important ingredient in our analysis is a thorough understanding of the endomorphism algebra 
\[
\End_{\qwebs}(\uparrow_{1}^{k}) \cong \End_{\qwebsupdown }(\uparrow_{1}^{k})
\]
for every $k \geq 1$.  We show it is isomorphic to the Sergeev algebra, $\Ser_{k}$, and that it admits a natural analogue of the Jones-Wenzl projector which we call the \emph{Clasp idempotent}, $Cl_{k}$.  Using these we also introduce braiding isomorphisms and show these categories are symmetric braided monoidal supercategories.

\subsection{The Lie superalgebra of type Q}  Let $\fq (n)$ be the Lie superalgebra of $2n \times 2n$ complex matrices of the form 
\begin{equation*}
\fq(n)=\left\{\left(\begin{matrix}A & B \\
                            B & A 
\end{matrix} \right)\right\},
\end{equation*} where $A$ and $B$ are $n \times n$ matrices and where the Lie bracket is given by the graded matrix commutator.  This is a type $Q$ Lie superalgebra.  

The representations of $\fq (n)$ do not have a classical analogue. Despite the important early work done by Penkov-Serganova, Brundan, and others to obtain character formulas and other information (see \cite{PS,Br} and references therein), the representation theory in type Q remains mysterious.  For example, only very recently the structure of category $\mathcal{O}$ for $\fq(n)$ became clear thanks to the work of Chen \cite{Chen}, Cheng-Kwon-Wang \cite{CKW}, and Brundan-Davidson \cite{BD2,BD3}. 

If $V_{n}$ denotes the natural $\fq (n)$-module given by column vectors of height $2n$, then we can consider the degree $k$ supersymmetric tensors $S^{k}(V_{n})$.  Let $\mods$ and $\modsupdown$ denote the monoidal supercategory generated by the supersymmetric tensors, and supersymmetric tensors and their duals, respectively.  Note we allow all morphisms, not just grading preserving ones. This causes new phenomena to arise.  For example, $V_{n}$ is a simple $\fq (n)$-module and yet Schur's Lemma in this setting implies $\End_{\fq (n)}(V_{n})$ is two-dimensional. Similarly, for all $k\geq 1$ there is an odd isomorphism between $S^{k}(V_{n})$ and $\Lambda^{k}(V_{n})$.  Hence, our categories also contain the exterior powers of the natural module.  The main goal of the present work is to show the supercategories $\qwebs$ and $\qwebsupdown$ provide combinatorial models for $\mods$ and $\modsupdown$, respectively, in the spirit of the work of Kuperberg, Cautis-Kamnitzer-Morrison, et al.  
\subsection{Main Results}

Using a Howe duality for $\fq (m)$  and $\fq (n)$ first introduced by Cheng-Wang \cite{CW1} and the philosophy of Cautis-Kamnitzer-Morrison, we provide essentially surjective, full functors of monoidal supercategories 
\begin{align*} \notag
\Psi_{n}: & \qwebs  \to \mods, \\
\Psi_{n}: & \qwebsupdown   \to \modsupdown.
\end{align*}  On objects $\Psi_{n}(\uparrow_{k}) = S^{k}(V_{n})$ and $\Psi_{n}(\downarrow_{k}) = S^{k}(V_{n})^{*}$.  For example, 
\[
\Psi_{n}(\uparrow_{2}\downarrow_{3}\uparrow_{5}\uparrow_{3}) = S^{2}(V_{n}) \otimes S^{3}(V_{n})^{*} \otimes S^{5}(V_{n}) \otimes S^{3}(V_{n}).
\]  The description of $\Psi_{n}$ on morphisms is equally explicit.  See \cref{psi-functor,T:Psiupdown}.

For $n \geq 1$ and $k=(n+1)(n+2)/2$, we follow Sergeev and introduce an explicit quasi-idempotent $e_{\lambda(n)} \in \End_{\qwebs}(\uparrow_{1}^{k}) \cong \End_{\qwebsupdown}(\uparrow_{1}^{k})$.  Define $\webs$ and $\websupdown$ to be, respectively, the monoidal supercategories given by the same generators and relations as $\qwebs$ and $\qwebsupdown$ along with the additional relation,
\[
e_{\lambda(n)}=0.
\]

The main result of the paper is given in \cref{T:maintheorem}. We show the above functors induce equivalences of monoidal supercategories 
\begin{align}\label{E:mainresultintro} \notag
\Psi_{n}: & \webs  \xrightarrow{\cong} \mods, \\
\Psi_{n}: & \websupdown   \xrightarrow{\cong} \modsupdown.
\end{align}
In short, the categories $\webs$ and $\websupdown$ provide complete combinatorial models for these categories of $\fq (n)$-modules.  Recall, the Karoubi envelope of $\modsupdown$ is the category of all finite-dimensional $\fq(n)$-modules which appear as a direct summand of a direct sum of tensor products of such modules. In principle, all such $\fq (n)$-modules can be studied using the combinatorics of webs.

\subsection{Related Work}  This paper is part of a larger program to develop diagrammatic categories of type $Q$. That is, monoidal supercategories which have odd morphisms which square to a non-zero scalar multiple of the identity.  In \cite{CK} Comes and the second author introduced the oriented Brauer-Clifford category $\OBC$ and its affine analogue $\AOBC$.  The category $\OBC$ can be viewed as the special case of oriented webs where all edges have label $1$.  One corollary of the present work is an analogue of \cref{E:mainresultintro} for $\OBC$.  In other work the first author uses the category $\qwebs$ to give a complete combinatorial description of the full subcategory of permutation modules for the Sergeev algebra \cite{Br}.  Forthcoming work by the authors and Davidson will use quantum webs of type $Q$ to obtain analogues of the results of this paper for the quantized enveloping algebra associated to $\fq (n)$ \cite{BDK}.  Finally, in a somewhat different direction, Comes and the second author will show diagrammatic categories of type $Q$ can be used to categorify the twisted Heisenberg algebra at arbitrary integral level \cite{CK2}.

\subsection{Acknowledgements}  The authors are pleased to acknowledge Jonathan Comes, Nicholas Davidson, and Daniel Tubbenhauer for helpful discussions during this project.  Some of those conversations occured while the second author visited the Hausdorff Institute and the Mittag-Leffler Institute. He is grateful for their hospitality.

\section{Monoidal supercategories}\label{section: monoidal supercategories}
In this section we give a brief introduction to monoidal supercategories following \cite{BE, CK}. We refer the reader to \emph{op.~cit.}~for further details.

\subsection{Superspaces}\label{SS:superspaces} Throughout the ground field will be the complex numbers, $\k$.  A \emph{superspace} $V = V_{\0} \oplus V_{\1}$ is a $\Z_{2}$-graded $\k$-vector space.  Elements of $V_{\0}$ (resp.\ $V_{\1}$) are said to have parity $\0$ or to be \emph{even} (resp.\ parity $\1$ or \emph{odd}). Given a homogeneous element $v \in V$ we write $\p{v} \in \Z_{2}$ for the parity of the element. Given two superspaces $V$ and $W$, the set of all linear maps $\Hom_{\k}(V,W)$ is naturally $\Z_{2}$-graded by declaring that $f:V \to W$ has parity $\varepsilon \in \Z_{2}$ if $f(V_{\varepsilon'}) \subseteq V_{\varepsilon+\varepsilon'}$ for all $\varepsilon'\in  \Z_{2}$. Let $\svec$ denote the category of all superspaces with $\Hom_\svec(V,W)=\Hom_\k(V,W)$; in particular, we allow maps which do not preserve the $\Z_{2}$-grading.  Let $\usvec$ denote the underlying purely even category; that is, the subcategory of $\svec$ consisting of all $\k$-superspaces, but only the grading preserving linear maps.

Given superspaces $V$ and $W$, the tensor product $V \otimes W$ as vector spaces is also naturally a superspace with $\Z_{2}$-grading given by declaring $\p{v\otimes w}= \p{v}+\p{w}$ for all homogeneous $v \in V$ and $w \in W$.   The tensor product of linear maps between superspaces is defined via $(f\otimes g)(v\otimes w)=(-1)^{\p{g}\p{v}}f(v)\otimes g(w)$. This gives $\usvec$ (but \emph{not} $\svec$) the structure of a monoidal category with $\unit=\k$ (viewed as superspace concentrated in even parity). The graded flip map $v\otimes w\mapsto(-1)^{\p{v}\p{w}}w\otimes v$ gives $\usvec$ the structure of a symmetric monoidal category. Here and elsewhere we write the formula only for homogeneous elements with the general case given by extending linearly.

\subsection{Monoidal supercategories}\label{SS: monoidal supercats}    By a \emph{supercategory} we mean a category enriched in $\usvec$. Similarly, a \emph{superfunctor} is a functor enriched in $\usvec$.  Unless otherwise stated, all functors in this paper will be superfunctors which preserve the parity of homogeneous morphisms.  Given two supercategories $\cat{A}$ and $\cat{B}$, there is a supercategory $\cat{A}\boxtimes\cat{B}$ whose objects are pairs $(\ob a,\ob b)$ of objects $\ob a\in\cat{A}$ and $\ob b\in\cat{B}$ and whose morphisms are given by the tensor product of superspaces $\Hom_{\cat{A}\boxtimes\cat{B}}((\ob a,\ob b),(\ob a',\ob b'))=\Hom_\cat{A}(\ob a,\ob a')\otimes\Hom_\cat{B}(\ob b,\ob b')$ with composition defined using the symmetric braiding on $\usvec$: 
\begin{equation}\label{E:superinterchange}
(f\otimes g)\circ(h\otimes k)=(-1)^{\p{g}\p{h}}(f\circ h)\otimes(g\circ k).
\end{equation}  

By a \emph{monoidal supercategory} we mean a supercategory $\cat{A}$ equipped with a functor $-\otimes-:\cat{A}\boxtimes\cat{A}\to\cat{A}$, a unit object $\unit$, and even supernatural isomorphisms $(-\otimes-)\otimes-\arup{\sim}-\otimes(-\otimes-)$ and $\unit\otimes-\arup{\sim}-\larup{\sim}-\otimes\unit$ called \emph{coherence maps} satisfying certain axioms analogous to the ones for a monoidal category. A monoidal supercategory is called \emph{strict} if its coherence maps are identities.  
A \emph{monoidal functor} between two monoidal supercategories $\cat{A}$ and $\cat{B}$ is a functor $F:\cat{A}\to\cat{B}$ equipped with an even supernatural isomorphism $(F-)\otimes (F-)\arup{\sim}F(-\otimes-)$ and an even isomorphism $\unit_\cat{B}\arup{\sim}F\unit_\cat{A}$ satisfying axioms analogous to the ones for a monoidal functor. 

A \emph{braided monoidal supercategory} is a monoidal supercategory $\cat{A}$ equipped with a $\usvec$-enriched version of a braiding. More precisely, let $T:\cat{A}\boxtimes\cat{A}\to\cat{A}$ denote the functor defined on objects by $(\ob a,\ob b)\mapsto \ob b\otimes\ob a$ and on morphisms by $f\otimes g\mapsto (-1)^{\p{f}\p{g}}g\otimes f$. A braiding on $\cat{A}$ is a supernatural isomorphism $\gamma:-\otimes-\to T$ satisfying the usual hexagon axioms. A \emph{symmetric monoidal supercategory} is a braided monoidal supercategory $\cat{A}$ with $\gamma_{\ob a,\ob b}^{-1}=\gamma_{\ob b,\ob a}$ for all objects $\ob a,\ob b\in\cat{A}$.  In this paper all braidings will be even.

Given a monoidal supercategory $\cat{A}$ and an object $\ob a\in\cat{A}$, by a \emph{(left) dual} to $\ob a$ we mean an object $\ob a^*$ equipped with \emph{evaluation} and \emph{coevaluation} morphisms $\ev_{\ob{a}}: \ob a^*\otimes \ob a\to\unit$ and $\coev_{\ob{a}}: \unit\to\ob a\otimes\ob a^*$, respectively, in which $\ev_{\ob{a}}$ and $\coev_{\ob{a}}$ have the same parity and satisfy the super version of the usual adjunction axioms.  In our case the evaluation and coevaluation maps will be even and, hence, satisfy the usual adjunction axioms (see \cref{straighten-zigzag} for a diagrammatic example). For example, given a finite dimensional superspace $V$ with homogeneous basis $\{v_i~|~i\in I\}$, then $V^*=\Hom_\k(V,\k)$ with the evaluation and coevaluation given by $f\otimes v\mapsto f(v)$ and $1\mapsto\sum_{i\in I}v_i\otimes v^*_i$ respectively, where $v^*_i\in V^*$ is defined by $v_i^*(v_j)=\delta_{i,j}$.  A monoidal supercategory in which every object has a (left) dual is called \emph{(left) rigid}.

The following two examples will be relevant for this paper.

\begin{itemize} 
    \item[(i)] The tensor product and braiding defined in \cref{SS:superspaces} give $\svec$ the structure of a symmetric monoidal supercategory with $\unit=\k$ (viewed as a superspace concentrated in parity $\0$).  The symmetric braiding  $M \otimes N \to N \otimes M$ is given by the graded flip map.  The full subcategory of finite-dimensional superspaces is rigid.

    \item[(ii)] Given a Lie superalgebra $\fg = \fg_{\0}\oplus \fg_{\1 }$ over the field $\k$, let $\gsmod$ denote the category of all $\fg$-supermodules.  That is, superspaces $M = M_{\0}\oplus M_{\1}$ with an action by $\fg$ which respects the grading in the sense that $\fg_{\varepsilon}.M_{\varepsilon'} \subseteq M_{\varepsilon+\varepsilon'}$.  The tensor product $M \otimes N$ has action given by $x.(m\otimes n) = (x.m) \otimes n + (-1)^{\p{x}\p{m}}m \otimes (x.n)$ for all homogeneous $x \in \fg$, $m\in M$, and $n\in N$ and the graded flip map provides a symmetric braiding.  The unit object $\unit$ is the ground field $\k$ concentrated in parity $\0$ and with trivial $\fg$-action.  In this way $\gsmod$ is a symmetric monoidal supercategory. The full subcategory of finite-dimensional $\fg$-supermodules is rigid with the action given on $M^{*}$ by $(x.f)(m) = -(-1)^{\p{x}\p{f}}f(x.m)$.
\end{itemize}

When working with monoidal supercategories it will sometimes be convenient to use the following notation. Given objects $\ob a$ and $\ob b$ in a monoidal supercategory, we write $\ob a\ob b:=\ob a\otimes\ob b$. We will also write ${\ob a}^r:=\underbrace{\ob a\otimes\cdots\otimes\ob a}_{r\text{ times}}$.

\subsection{String calculus}\label{SS:stringcalculus}
There is an established string calculus for strict monoidal supercategories.  We briefly describe it here and refer the reader to \cite{BE} for details. A morphism $f:\ob a\to \ob b$ is drawn as
\begin{equation*}
    \begin{tikzpicture}[baseline = 12pt,scale=0.5,color=\clr,inner sep=0pt, minimum width=11pt]
        \draw[-,thick] (0,0) to (0,2);
        \draw (0,1) node[circle,draw,thick,fill=white]{$f$};
        \draw (0,-0.2) node{$\ob a$};
        \draw (0, 2.3) node{$\ob b$};
    \end{tikzpicture}
    \qquad\text{or simply as}\qquad
    \begin{tikzpicture}[baseline = 12pt,scale=0.5,color=\clr,inner sep=0pt, minimum width=11pt]
        \draw[-,thick] (0,0) to (0,2);
        \draw (0,1) node[circle,draw,thick,fill=white]{$f$};
    \end{tikzpicture}
\end{equation*}
when the objects are left implicit.  Notice that the convention used in this paper is to read diagrams from bottom to top. 
The products of morphisms $f\otimes g$ and $f\circ g$ are given by horizontal and vertical stacking respectively:
\begin{equation*}
    \begin{tikzpicture}[baseline = 19pt,scale=0.5,color=\clr,inner sep=0pt, minimum width=11pt]
        \draw[-,thick] (0,0) to (0,3);
        \draw[-,thick] (2,0) to (2,3);
        \draw (1,1.5) node[color=black]{$\otimes$};
        \draw (0,1.5) node[circle,draw,thick,fill=white]{$f$};
        \draw (2,1.5) node[circle,draw,thick,fill=white]{$g$};
        % \draw (0,-0.2) node{$\ob c$};
        % \draw (0, 3.3) node{$\ob d$};
        % \draw (2,-0.2) node{$\ob a$};
        % \draw (2, 3.3) node{$\ob b$};
    \end{tikzpicture}
    ~=~
    \begin{tikzpicture}[baseline = 19pt,scale=0.5,color=\clr,inner sep=0pt, minimum width=11pt]
        \draw[-,thick] (0,0) to (0,3);
        \draw[-,thick] (2,0) to (2,3);
        \draw (0,1.5) node[circle,draw,thick,fill=white]{$f$};
        \draw (2,1.5) node[circle,draw,thick,fill=white]{$g$};
        % \draw (0,-0.2) node{$\ob c$};
        % \draw (0, 3.3) node{$\ob d$};
        % \draw (2,-0.2) node{$\ob a$};
        % \draw (2, 3.3) node{$\ob b$};
    \end{tikzpicture}
    ~,\qquad
    \begin{tikzpicture}[baseline = 19pt,scale=0.5,color=\clr,inner sep=0pt, minimum width=11pt]
        \draw[-,thick] (0,0) to (0,3);
        \draw[-,thick] (2,0) to (2,3);
        \draw (1,1.5) node[color=black]{$\circ$};
        \draw (0,1.5) node[circle,draw,thick,fill=white]{$f$};
        \draw (2,1.5) node[circle,draw,thick,fill=white]{$g$};
        % \draw (0, 3.3) node{$\ob c$};
        % \draw (0,-0.2) node{$\ob b$};
        % \draw (2, 3.3) node{$\ob b$};
        % \draw (2,-0.2) node{$\ob a$};
    \end{tikzpicture}
    ~=~
    \begin{tikzpicture}[baseline = 19pt,scale=0.5,color=\clr,inner sep=0pt, minimum width=11pt]
        \draw[-,thick] (0,0) to (0,3);
        \draw (0,2.2) node[circle,draw,thick,fill=white]{$f$};
        \draw (0,0.8) node[circle,draw,thick,fill=white]{$g$};
        % \draw (0,-0.2) node{$\ob c$};
        % \draw (0.3,1.5) node{$\ob b$};
        % \draw (0, 3.3) node{$\ob a$};
    \end{tikzpicture}
    ~.
\end{equation*}
Pictures involving multiple products should be interpreted by \emph{first composing horizontally, then composing vertically}. For example, 
\begin{equation*}
    \begin{tikzpicture}[baseline = 19pt,scale=0.5,color=\clr,inner sep=0pt, minimum width=11pt]
        \draw[-,thick] (0,0) to (0,3);
        \draw[-,thick] (2,0) to (2,3);
        \draw (0,2.2) node[circle,draw,thick,fill=white]{$f$};
        \draw (2,2.2) node[circle,draw,thick,fill=white]{$g$};
        \draw (0,0.8) node[circle,draw,thick,fill=white]{$h$};
        \draw (2,0.8) node[circle,draw,thick,fill=white]{$k$};
    \end{tikzpicture}
\end{equation*}
should be interpreted as $(f\otimes g)\circ(h\otimes k)$. In general, this is \emph{not} the same as $(f\circ h)\otimes(g\circ k)$ because of the \emph{super-interchange law} given in \cref{E:superinterchange}.  Diagrammatically:
\begin{equation}\label{super-interchange}
    \begin{tikzpicture}[baseline = 19pt,scale=0.5,color=\clr,inner sep=0pt, minimum width=11pt]
        \draw[-,thick] (0,0) to (0,3);
        \draw[-,thick] (2,0) to (2,3);
        \draw (0,2) node[circle,draw,thick,fill=white]{$f$};
        \draw (2,1) node[circle,draw,thick,fill=white]{$g$};
    \end{tikzpicture}
    ~=~
    \begin{tikzpicture}[baseline = 19pt,scale=0.5,color=\clr,inner sep=0pt, minimum width=11pt]
        \draw[-,thick] (0,0) to (0,3);
        \draw[-,thick] (2,0) to (2,3);
        \draw (0,1.5) node[circle,draw,thick,fill=white]{$f$};
        \draw (2,1.5) node[circle,draw,thick,fill=white]{$g$};
    \end{tikzpicture}
    ~=(-1)^{\p{f}\p{g}}~
    \begin{tikzpicture}[baseline = 19pt,scale=0.5,color=\clr,inner sep=0pt, minimum width=11pt]
        \draw[-,thick] (0,0) to (0,3);
        \draw[-,thick] (2,0) to (2,3);
        \draw (0,1) node[circle,draw,thick,fill=white]{$f$};
        \draw (2,2) node[circle,draw,thick,fill=white]{$g$};
    \end{tikzpicture}
    ~.
\end{equation}

\subsection{Supercategories and idempotent superalgebras}\label{SS:IdempotentSuperalgebras}  Let $\mathcal{C}$ be a small $\k$-linear supercategory and let $\Lambda$ be a complete irredundant set of objects.  To $\mathcal{C}$ we can associate an associative $\k $-superalgebra $A_{\mathcal{C}}$.  As a superspace,
\[
A_{\mathcal{C}} = \oplus_{\lambda, \mu \in \Lambda} \Hom_{\mathcal{C}}(\lambda, \mu),
\] and the multiplication is given by extending composition of morphisms linearly.  In particular, the set of identity morphisms, $\left\{1_{\lambda}: \lambda \to \lambda \mid \lambda \in \Lambda \right\}$, is a distinguished set of orthogonal even idempotents which make $A_{\mathcal{C}}$ a locally unital superalgebra.  For each even $\k $-linear functor $F: \mathcal{C} \to \svec$ one can define an $A_{\mathcal{C}}$-supermodule by setting 
\[
M= \oplus_{\lambda \in \Lambda} M_{\lambda},
\] where $M_{\lambda}:=F(\lambda)$.  Given $g \in \Hom_{\mathcal{C}}(\lambda, \mu)$ we have the $\k$-linear morphism $F(g): F(\lambda) \to F(\mu)$.  The action of $g \in \Hom_{\mathcal{C}}(\lambda, \mu)$ on $M_{\gamma}$ is given by $g.x = F(g)(x)$ if $\gamma = \lambda$ and $g.x=0$ if $\gamma \neq  \lambda$; the action is given in general by extending linearly.  This makes $M$ into an $A_{\mathcal{C}}$-supermodule.  Moreover, we see that $M$ is locally unital since $1_{\lambda}$ acts on $M$ by projecting onto  $M_{\lambda}$ for each $\lambda \in \Lambda$.

Conversely, if $A$ is a locally unital superalgebra with a distinguished set of orthogonal even idempotents $\left\{1_{\lambda} \mid \lambda \in \Lambda \right\}$, then one can define a supercategory $\mathcal{C}_{A}$ as follows. The objects are the elements of the set $\Lambda$ and $\Hom_{\mathcal{C}_{A}}(\lambda, \mu) := 1_{\mu}A1_{\lambda}$ for all $\lambda, \mu \in \Lambda$.  The composition of morphisms is given by multiplication in $A$.  If $M$ is a locally unital $A$-supermodule, then there is an even $k$-linear functor $F_{M}: \mathcal{C}_{A} \to \svec$ given by $F(\lambda) = 1_{\lambda}M$ and, given $g \in \Hom_{\mathcal{C}_{A}}(\lambda, \mu) = 1_{\mu}A1_{\lambda}$, $F(g): 1_{\lambda}M \to 1_{\mu}M$ is the linear map given by the action of $g$.

These constructions are mutually inverse.  We will freely switch between the categorical and superalgebra points of view when convenient.  

\subsection{Monoidal supercategories and tensor ideals}\label{SS:MonoidalSupercatsandIdeals}  Suppose $\mathcal{C}$ is a monoidal supercategory. A \emph{tensor ideal} $\mathcal{I}$ of $\mathcal{C}$ consists of a subsuperspace $\mathcal{I}(\ob{a},\ob{b}) \subseteq \Hom_{\mathcal{C}}(\ob{a},\ob{b})$ for each pair of objects $\ob{a}$, $\ob{b}$ in $\mathcal{C}$, such that for all objects $\ob{a}$, $\ob{b}$, $\ob{c}$, $\ob{d}$ we have $h \circ g \circ f \in \mathcal{I}(\ob{a},\ob{d})$ whenever $f \in \Hom_{\mathcal{C}}(\ob{a},\ob{b})$, $g \in \Hom_{\mathcal{C}}(\ob{b},\ob{c})$, $h \in \Hom_{\mathcal{C}}(\ob{c},\ob{d})$ and $g \otimes \id_{c} \in \mathcal{I}(\ob{a}\otimes \ob{c}, \ob{b}\otimes \ob{c})$ and $ \id_{c} \otimes g \in \mathcal{I}(\ob{c}\otimes \ob{a}, \ob{c}\otimes \ob{b})$.

The quotient $\mathcal{C}/\mathcal{I}$ of $\mathcal{C}$ by the tensor ideal $\mathcal{I}$ is the supercategory with the same objects as $\mathcal{C}$ and morphisms given by $\Hom_{\mathcal{C}/\mathcal{I}}(\ob{a}, \ob{b}) = \Hom_{\mathcal{C}}(\ob{a},\ob{b})/\mathcal{I}(\ob{a},\ob{b})$.  It is straightforward to check that the quotient of a monoidal supercategory by a tensor ideal is again a monoidal supercategory. Moreover, the quotient of a braided (resp.\  symmetric) monoidal supercategory is again braided (resp.\  symmetric).  

Since the intersection (defined in the obvious way) of tensor ideals is again a tensor ideal, there is a unique minimal tensor ideal which contains a given fixed collection of morphisms.  We call this the tensor ideal generated by this collection of morphisms.  If $F: \mathcal{C} \to \mathcal{D}$ is an even functor of monoidal supercategories, then the kernel, $\mathcal{K}$, of $F$ given by setting $\mathcal{K}(\ob{a},\ob{b})= \left\{f \in \Hom_{\mathcal{C}}(\ob{a},\ob{b}) \mid F(f) =0 \right\}$ is a tensor ideal of $\mathcal{C}$ and there is an induced even functor $\tilde{F}: \mathcal{C}/\mathcal{K} \to \mathcal{D}$.

If $\mathcal{C}$ is only a supercategory, then one can instead consider an ideal $\mathcal{I}$ as above which satisfies only the condition on composition.  One can then form the quotient supercategory $\mathcal{C}/\mathcal{I}$ as above and similar non-monoidal statements still hold.
\section{Lie superalgebras of type Q}\label{S:Liesuperalgebras}

\subsection{Representations of superalgebras}\label{SS:repsofsuperalgebras}  We give a brief overview of the representation theory of associative unital superalgebras.  Details can be found in, for example, \cite{BK, CW2}.  Let $A$ be a finite-dimensional superalgebra and let $S$ be a simple $A$-supermodule.  We call $S$ type $M$ if it remains irreducible as an $A$-\emph{module} (ie.\  if $\Z_{2}$-gradings are ignored).  Otherwise we call $S$ type $Q$.  In the latter case $\End_{A}(S)$ is two-dimensional and, in particular, $S$ admits an odd involution. Correspondingly, there are two types of simple finite-dimensional associative superalgebras: the superalgebra of linear endomorphisms of an $m|n$-dimensional superspace $V$, $M(V)$, and the superalgebra of linear endomorphisms of an $n|n$-dimensional superspace $V$ which preserve an odd involution, $Q(V)$.  In both cases $V$ is the unique\footnote{Here and elsewhere, unless explicitly stated, we do not distinguish between a supermodule and its parity shift.} simple supermodule and is of type M and type Q, respectively.  More generally, if $A$ is a semisimple finite-dimensional associative unital superalgebra, then 
\[
A \cong \bigoplus_{V \text{ of type M}} M(V) \oplus \bigoplus_{V \text{ of type Q}} Q(V)
\] as superalgebras where the direct sum is over a complete, irredundant set of simple supermodules for $A$.

If $A$ and $B$ are associative superalgebras and $S$ and $T$ are simple $A$- and $B$-supermodules, respectively, then  $A\otimes B$ is naturally a superalgebra and the outer tensor product $S \boxtimes T$ is naturally a supermodule for $A \otimes B$.  If both $S$ and $T$ are of type $M$, then $S \boxtimes T$ is a simple $A \otimes B$-supermodule of type $M$. If exactly one of $S$ and $T$ are of type $Q$, then $S \boxtimes T$ is a simple $A \otimes B$-supermodule of type $Q$. If both $S$ and $T$ are of type $Q$, then $S \boxtimes T$ is a direct sum of two isomorphic $A \otimes B$-supermodules of type $M$.  We write $S \star T$ for this simple $A \otimes B$-supermodule.  We extend notation by declaring $S \star T$ to be the simple $A \otimes B$-supermodule $S \boxtimes T$ in the other cases.

\subsection{The Lie superalgebra \texorpdfstring{$\fq(n)$}{q(n)}}\label{SS:Liesuperalgebras}

Fix a superspace $V= V_{\0} \oplus V_{\1}$ with $\dim_{\k}(V_{\0}) = \dim_{\k}(V_{\1})=n$.   Fix a homogeneous basis $v_{1}, \dotsc , v_{n}, v_{\bar{1}}, \dotsc , v_{\bar{n}}$ with $\p{v_{i}}=\0$ and $\p{v_{\bar{i}}}=\1$ for $i=1, \dotsc , n$.  We write $I=I(n|n)$ for the index set $\left\{1, \dotsc , n, \bar{1}, \dotsc , \bar{n} \right\}$ and $I_{0}=I_{0}(n|n)$ for the set $\left\{1, \dotsc , n \right\}$.  There is an involution on $I$ given by $i \mapsto \bar{i}$ and $\bar{i}\mapsto i$ for any $i=1, \dotsc , n$. It is convenient to adopt the convention that $\bar{\bar{i}}=i$ for all $i \in I$.  Let $c: V\to V$ be the odd linear map given by $c(v_{i})=(-1)^{\p{v_{i}}}\sqrt{-1}v_{\bar{i}}$ for all $i \in I$.

The vector space of all linear endomorphisms of $V$, $\gl (V)$, is naturally $\Z_{2}$-graded as in \cref{SS:superspaces}.  Furthermore, $\gl (V)$ is a Lie superalgebra under the graded commutator bracket; this, by definition, is given by $[x,y]=xy-(-1)^{\p{x}\p{y}}yx$ for all homogeneous $x,y \in \gl (V)$.  For $i,j \in I$ we write $e_{i,j} \in \gl (V)$ for the linear map $e_{i,j}(v_{k})= \delta_{j,k}v_{i}$. These are the matrix units and they form a homogeneous basis for $\gl(V)$ with $\p{e_{i,j}}=\p{v_{i}}+\p{v_{j}}$.

By definition $\fq (V)$ is the Lie subsuperalgebra of $\gl (V)$ given by 
\[
\fq (V) = \left\{x \in \gl (V) \mid [x,c]=0 \right\}.
\]  Then $\fq (V)$ has a homogenous basis given by $e_{i,j}^{\0}:=e_{i,j}+e_{\ibar,\jbar}$ and $e_{i,j}^{\1 }:=e_{\ibar,j}+e_{i,\jbar}$ for $1\leq i,j\leq n$.  Note that $\p{e_{i,j}^{\varepsilon}} = \varepsilon$ for all $1 \leq i,j \leq n$ and $\varepsilon \in \Z_{2}$.

Realized as matrices with respect to our choice of basis for $V$, $\gl (V)=\gl (n|n)$ is all $2n \times 2n$ matrices with entries from $\k$.  In this matrix realization $\fq (V)= \fq(n)$ is the subspace 
\begin{equation}\label{E:qndef}
\fq (V)= \fq(n)=\left\{\left(\begin{matrix}A & B \\
                            B & A 
\end{matrix} \right)\right\}  \subseteq \gl (n|n) .
\end{equation} Then $\q(n)_{\0}$ (resp.\ $\q (n)_{\1}$) is the subspace of all such matrices with $B=0$ (resp.\ $A=0$).

 Fix the Cartan subalgebra of $\fh \subseteq \fq(n)$ consisting of matrices as in \cref{E:qndef} with $A$ and $B$ diagonal.  For $i=1, \dotsc , n$, let $\varepsilon_{i}: \fh_{\0} \to \k$ be defined by $\varepsilon_{i}(e^{\0}_{j,j}) =\delta_{i,j}$.  Set $X(T) = X(T_{n}) = \oplus_{i=1}^{n} \Z\varepsilon_{i} \subseteq \fh_{\0}^{*}$.  Fix the Borel subalgebra $\fb \subseteq \fq (n)$ consisting of matrices with $A$ and $B$ upper triangular. Corresponding to this choice, the set of roots, positive roots, and simple roots are $\left\{\alpha_{i,j} = \varepsilon_{i}-\varepsilon_{j} \mid 1 \leq i,j \leq n \right\}$,  $\left\{\alpha_{i,j} = \varepsilon_{i}-\varepsilon_{j} \mid 1 \leq i<j \leq n \right\}$, $\left\{\alpha_{i} = \varepsilon_{i}-\varepsilon_{i+1} \mid 1 \leq i \leq n-1 \right\}$, respectively.  %If $W$ is a $\fq$-supermodule, then a homogeneous vector $w \in W$ is called a highest weight vector of highest weight $\lambda \in \fh _{\0}^{*}$ if $w \in W_{\lambda}$ and $U(\fb)w \subseteq \k w$.

A $\fq (n)$-supermodule is a $\Z_{2}$-graded vector space $W = W_{\0}\oplus W_{\1}$ such that the action of $\fq(n)$ respects the $\Z_{2}$-grading and satisfies graded versions of the usual axioms for a supermodule for a Lie algebra.   A \emph{weight supermodule} for $\fq(n)$ is a $\fq(n)$-supermodule $W$ for which there is a superspace decomposition $W=\bigoplus_{\lambda\in\fh_{\0}^*}W_\lambda$ where $W_\lambda$ is the $\lambda$-\emph{weight space}, $$W_\lambda:=\{w\in W \mid h.w=\lambda(h)w\text{ for all }h\in\frakh_{0}\}.$$   We identify $X(T)$ with $\Z^{n}$ via the map $\sum_{i} a_{i}\varepsilon_{i} \leftrightarrow (a_{1}, \dotsc , a_{n})$.  All supermodules considered in this paper will be weight supermodules with weights lying in $X(T)$.

Given a superspace $W$ and $d\geq 0$, let $S^{d}(W)$ denote the $d$th supersymmetric power of $W$ and $\Lambda^{d}(W)$ the $d$th skew supersymmetric power.  If $W$ is a module for a Lie superalgebra, then $S^{d}(W)$ and $\Lambda^{d}(W)$ are naturally modules for the Lie superalgebra via the coproduct.  Furthermore, if $\Pi W$ denotes the parity shift of $W$, then $S^{d}(W) \cong  \Lambda^{d}(\Pi W)$ as supermodules.  In this paper we will be interested in the case when $W=V_{n}$ is the natural module for $\fq (n)$.  In this case $V_{n} \cong \Pi V_{n}$ by an odd map.  Consequently, there is no loss of generality in only considering the symmetric powers.

\subsection{The Sergeev algebra}\label{SS:SergeevAlgebra}
For each $k\geq 1$ the \emph{Sergeev} \emph{superalgebra} $\Ser_{k}$ is the associative, unital superalgebra generated by the even elements $s_1,\dots,s_{k-1}$ and odd elements $c_1,\dots,c_k$ subject to the relations:

\begin{equation}\label{sergeev-relations}\begin{gathered}
c_i^2=1,\quad c_ic_j=-c_jc_i,\\
s_i^2=1,\quad s_is_j=s_js_i\quad\text{if }i\neq j\pm 1,\quad s_is_{i+1}s_i=s_{i+1}s_is_{i+1},\\
 c_is_j=s_jc_i\quad\text{if }i\neq j\pm1,\quad s_ic_i=c_{i+1}s_i,\quad s_ic_{i+1}=c_is_i
 \end{gathered}
\end{equation} for all admissible $i,j$.  As a superspace $\Ser_k \cong C_{k} \otimes \k \Sigma_{k}$, where $C_{k}$ is the Clifford superalgebra on odd generators $c_{1}, \dotsc , c_{k}$, and $\k \Sigma_{k}$ is the group algebra of the symmetric group viewed as a superalgebra concentrated in parity $\0$.  Then $C_{k} \cong C_{k}\otimes 1$ and $\C \Sigma_{k} \cong 1 \otimes \k \Sigma_{k}$ as superalgebras and we have the mixed relation $sc_{i}= c_{s(i)}s$ for all $s \in \Sigma_{k}$.

 A \emph{strict partition} of $k$ is a nonincreasing sequence of nonnegative integers $\lambda=(\lambda_1,\lambda_2,\lambda_3,\dots)$ such that $\sum_i \lambda_i=k$ and $\lambda_i=\lambda_{i+1}$ implies $\lambda_i=0$. Let $\SP(k)$ denote the set of all strict partitions of $k$.  For a partition $\lambda \in \SP (k)$, set $|\lambda| = \sum_{i}\lambda_{i}$,  $\ell(\lambda)$ equal to the number of nonzero parts in $\lambda$, and 
\[
\delta(\lambda) = \begin{cases} 0, & \ell(\lambda) \text{ is even};\\
                                1,  & \ell(\lambda) \text{ is odd}.
\end{cases}
\]

It is known (e.g.\ see \cite[Lemma 3.6]{BK}) that $\Ser_{k}$ is semisimple and the simple supermodules are labelled by the strict partitions of $k$.  If we write $T^{\lambda}$ for the simple supermodule labelled by $\lambda \in \SP (k)$, then $T^{\lambda}$ is of type $Q$ or $M$ if $\delta(\lambda)=1$ or $0$, respectively.  That is, as discussed in \cref{SS:repsofsuperalgebras},  
\begin{equation}\label{E:ArtinWedderburn}
\Ser_{k}\simeq\bigoplus_{\substack{\lambda\in\SP(k)\\ \delta(\lambda)=0}}M(T^\lambda)\oplus\bigoplus_{\substack{\lambda\in\SP(k)\\ \delta(\lambda)=1}}Q(T^\lambda)
\end{equation} as superalgebras.

We next recall the definition from \cite{Se2} of certain quasi-idempotents $e_{\lambda}\in\Ser_{k}$ parameterized by $\SP (k)$.  To every strict partition $\lambda$ we associate the \emph{shifted frame} $[\lambda]$, the array of squares with $\lambda_i$ squares in row $i$ for $1\leq i\leq l(\lambda)$, such that row $i$ has been shifted to the right $i-1$ units, e.g. $$[(4,3,1)]=\ytableausetup{smalltableaux}\ydiagram{4,1+3,2+1} \ .$$ For strict partitions $\lambda,\mu$, we write $\lambda\subseteq\mu$ if $\lambda_i\leq\mu_i$ for all $i$, or, equivalently, if $[\lambda]$ is a shifted subframe of $[\mu]$.

For $1\leq i<j\leq k-1,$ define the elements $s_{i,j},\tau_{i,j},\pi_j\in\Ser_{k}$ by letting $s_{i,j}\in \Sigma_k$ be the transposition interchanging $i$ and $j$ and setting $\pi_{1}=0$, and 
\begin{align}\label{E:OddJMelements}
\tau_{i,j} &=\frac{1}{\sqrt{2}}(c_i-c_j)s_{i,j}, \\
\quad\pi_j&=\tau_{1,j}+\tau_{2,j}+\cdots+\tau_{j-1,j}.
\end{align}
 The $\pi_j$ are the \emph{odd Jucys-Murphy elements} of \cite{Se2}.

For $\lambda\in\SP(k)$, let $T_\lambda$ be the \emph{canonical filling} of $[\lambda]$ obtained by numbering the squares $1,2,3,\dots$ left-to-right in each row, starting from the top and working down; for example, $$T_{(4,3,1)}=\begin{ytableau}1&2&3&4\\ \none&5&6&7\\ \none&\none&8\end{ytableau} \ .$$  Given $T_{\lambda}$, let $\col(i)$ be the number of the column occupied by $i$ in $T_\lambda$.  Define $a_\lambda\in\Ser_{k}$ by $$a_\lambda=\prod_{i=1}^k\left(\frac{\col(i)(\col(i)+1)}{2}-\pi_i^2\right).$$ For example, $$a_{(4,3,1)}=1\cdot(3-\pi_2^2)(6-\pi_3^2)(10-\pi_4^2)(3-\pi_5^2)(6-\pi_6^2)(10-\pi_7^2)(6-\pi_8^2).$$ Let $R_\lambda\subseteq\Sigma_k$ denote the row stabilizer of $T_\lambda$ and set $b_\lambda\in\Ser_{k}$ to be $$b_\lambda=\sum_{\sigma\in R_\lambda}\sigma.$$ Finally, define $e_\lambda\in\Ser_{k}$ to be
\begin{equation}\label{E:Sergeevidempotent}
e_\lambda:=a_\lambda b_\lambda.
\end{equation}
 By \cite[Corollary 3.3.4]{Se2}, each $e_\lambda$ is quasi-idempotent. Moreover, $\Ser_{k}e_{\lambda}$ is a direct summand of the $T^{\lambda}$-isotypy component of the regular supermodule, $\Ser_{k}$. In particular, $e_{\lambda}$ lies in the summand of \cref{E:ArtinWedderburn} corresponding to $T^{\lambda}$. 

One of Sergeev's quasi-idempotents will play a distinguished role in what follows.  For $n\in\Z_{>0}$, define the strict partition $\lambda(n):=(n+1,n,n-1,\dots,3,2,1)$. The shifted frame of this partition will be an inverted staircase, e.g. $$[\lambda(2)]=\ydiagram{3,1+2,2+1} \ .$$  For $k = (n+1)(n+2)/2$, there is the corresponding quasi-idempotent 
\begin{equation}\label{E:staircaseidempotent}
e_{\lambda(n)} \in \Ser_{k}.
\end{equation}

%Since $\lambda(n)$ is the unique strict partition with length exceeding $n$ and minimal size, every strict partition $\mu$ with $l(\mu)>n$ has $\lambda\subseteq\mu$. 

\subsection{Sergeev Duality}\label{SS:Sergeevdualities}

We have the following \emph{Sergeev duality}  established by Sergeev \cite{Se1} (see also \cite[Section 3.4.1]{CW2}).

 \begin{theorem}\label{T:sergeev-duality} Let $V=V_{n}$ be the natural $\fq(n)$-supermodule and let $V^{\otimes k}$ denote its $k$-fold tensor product. 

\begin{enumerate}
\item If $r \neq s$, then $\Hom_{\fq (n)}\left(V^{\otimes r}, V^{\otimes s} \right)=0$.
\item There is a homomorphism 
\[
\psi: \Ser_k \to \End_{\fq (n)}\left(V^{\otimes k} \right)
\] given by:
\begin{gather*}
\psi(s_{i})(v_{1} \otimes \dotsb \otimes v_{i} \otimes v_{i+1} \otimes \dotsb \otimes v_{k})= (-1)^{\p{v_{i}}\p{v_{i+1}}}v_{1} \otimes \dotsb \otimes v_{i+1} \otimes v_{i} \otimes \dotsb \otimes v_{k} \\
\psi(c_{i})(v_{1} \otimes \dotsb \otimes v_{i} \otimes \dotsb \otimes v_{k}) = (-1)^{\p{v_{1}}+\dotsb + \p{v_{i-1}}} v_{1} \otimes \dotsb \otimes c(v_{i}) \otimes \dotsb \otimes v_{k}
\end{gather*}
\item The homomorphism $\psi$ is surjective.
\item As a $\U(\q(n))\otimes\Ser_{k}$-supermodule we have the following multiplicity-free decomposition:
\[
V^{\otimes k}\simeq\bigoplus_{\substack{\lambda \in \SP (k)\\ \ell(\lambda)\leq n }} L_n(\lambda)\star T^{\lambda}.
\] Here $L_{n}(\lambda)$ is the simple $\q (n)$-supermodule of highest weight $\lambda \in X(T_{n})$ and $T^{\lambda}$ the simple $\Ser_{k}$-supermodule labelled by $\lambda \in \SP (k)$.  These simple supermodules are both of type $Q$ if $\delta(\lambda)=1$ and are both of type $M$ otherwise.
\end{enumerate}
 \end{theorem}

We have the following description of the kernel of the homomorphism $\psi$.

 \begin{proposition}\label{sergeev-kernel}
 For each $k, n \geq 1$ the kernel of the homomorphism $\psi:\Ser_k\to\End_{\q(n)}(V_{n}^{\otimes k})$ is generated as a two-sided ideal by 
\begin{equation}\label{}
\left\{e_{\lambda} \mid \lambda \in \SP (k), \ell (\lambda) > n \right\}.
\end{equation}  Moreover, $\psi$ is an isomorphism if and only if $k < (n+1)(n+2)/2$.
 \end{proposition}
 
 \begin{proof} In light of the decomposition into simple superalgebras given in \cref{SS:repsofsuperalgebras}, a direct summand given there will be in the kernel of $\psi$ if and only if $T^{\lambda}$ does not appear as a summand of the $\Ser_{k}$-supermodule $V_{n}^{\otimes k}$.  By \cref{E:ArtinWedderburn} this happens if and only if $\ell (n) > n$.  On the other hand, since $e_{\lambda}$ lies in the simple summand corresponding to $T^{\lambda}$, it follows that the kernel is generated by $\left\{e_{\lambda} \mid \lambda \in \SP (k), \ell(\lambda) > n \right\}$, as claimed.  The last statement follows from \cref{T:sergeev-duality} and observing there exist strict partitions of length greater than $n$ if and only if $k$ is greater than or equal to the given bound.
\end{proof}

\subsection{Howe duality of type Q}\label{commuting-actions}

We now describe a Howe duality of type Q first introduced by Cheng-Wang \cite{CW1}.  Let $\Vt_{m}$ and  $V_{n}$ denote the natural supermodules of $\fq (m)$ and $\fq (n)$, respectively.  Recall, $c:V_{m}\to V_{m}$ and $c:V_{n} \to V_{n}$ are the odd involutions used to define $\fq (m)$ and $\fq (n)$.

Let $\Vt_m\boxtimes V_n$ denote the  $U(\fqt(m))\otimes U(\q(n))$-supermodule given by the outer tensor product of $\Vt_{m}$ and $V_{n}$. Since both are of type $Q$, $\Vt_m\boxtimes V_n \cong \Vt_m\star V_n \oplus \Vt_m\star V_n$ as discussed in \cref{SS:repsofsuperalgebras}.  This decomposition can be made overt as follows.  There is an even supermodule involution $p:=\sqrt{-1}c\boxtimes c : \Vt_{m} \boxtimes V_{n} \to \Vt_{m} \boxtimes V_{n}$ given on homogeneous pure tensors by $p(v_{a} \otimes v_{b}) = (-1)^{\p{v_{a}}}\sqrt{-1} c(v_{a}) \otimes c(v_{b})$.  Since this map is an involution we may decompose $\Vt_{m} \boxtimes V_{n}$ into $\pm 1$-eigenspaces. This decomposition is as $U(\fqt(m))\otimes U(\q(n))$-supermodules and it provides the decomposition described above.  We choose $\Vt_m\star V_n$ to be the $+1$ eigenspace of $p$. 

We can describe $\Vt_m\star V_n$ completely explicitly as follows.  Set 
\begin{align}\label{E:symmetricgenerators}
x_{i,j}&=v_{i} \otimes v_{j}+\sqrt{-1}v_{\overline{i}}\otimes v_{\overline{j}},\\ \notag
 y_{i,j}&=v_{i} \otimes v_{\overline{j}}-\sqrt{-1}v_{\overline{i}} \otimes v_{j}. 
\end{align}

Then $\Vt_{m} \star V_{n}$ has a homogeneous basis $\{x_{i,j},y_{i,j}\mid i\in I_{0}(m|m), j\in I_{0}(n|n)\}$ with the parity of $x_{i,j}$ (resp.\ $y_{i,j})$ equal to $\0$ (resp.\ $\1$). As a $U(\fqt(m))$-supermodule $\Vt_m\star V_n \cong  \Vt_m^{\oplus n}$ with the  $U(\fqt(m))$ acting on $x_{i,j}$ (resp.\ $y_{i,j}$) as on $v_{i}$ (resp.\ $-\sqrt{-1}v_{\bar{i}}$).  As a $U(\q(n))$-supermodule $\Vt_m\star V_n \cong  V_n^{\oplus m}$ with  $U(\q(n))$ acting on $x_{i,j}$ (resp.\ $y_{i,j}$) as on $v_{j}$ (resp.\ $v_{\bar{j}}$). 

Note that since $U(\fqt(m))$ and $U(\fqt(n))$ are Hopf superalgebras, there is a natural Hopf superalgebra structure on $U(\fqt(m))\otimes U(\q(n))$.  In particular, the symmetric superalgebra 
\begin{equation*}
\s:=S \left(\Vt_m\star V_n \right)
\end{equation*}
is a $U(\fqt(m))\otimes U(\q(n))$-supermodule via this coproduct.  As a superalgebra $\s$ is the free supercommutative superalgebra generated by $\left\{x_{i,j}, y_{i,j} \mid i \in I_{0}(m|m), j\in I_{0}(n|n) \right\}$.  

A direct calculation verifies that $\s$ is a weight supermodule for $\fqt(m)$ with weights lying in 
\begin{equation*}
X(T_{m})_{\geq 0} := \left\{\lambda = \sum_{i=1}^{m} \lambda_{i}\varepsilon_{i} \in X(T_{m}) \mid \lambda_{i} \geq 0 \text{ for $i=1, \dotsc , m$} \right\}.
\end{equation*}
  Since the actions of $U(\fqt(m))$ and $U(\fqt(m))$ commute, the decomposition of $\s$ into weight spaces for $U(\fqt(m))$ is a decomposition into $\fq (n)$-supermodules. Given $t\geq 0$ and a tuple of nonnegative integers, $\lambda = (\lambda_{1}, \lambda_{2}, \dotsc , \lambda_{t})$, for short write 
\[
S^{\lambda}=S^{\lambda_{1}}(V_{n}) \otimes \dotsb \otimes  S^{\lambda_{t}}(V_{n}).
\] 
A straightforward computation using \cref{E:symmetricgenerators} verifies the following.
 
\begin{lemma}\label{weight-space-isomorphism}  Let 
\[
\s = \bigoplus_{\lambda \in X(T_{m})} \s_{\lambda}\quad 
\] be the decomposition into weight spaces with respect to the Cartan subalgebra of $\fqt (m)$.  Then $\s_{\lambda} \neq 0$ if and only if $\lambda \in X(T_{m})_{\geq 0}$.
  Furthermore, if $\lambda = \sum_{i=1}^{m}\lambda_{i}\varepsilon_{i} \in X(T_{m})_{\geq 0}$, then as a $\q(n)$-supermodule  
\[
\s_{\lambda} \cong S^{\lambda_1}(V_n)\otimes\cdots\otimes S^{\lambda_m}(V_n).
\]
\end{lemma}

\subsection{The supercategory  \texorpdfstring{$\bUdot(\fq (m))$}{Udot}}\label{SS:IdempotentAlgebra}
We now introduce  a $\C$-linear supercategory $\bUdot(\fq (m))$ via generators and relations. When writing compositions of morphisms we often write them as products (e.g.\ $fg = f \circ g$).   To lighten notation we use the same name for morphisms between different objects and leave the objects implicit when there is no risk of confusion (e.g.\ when a statement is true for all morphisms for which it makes sense regardless of domain and range). With this convention in mind, if $f$ is a family of morphisms of the same name which can be composed and $k \geq 1$, then  $f^{(k)}$ denotes the morphism $f^{k}/k!$.  Finally, when a statement requires that we specify the domain and/or range, we do so by pre/post-composing with the relevant identity morphisms. In what follows recall $\left\{\alpha_{i} \mid i=1, \dotsc ,m-1 \right\} \subset X(T_{m})$ denotes the set of simple roots for $\fq (m)$ (see \cref{SS:Liesuperalgebras}).

\begin{definition}\label{D:dotU2}

Let $\bUdot(\q(m))$ be the $\C$-linear supercategory with set of objects $X(T_{m})$ and with the morphisms generated by $e_i,e_{\bar{i}}: \lambda \to \lambda+\alpha_{i}$, $f_{i}, f_{\bar{i}}: \lambda \to \lambda-\alpha_{i}$, and $h_{\bar{j}}: \lambda \to \lambda$ for all  $\lambda\in X(T_{m})$, $i=1, \dotsc, m-1$, $j=1, \dotsc , m$.  For all $i,j$ the parities of $e_{i}$ and $f_{i}$ are even and the parities of $e_{\bar{i}}$, $f_{\bar{i}}$, and $h_{\bar{j}}$ are odd. We write $1_{\lambda}: \lambda \to \lambda$ for the identity morphism. 

The morphisms in $\bUdot(\fq (m))$ are subject to the following relations for all objects $\lambda$ and all admissible $i,j$:
\begin{align}\label{Q1} \notag
h_{\overline{i}}^21_\lambda&=\lambda_i1_\lambda,\quad \\ \tag{Q1}
h_{\overline{i}}h_{\overline{j}}&=-h_{\overline{j}}h_{\overline{i}}\quad\text{if }i\neq j,
\end{align}
\begin{align}\label{Q3} \notag
h_{\overline{i}}e_j-e_jh_{\overline{i}}&= \delta_{i,j}e_{\overline{j}} - \delta_{i,j+1}e_{\overline{j}},\\ \notag
h_{\overline{i}}f_j-f_jh_{\overline{i}}&= -\delta_{i,j}f_{\overline{j}}+\delta_{i,j+1} f_{\overline{j}},\\ \tag{Q3} 
h_{\overline{i}}e_{\overline{j}}+e_{\overline{j}}h_{\overline{i}}&= (\delta_{i,j}+\delta_{i,j+1})e_{j},\\ \notag
h_{\overline{i}}f_{\overline{j}}+f_{\overline{j}}h_{\overline{i}}&= (\delta_{i,j}+\delta_{i,j+1})f_{j}, \notag
\end{align}
\begin{align}\label{Q4} \notag
(e_if_j-f_je_i)1_\lambda&=\delta_{i,j}(\lambda_i-\lambda_{i+1})1_\lambda,\\ \notag
(e_{\overline{i}}f_{\overline{j}}+f_{\overline{j}}e_{\overline{i}})1_\lambda&=\delta_{i,j}(\lambda_i+\lambda_{i+1})1_\lambda,\\ \tag{Q4}
e_{\overline{i}}f_j-f_je_{\overline{i}}&=\delta_{i,j}(h_{\overline{i}}-h_{\overline{i+1}}),\\ \notag
e_if_{\overline{j}}-f_{\overline{j}}e_i&=\delta_{i,j}(h_{\overline{i}}-h_{\overline{i+1}}), \notag
\end{align}
\begin{align}\label{Q5}\notag
e_ie_{\overline{j}}-e_{\overline{j}}e_i&=e_{\overline{i}}e_{\overline{j}}+e_{\overline{j}}e_{\overline{i}}=0, \quad  \text{if }i\neq j\pm1, \\ \tag{Q5}
f_if_{\overline{j}}-f_{\overline{j}}f_i&=f_{\overline{i}}f_{\overline{j}}+f_{\overline{j}}f_{\overline{i}}=0, \quad  \text{if }i\neq j\pm1,\\ \notag
e_ie_j-e_je_i&=f_if_j-f_jf_i=0\quad\text{if }|i-j|>1, \notag
\end{align}
\begin{align}\label{Q6} \notag
e_ie_{i+1}-e_{i+1}e_i&=e_{\overline{i}}e_{\overline{i+1}}+e_{\overline{i+1}}e_{\overline{i}}, \\\notag 
e_ie_{\overline{i+1}}-e_{\overline{i+1}}e_i&=e_{\overline{i}}e_{i+1}-e_{i+1}e_{\overline{i}},\\ \tag{Q6}
f_if_{i+1}-f_{i+1}f_i&=f_{\overline{i}}f_{\overline{i+1}}+f_{\overline{i+1}}f_{\overline{i}},\\  \notag 
f_if_{\overline{i+1}}-f_{\overline{i+1}}f_i&=f_{\overline{i}}f_{i+1}-f_{i+1}f_{\overline{i}},  \notag 
\end{align}
\begin{align}\label{Q7} \notag
e_i^{(2)}e_j-e_ie_je_i+e_je_i^{(2)}&=0 \quad\text{if }i=j\pm 1, \\ \notag
e_{\overline{i}}e_ie_j-e_{\overline{i}}e_je_i-e_ie_je_{\overline{i}}+e_je_ie_{\overline{i}}&=0\quad\text{if }i=j\pm 1,\\ \tag{Q7}
f_i^{(2)}f_j-f_if_jf_i+f_jf_i^{(2)}&=0 \quad\text{if }i=j\pm 1,\\ \notag
f_{\overline{i}}f_if_j-f_{\overline{i}}f_jf_i-f_if_jf_{\overline{i}}+f_jf_if_{\overline{i}}&=0\quad\text{if }i=j\pm 1. \notag
\end{align}
\end{definition}

\begin{remark}\label{R:idempotentalgebra}
As discussed in \cref{SS:IdempotentSuperalgebras}, the supercategory $\bUdot (\fq (m))$ defines a locally unital superalgebra
\[
\Udot(\q(m))=\bigoplus_{\lambda,\mu\in X(T_{m})} \Hom_{\bUdot(\q(m))}(\lambda, \mu).
\] Using the presentation given in \cite[Proposition 2.1]{DW} one sees that $\Udot (\fq (m))$ is an idempotent version of the enveloping superalgebra $U(\fq (m))$. 
\end{remark}

\begin{definition}\label{D:Udotgeq0}   Let $\bUdot(\q(m))_{\geq 0}$ be  the quotient of  $\bUdot(\q(m))$ given by setting $1_{\lambda}=0$ for all $\lambda \not\in X(T_{m})_{\geq 0}$ (i.e., it is $\bUdot (\q(m))/\mathcal{I}$, where $\mathcal{I}$ is the ideal generated by $\left\{ 1_{\lambda} \mid \lambda \notin X(T_{m})_{\geq 0} \right\}$).  On the algebra side, the corresponding superalgebra $\Udot (\fqt (m))_{\geq 0}$ is the quotient of $\Udot(\fqt (m)) $ by the ideal generated by $\left\{ 1_{\lambda} \mid \lambda \notin X(T_{m})_{\geq 0} \right\}$. 

\end{definition}

Define $\mods$ to be the monoidal supercategory of $\q (n)$-supermodules tensor generated by $\left\{ S^{p}(V_{n})\mid p \geq 0 \right\}$.  That is, it is the full subcategory of $\q(n)$-supermodules consisting of objects of the form 
\[
S^{p_{1}}(V_{n}) \otimes \dotsb \otimes S^{p_{t}}(V_{n}),
\] where $t, p_{1}, \dotsc , p_{t} \in  \Z_{\geq 0}$. In particular, by convention, $S^{0}(V_{n})$ is the trivial supermodule, $\k$.

Since $\s$ is a weight supermodule for $U(\fqt (m))$ with all weights lying in $X(T_{m})$, it is a supermodule for the idempotent version of the enveloping superalgebra, $\Udot(\fqt (m))$.  Moreover, since the weight space $\s_{\lambda}$ equals zero whenever $\lambda \not\in X(T_{m})_{\geq 0}$, this representation factors through and defines a representation of $\Udot (\fqt (m))_{\geq 0}$.  In an abuse of notation we write 
\[
\phi : \Udot (\fqt (m))_{\geq 0} \to \End_{\fq (n)}\left(\s \right)
\] for this representation.  As discussed in \cref{SS:IdempotentAlgebra}, the notion of a supermodule for a locally unital superalgebra with a distingushed set of idempotents is equivalent to having a functor $\bUdot (\fqt (m))_{\geq 0} \to \svec$.  The fact that the action of $\Udot (\fqt (m))_{\geq 0}$ on $\s$ commutes with the action of $\fq (n)$ implies that this functor can be viewed as having codomain the category of $\fq (n)$-supermodules.  The existence of this functor is summarized in the following result.

\begin{proposition}\label{Phi-up}
For every $m,n \geq 1$ there exists a functor of supercategories 
\begin{equation*}
\Phi_{m,n}: \bUdot(\fqt(m))_{\geq 0} \to \mods.
\end{equation*}
On objects, 
\[
\Phi_{m,n}(\lambda) = S_{\lambda} \cong S^{\lambda_{1}}(V_{n}) \otimes \dotsb \otimes S^{\lambda_{m}}(V_{n}),
\]
for all $\lambda = \sum_{i=1}^{m}\lambda_{i}\varepsilon_{i} \in X(T_{m})_{\geq 0}$.  On a morphism $x \in \Hom_{\bUdot (\fqt (m))}(\lambda, \mu)$,
\[
\Phi_{m,n} (x) = \phi(x).
\]
\end{proposition}

\begin{remark}\label{R:Compatibility}  Let $m', m$ be positive integers with $m' \geq m$.  Given an element $\lambda = (\lambda_{1}, \dotsc , \lambda_{m}) \in X(T_{m})$ we can view $\lambda$ as the element $\lambda= (\lambda_{1}, \dotsc , \lambda_{m}, 0, \dotsc , 0) \in X(T_{m'})$ by extending by $m'-m$ zeros.  With this identification in mind, there is a functor of supercategories, 
\[
\Theta_{m,m'}: \bUdot (\fqt(m)) \to \bUdot (\fqt(m')),
\]
 given by sending the objects and generating morphisms of $\bUdot (\fqt(m))$ to the objects and morphisms of the same name in $\bUdot (\fqt(m'))$.  Moreover this functor defines a functor, which we call by the same name,
\[
\Theta_{m,m'}: \bUdot (\fqt(m))_{\geq 0} \to \bUdot (\fqt(m'))_{\geq 0}.
\]  For any fixed $n\geq 1$ the functors $\Phi_{m,n}$ and $\Phi_{m',n}$ are compatible in the sense that $\Phi_{m',n}\circ \Theta_{m,m'}$ and $\Phi_{m,n}$ are canonically isomorphic. 
\end{remark}

\section{Upward Webs}\label{S:Upward Webs of Type Q}

\subsection{Upward Webs of Type Q}\label{SS:UpwardWebs}
  
\begin{definition}\label{D:UpwardWebs}  Let $\qwebs$ be the monoidal supercategory generated by the countable collection of objects $\{\up_{k} \mid k \in \Z_{\geq 0}\}$, and with generating morphisms
 
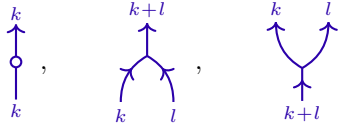
\begin{equation*}
\xy
(0,0)*{
\begin{tikzpicture}[color=\clr, scale=1]
	\draw[color=\clr, thick, directed=1] (1,0) to (1,1);
	\node at (1,-0.15) {\scriptsize $k$};
	\node at (1,1.15) {\scriptsize $k$};
	\draw (1,0.5) \wdot;
\end{tikzpicture}
};
\endxy \ ,\quad\quad
\xy
(0,0)*{
\begin{tikzpicture}[color=\clr, scale=.35]
	\draw [color=\clr,  thick, directed=1] (0, .75) to (0,2);
	\draw [color=\clr,  thick, directed=.65] (1,-1) to [out=90,in=330] (0,.75);
	\draw [color=\clr,  thick, directed=.65] (-1,-1) to [out=90,in=210] (0,.75);
	\node at (0, 2.5) {\scriptsize $k\! +\! l$};
	\node at (-1,-1.5) {\scriptsize $k$};
	\node at (1,-1.5) {\scriptsize $l$};
\end{tikzpicture}
};
\endxy \ ,\quad\quad
\xy
(0,0)*{
\begin{tikzpicture}[color=\clr, scale=.35]
	\draw [color=\clr,  thick, directed=.65] (0,-0.5) to (0,.75);
	\draw [color=\clr,  thick, directed=1] (0,.75) to [out=30,in=270] (1,2.5);
	\draw [color=\clr,  thick, directed=1] (0,.75) to [out=150,in=270] (-1,2.5); 
	\node at (0, -1) {\scriptsize $k\! +\! l$};
	\node at (-1,3) {\scriptsize $k$};
	\node at (1,3) {\scriptsize $l$};
\end{tikzpicture}
};
\endxy
\end{equation*}
 for $k,l\in\Z_{>0}$.  We call these \emph{dots}, \emph{merges}, and \emph{splits}, respectively. The $\Z_{2}$-grading is given by declaring  merges and splits to have parity $\0$ and dots to have parity $\1$. The morphisms in $\qwebs$ are subject to the relations \cref{associativity,digon-removal,dot-collision,dots-past-merges,dumbbell-relation,square-switch,square-switch-dots,double-rungs-1,double-rungs-2}.

\end{definition}

We use the diagrammatic calculus described in \cref{SS:stringcalculus} to work with morphisms in $\qwebs$.  In particular, diagrams are read bottom to top.  Vertical concatenation is composition and horizontal concatenation is the tensor product. In this way any finite sequence of these operations applied to merges, splits, and dots yields a diagram which is a morphism in $\qwebs$.  We call such a diagram a \emph{web}. If $\ob{a}=(a_{1}, \dotsc , a_{r})$ is a tuple of nonnegative integers, then we write $\up_{\ob{a}}$ for the object $\up_{a_{1}} \up_{a_{2}}\dotsb \up_{a_{r}}$.  If $\ob{a}=(a_{1}, \dotsc , a_{r})$ and $\ob{b}=(b_{1}, \dotsc , b_{s})$ are tuples of positive integers, then we say a web is of type $\ob{a}\to \ob{b}$ if it is a morphism from $\up_{\ob{a}} \to \up_{\ob{b}}$.  For example, the following is a web of type $(4,9,6,7)\to(6,5,1,4,8,2)$:
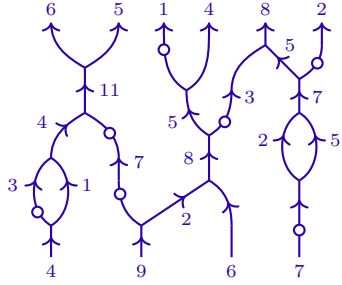
\[
\xy
(0,0)*{
\bt[color=\clr, scale=1.2]
	\node at (2,0) {\scriptsize $4$};
	\node at (3,0) {\scriptsize $9$};
	\node at (4,0) {\scriptsize $6$};
	\node at (4.75,0) {\scriptsize $7$};
	\node at (2,2.9) {\scriptsize $6$};
	\node at (2.75,2.9) {\scriptsize $5$};
	\node at (3.25,2.9) {\scriptsize $1$};
	\node at (3.75,2.9) {\scriptsize $4$};
	\node at (4.375,2.9) {\scriptsize $8$};
	\node at (5,2.9) {\scriptsize $2$};
	\draw [thick, directed=0.65] (2,0.15) to (2,0.5);
	\draw [thick, directed=0.65] (3,0.15) to (3,0.5);
	\draw [thick, ] (4,0.15) to (4,0.5);
	\draw [thick, directed=0.75] (4.75,0.15) to (4.75,1);
	\draw [thick, directed=0.65] (2,0.5) [out=30, in=330] to (2,1.25);
	\draw [thick, directed=0.65] (2,0.5) [out=150, in=210] to (2,1.25);
	\draw [thick, directed=0.65] (2,1.25) [out=90, in=210] to (2.375,1.75);
	\draw [thick, ] (3,0.5) [out=150, in=270] to (2.75,1);
	\draw [thick, directed=0.65] (3,0.5) to (3.75,1);
	\draw [thick, ] (2.75,1) to (2.75,1.25);
	\draw [thick, directed=1] (3.5,2) [out=30, in=270] to (3.75,2.75);
	\draw [thick, directed=1] (3.5,2) [out=150, in=270] to (3.25,2.75);
	\draw [thick, directed=0.01] (2.75,1.25) [out=90, in=330] to (2.375,1.75);
	\draw [thick, directed=0.65] (2.375,1.75) to (2.375,2.25);
	\draw [thick, directed=1] (2.375,2.25) [out=30, in=270] to (2.75,2.75);
	\draw [thick, directed=1] (2.375,2.25) [out=150, in=270] to (2,2.75);
	\draw [thick, directed=0.45] (4,0.5) [out=90, in=330] to (3.75,1);
	\draw [thick, directed=0.65] (3.75,1) to (3.75,1.5);
	\draw [thick, directed=0.65] (3.75,1.5) [out=150, in=270] to (3.5,2);
	\draw [thick, ] (3.75,1.5) [out=30, in=270] to (4,2);
	\draw [thick, directed=0.01] (4,2) [out=90, in=210] to (4.375,2.5);
	\draw [thick, directed=0.65] (4.75,1) [out=30, in=330] to (4.75,1.75);
	\draw [thick, directed=0.65] (4.75,1) [out=150, in=210] to (4.75,1.75);
	\draw [thick, directed=0.65] (4.75,1.75) to (4.75,2.125);
	\draw [thick, directed=0.65] (4.75,2.125) to (4.375,2.5);
	\draw [thick, ] (4.75,2.125) [out=30, in=270] to (5,2.5);
	\draw [thick, directed=1] (5,2.5) to (5,2.75);
	\draw [thick, directed=1] (4.375,2.5) to (4.375,2.75);
	\node at (1.575,0.95) {\scriptsize $3$}; 
	\node at (2.4,0.95) {\scriptsize $1$}; 
	\node at (2.975,1.2) {\scriptsize $7$}; 
	\node at (2.65,2) {\scriptsize $11$}; 
	\node at (1.9,1.65) {\scriptsize $4$}; 
	\node at (3.5,0.575) {\scriptsize $2$}; 
	\node at (3.525,1.25) {\scriptsize $8$}; 
	\node at (3.35,1.65) {\scriptsize $5$}; 
	\node at (4.2,1.925) {\scriptsize $3$}; 
	\node at (4.95,1.925) {\scriptsize $7$}; 
	\node at (5.15,1.45) {\scriptsize $5$}; 
	\node at (4.35,1.45) {\scriptsize $2$}; 
	\node at (4.65,2.5) {\scriptsize $5$}; 
	\draw (1.85,0.65) \wdot;
	\draw (2.765,0.85) \wdot;
	\draw (2.65,1.525) \wdot;
	\draw (4.75,0.45) \wdot;
	\draw (4.95,2.3) \wdot;
	\draw (3.925,1.65) \wdot;
	\draw (3.25,2.45) \wdot;
\et
};
\endxy
\]

We follow the convention that an edge labeled by zero is understood to mean the edge is omitted.  We declare any web containing an edge labeled by a negative integer to be the zero morphism. When no confusion is possible, we sometimes choose to suppress edge labels.

An arbitrary morphism from $\ob{a}$ to $\ob{b}$ is a linear combination of webs of type $\ob{a} \to \ob{b}$.  To write the relations for $\qwebs$ it is convenient to set the following shorthand. A \emph{ladder} is a web which is finite sequence of tensor products and compositions of identities, dots, and webs of the form
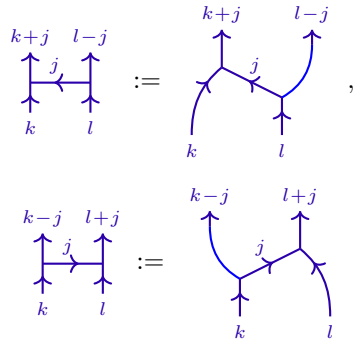
\[
\xy
(0,0)*{\reflectbox{
\bt[color=\clr, scale=.4]
	\draw [ thick, color=\clr,directed=1, directed=0.3 ] (2,-1) to (2,1);
	\draw [ thick, color=\clr,directed=0.6 ] (0,0) to (2,0);
	\draw [ thick, color=\clr,directed=1, directed=0.3 ] (0,-1) to (0,1);
	\node at (0,-1.5) {\reflectbox{\scriptsize $l$}};
	\node at (2,-1.5) {\reflectbox{\scriptsize $k$}};
	\node at (0,1.5) {\reflectbox{\scriptsize $l\! -\! j$}};
	\node at (2,1.5) {\reflectbox{\scriptsize $k\! +\! j$}};
	\node at (0.95,0.55) {\reflectbox{\scriptsize $j$ \ }};
\et
}};
\endxy:= \ 
\xy
(0,0)*{\reflectbox{
\bt[color=\clr, scale=.4]
	\draw [ thick, color=\clr,directed=1] (1,.75) to (1,2);
	\draw [ thick, color=\clr,directed=0.75] (2,-1.5) to [out=90,in=320] (1,.75);
	\draw [ thick, color=\clr,directed=.65] (-1,-1.5) to (-1,-.25);
	\draw [ thick, blue] (-1,-0.25) to [out=150,in=270] (-2,1.5);
	\draw [ thick, color=\clr,directed=1] (-2,1.5) to (-2,2);
	\draw [ thick, color=\clr,directed=.55] (-1,-.25) to (1,.75);
	\node at (-1,-2) {\reflectbox{\scriptsize $l$}};
	\node at (2,-2) {\reflectbox{\scriptsize $k$}};
	\node at (-2,2.5) {\reflectbox{\scriptsize $l\! -\! j$}};
	\node at (1,2.5) {\reflectbox{\scriptsize $k\! +\! j$}};
	\node at (-0.35,.8) {\reflectbox{\scriptsize $j$ \ }};
\et
}};
\endxy,
\]
\[
\xy
(0,0)*{
\bt[color=\clr, scale=.4]
	\draw [ thick, color=\clr,directed=1, directed=0.3 ] (2,-1) to (2,1);
	\draw [ thick, color=\clr,directed=0.6 ] (0,0) to (2,0);
	\draw [ thick, color=\clr,directed=1, directed=0.3 ] (0,-1) to (0,1);
	\node at (0,-1.5) {\scriptsize $k$};
	\node at (2,-1.5) {\scriptsize $l$};
	\node at (0,1.5) {\scriptsize $k\! -\! j$};
	\node at (2,1.5) {\scriptsize $l\! +\! j$};
	\node at (0.95,0.55) {\scriptsize $j$ \ };
\et
};
\endxy:= \ 
\xy
(0,0)*{
\bt[color=\clr, scale=.4]
	\draw [ thick, color=\clr,directed=1] (1,.75) to (1,2);
	\draw [ thick, color=\clr,directed=0.75] (2,-1.5) to [out=90,in=320] (1,.75);
	\draw [ thick, color=\clr,directed=.65] (-1,-1.5) to (-1,-.25);
	\draw [ thick, blue] (-1,-0.25) to [out=150,in=270] (-2,1.5);
	\draw [ thick, color=\clr,directed=1] (-2,1.5) to (-2,2);
	\draw [ thick, color=\clr,directed=.55] (-1,-.25) to (1,.75);
	\node at (-1,-2) {\scriptsize $k$};
	\node at (2,-2) {\scriptsize $l$};
	\node at (-2,2.5) {\scriptsize $k\! -\! j$};
	\node at (1,2.5) {\scriptsize $l\! +\! j$};
	\node at (-0.2,.8) {\scriptsize $j$ \ };
\et
};
\endxy
\]
for $k,l\in\Z_{>0}$ and $j\in\Z_{\geq0}$. The edge which connects two vertical strands is called a \emph{rung}.
Note that by choosing a suitable $k$, $l$, and $j$ every merge and split is itself a ladder.

The morphisms in $\qwebs$ are subject to the following relations  for all nonnegative integers $h,k,l$, along with the relations obtained by reflecting the webs in \cref{dots-past-merges} across a vertical axis, and by reversing all rung orientations\footnote{Where edge labels at the top of the diagram are changed as needed to make a valid web.}  of the ladders in \cref{double-rungs-1} and \cref{double-rungs-2}.  In what follows, to avoid confusing scalars with edge labels we usually choose to write the scalers in parentheses. Also, as usual, $\binom{k+l}{l} = \frac{(k+l)!}{k!l!}$.

\input{UpwardWebRelations}

\subsection{Additional Relations}\label{SS:AdditionalRelations}  From the defining relations of $\qwebs$ we deduce the following useful relations.

\begin{lemma}\label{additional-relations}
The following relation holds in $\qwebs$:
\beq\label{2-dots-zero}
\xy
(0,0)*{
\begin{tikzpicture}[color=\clr, scale=.3]
	\draw [ color=\clr, thick, directed=1] (0,.75) to (0,2);
	\draw [ color=\clr, thick, directed=.85] (0,-2.75) to [out=30,in=330] (0,.75);
	\draw [ color=\clr, thick, directed=.85] (0,-2.75) to [out=150,in=210] (0,.75);
	\draw [ color=\clr, thick, directed=.65] (0,-4) to (0,-2.75);
	\node at (0,-4.5) {\scriptsize $2$};
	\node at (0,2.5) {\scriptsize $2$};
	\node at (-1.5,-1) {\scriptsize $1$};
	\node at (1.5,-1) {\scriptsize $1$};
	\draw (-0.85,-1) \wdot;
	\draw (0.85,-1) \wdot;
\end{tikzpicture}
};
\endxy=0
\eeq
For $k \geq 1$,
\beq\label{dot-on-k-strand}
\xy
(0,0)*{
\begin{tikzpicture}[color=\clr, scale=.3]
	\draw [ color=\clr, thick, directed=1] (0,.75) to (0,2);
	\draw [ color=\clr, thick, directed=.85] (0,-2.75) to [out=30,in=330] (0,.75);
	\draw [ color=\clr, thick, directed=.85] (0,-2.75) to [out=150,in=210] (0,.75);
	\draw [ color=\clr, thick, directed=.65] (0,-4) to (0,-2.75);
	\node at (0,-4.5) {\scriptsize $k$};
	\node at (0,2.5) {\scriptsize $k$};
	\node at (-1.5,-1) {\scriptsize $1$};
	\node at (2,-1) {\scriptsize $k\!-\!1$};
	\draw (-0.875,-1) \wdot;
\end{tikzpicture}
};
\endxy=
\xy
(0,0)*{
\begin{tikzpicture}[color=\clr, scale=.3]
	\draw [ color=\clr, thick, directed=1] (0,-4.) to (0,2);
	\node at (0,-4.5) {\scriptsize $k$};
	\node at (0,2.5) {\scriptsize $k$};
	\draw (0,-1) \wdot;
\end{tikzpicture}
};
\endxy=
\xy
(0,0)*{
\begin{tikzpicture}[color=\clr, scale=.3]
	\draw [ color=\clr, thick, directed=1] (0,.75) to (0,2);
	\draw [ color=\clr, thick, directed=.85] (0,-2.75) to [out=30,in=330] (0,.75);
	\draw [ color=\clr, thick, directed=.85] (0,-2.75) to [out=150,in=210] (0,.75);
	\draw [ color=\clr, thick, directed=.65] (0,-4) to (0,-2.75);
	\node at (0,-4.5) {\scriptsize $k$};
	\node at (0,2.5) {\scriptsize $k$};
	\node at (-2,-1) {\scriptsize $k\!-\!1$};
	\node at (1.5,-1) {\scriptsize $1$};
	\draw (0.875,-1) \wdot;
\end{tikzpicture}
};
\endxy.
\eeq
\end{lemma}

\begin{proof}
To prove \cref{2-dots-zero} is a direct calculation similar to the analogous result for type A \cite[Lemma 2.9]{TVW} and is as follows:
\bea
\xy
(0,1)*{
\begin{tikzpicture}[color=\clr, scale=.3]
	\draw [ color=\clr, thick, directed=1] (0,.75) to (0,7.5);
	\draw [ color=\clr, thick, directed=.85] (0,-2.75) to [out=30,in=330] (0,.75);
	\draw [ color=\clr, thick, directed=.85] (0,-2.75) to [out=150,in=210] (0,.75);
	\draw [ color=\clr, thick, directed=.65] (0,-4) to (0,-2.75);
	\node at (0,-4.5) {\scriptsize $2$};
	\node at (0,8) {\scriptsize $2$};
	\node at (-1.5,-1) {\scriptsize $1$};
	\node at (1.5,-1) {\scriptsize $1$};
	\draw (-0.85,-1) \wdot;
	\draw (0.85,-1) \wdot;
\end{tikzpicture}
};
\endxy
 & \stackrel{\eqref{digon-removal}}{=} & 
\left(\frac{1}{2}\right)
\xy
(0,0)*{
\begin{tikzpicture}[color=\clr, scale=.3]
	\draw [ color=\clr, thick, directed=.65] (0,.75) to (0,2.5);
	\draw [ color=\clr, thick, directed=.85] (0,-2.75) to [out=30,in=330] (0,.75);
	\draw [ color=\clr, thick, directed=.85] (0,-2.75) to [out=150,in=210] (0,.75);
	\draw [ color=\clr, thick, directed=.65] (0,-4) to (0,-2.75);
	\draw [ color=\clr, thick, directed=1] (0,6) to (0,7.5);
	\draw [ color=\clr, thick, directed=.55] (0,2.5) to [out=30,in=330] (0,6);
	\draw [ color=\clr, thick, directed=.55] (0,2.5) to [out=150,in=210] (0,6);
	\node at (0,-4.5) {\scriptsize $2$};
	\node at (0,8) {\scriptsize $2$};
	\node at (-1.5,-1) {\scriptsize $1$};
	\node at (1.5,-1) {\scriptsize $1$};
	\node at (-1.75,4.25) {\scriptsize $1$};
	\node at (1.75,4.25) {\scriptsize $1$};
	\node at (0.75,1.75) {\scriptsize $2$};
	\draw (-0.85,-1) \wdot;
	\draw (0.85,-1) \wdot;
\end{tikzpicture}
};
\endxy
 \ \stackrel{\eqref{dumbbell-relation}}{=} \ 
\left(\frac{1}{2}\right)
\xy
(0,0)*{
\begin{tikzpicture}[color=\clr, scale=.3]
	\draw [ color=\clr, thick, directed=.65] (0,.75) to (0,2.5);
	\draw [ color=\clr, thick, directed=.85] (0,-2.75) to [out=30,in=330] (0,.75);
	\draw [ color=\clr, thick, directed=.85] (0,-2.75) to [out=150,in=210] (0,.75);
	\draw [ color=\clr, thick, directed=.65] (0,-4) to (0,-2.75);
	\draw [ color=\clr, thick, directed=1] (0,6) to (0,7.5);
	\draw [ color=\clr, thick, directed=.85] (0,2.5) to [out=30,in=330] (0,6);
	\draw [ color=\clr, thick, directed=.85] (0,2.5) to [out=150,in=210] (0,6);
	\node at (0,-4.5) {\scriptsize $2$};
	\node at (0,8) {\scriptsize $2$};
	\node at (-1.5,-1) {\scriptsize $1$};
	\node at (1.5,-1) {\scriptsize $1$};
	\node at (-1.5,4.25) {\scriptsize $1$};
	\node at (1.5,4.25) {\scriptsize $1$};
	\node at (0.75,1.75) {\scriptsize $2$};
	\draw (-0.85,-1) \wdot;
	\draw (0.85,-1) \wdot;
	\draw (-0.7,-2) \wdot;
	\draw (0.7,-2) \wdot;
	\draw (-0.85,4.25) \wdot;
	\draw (0.85,4.25) \wdot;
\end{tikzpicture}
};
\endxy
+
\xy
(0,0)*{
\begin{tikzpicture}[color=\clr, scale=.3]
	\draw [ color=\clr, thick, directed=.55] (.9,-1) to (.9,4.25);
	\draw [ color=\clr, thick, directed=.55] (-.9,-1) to (-.9,4.25);
	\draw [ thick] (0,-2.75) to [out=30,in=270] (.9,-1);
	\draw [ thick] (0,-2.75) to [out=150,in=270] (-.9,-1);
	\draw [ color=\clr, thick, directed=.65] (0,-4) to (0,-2.75);
	\draw [ color=\clr, thick, directed=1] (0,6) to (0,7.5);
	\draw [ thick] (.9,4.25) to [out=90,in=330] (0,6);
	\draw [ thick] (-.9,4.25) to [out=90,in=210] (0,6);
	\node at (0,-4.5) {\scriptsize $2$};
	\node at (0,8) {\scriptsize $2$};
	\node at (-1.75,1.6) {\scriptsize $1$};
	\node at (1.75,1.6) {\scriptsize $1$};
	\draw (-0.85,-1) \wdot;
	\draw (0.85,-1) \wdot;
\end{tikzpicture}
};
\endxy
\stackrel{\eqref{super-interchange}}{=}
-\left(\frac{1}{2}\right)
\xy
(0,0)*{
\begin{tikzpicture}[color=\clr, scale=.3]
	\draw [ color=\clr, thick, directed=.65] (0,.75) to (0,2.5);
	\draw [ color=\clr, thick, directed=.85] (0,-2.75) to [out=30,in=330] (0,.75);
	\draw [ color=\clr, thick, directed=.85] (0,-2.75) to [out=150,in=210] (0,.75);
	\draw [ color=\clr, thick, directed=.65] (0,-4) to (0,-2.75);
	\draw [ color=\clr, thick, directed=1] (0,6) to (0,7.5);
	\draw [ color=\clr, thick, directed=.85] (0,2.5) to [out=30,in=330] (0,6);
	\draw [ color=\clr, thick, directed=.85] (0,2.5) to [out=150,in=210] (0,6);
	\node at (0,-4.5) {\scriptsize $2$};
	\node at (0,8) {\scriptsize $2$};
	\node at (-1.5,-1) {\scriptsize $1$};
	\node at (1.5,-1) {\scriptsize $1$};
	\node at (-1.5,4.25) {\scriptsize $1$};
	\node at (1.5,4.25) {\scriptsize $1$};
	\node at (0.75,1.75) {\scriptsize $2$};
	\draw (-0.85,-0.5) \wdot;
	\draw (0.85,-1.25) \wdot;
	\draw (-0.85,-1.25) \wdot;
	\draw (0.7,-2) \wdot;
	\draw (-0.85,4.25) \wdot;
	\draw (0.85,4.25) \wdot;
\end{tikzpicture}
};
\endxy
+
\xy
(0,0)*{
\begin{tikzpicture}[color=\clr, scale=.3]
	\draw [ color=\clr, thick, directed=1] (0,.75) to (0,2);
	\draw [ color=\clr, thick, directed=.85] (0,-2.75) to [out=30,in=330] (0,.75);
	\draw [ color=\clr, thick, directed=.85] (0,-2.75) to [out=150,in=210] (0,.75);
	\draw [ color=\clr, thick, directed=.65] (0,-4) to (0,-2.75);
	\node at (0,-4.5) {\scriptsize $2$};
	\node at (0,2.5) {\scriptsize $2$};
	\node at (-1.5,-1) {\scriptsize $1$};
	\node at (1.5,-1) {\scriptsize $1$};
	\draw (-0.85,-1) \wdot;
	\draw (0.85,-1) \wdot;
\end{tikzpicture}
};
\endxy\\
 & \stackrel{\eqref{dot-collision}}{=} &
-\left(\frac{1}{2}\right)
\xy
(0,0)*{
\begin{tikzpicture}[color=\clr, scale=.3]
	\draw [ color=\clr, thick, directed=.65] (0,.75) to (0,2.5);
	\draw [ color=\clr, thick, directed=.55] (0,-2.75) to [out=30,in=330] (0,.75);
	\draw [ color=\clr, thick, directed=.55] (0,-2.75) to [out=150,in=210] (0,.75);
	\draw [ color=\clr, thick, directed=.65] (0,-4) to (0,-2.75);
	\draw [ color=\clr, thick, directed=1] (0,6) to (0,7.5);
	\draw [ color=\clr, thick, directed=.85] (0,2.5) to [out=30,in=330] (0,6);
	\draw [ color=\clr, thick, directed=.85] (0,2.5) to [out=150,in=210] (0,6);
	\node at (0,-4.5) {\scriptsize $2$};
	\node at (0,8) {\scriptsize $2$};
	\node at (-1.75,-1) {\scriptsize $1$};
	\node at (1.75,-1) {\scriptsize $1$};
	\node at (-1.5,4.25) {\scriptsize $1$};
	\node at (1.5,4.25) {\scriptsize $1$};
	\node at (0.75,1.75) {\scriptsize $2$};
	\draw (-0.85,4.25) \wdot;
	\draw (0.85,4.25) \wdot;
\end{tikzpicture}
};
\endxy
+
\xy
(0,0)*{
\begin{tikzpicture}[color=\clr, scale=.3]
	\draw [ color=\clr, thick, directed=1] (0,.75) to (0,2);
	\draw [ color=\clr, thick, directed=.85] (0,-2.75) to [out=30,in=330] (0,.75);
	\draw [ color=\clr, thick, directed=.85] (0,-2.75) to [out=150,in=210] (0,.75);
	\draw [ color=\clr, thick, directed=.65] (0,-4) to (0,-2.75);
	\node at (0,-4.5) {\scriptsize $2$};
	\node at (0,2.5) {\scriptsize $2$};
	\node at (-1.5,-1) {\scriptsize $1$};
	\node at (1.5,-1) {\scriptsize $1$};
	\draw (-0.85,-1) \wdot;
	\draw (0.85,-1) \wdot;
\end{tikzpicture}
};
\endxy
\stackrel{\eqref{digon-removal}}{=}
-
\xy
(0,0)*{
\begin{tikzpicture}[color=\clr, scale=.3]
	\draw [ color=\clr, thick, directed=1] (0,.75) to (0,2);
	\draw [ color=\clr, thick, directed=.85] (0,-2.75) to [out=30,in=330] (0,.75);
	\draw [ color=\clr, thick, directed=.85] (0,-2.75) to [out=150,in=210] (0,.75);
	\draw [ color=\clr, thick, directed=.65] (0,-4) to (0,-2.75);
	\node at (0,-4.5) {\scriptsize $2$};
	\node at (0,2.5) {\scriptsize $2$};
	\node at (-1.5,-1) {\scriptsize $1$};
	\node at (1.5,-1) {\scriptsize $1$};
	\draw (-0.85,-1) \wdot;
	\draw (0.85,-1) \wdot;
\end{tikzpicture}
};
\endxy+
\xy
(0,0)*{
\begin{tikzpicture}[color=\clr, scale=.3]
	\draw [ color=\clr, thick, directed=1] (0,.75) to (0,2);
	\draw [ color=\clr, thick, directed=.85] (0,-2.75) to [out=30,in=330] (0,.75);
	\draw [ color=\clr, thick, directed=.85] (0,-2.75) to [out=150,in=210] (0,.75);
	\draw [ color=\clr, thick, directed=.65] (0,-4) to (0,-2.75);
	\node at (0,-4.5) {\scriptsize $2$};
	\node at (0,2.5) {\scriptsize $2$};
	\node at (-1.5,-1) {\scriptsize $1$};
	\node at (1.5,-1) {\scriptsize $1$};
	\draw (-0.85,-1) \wdot;
	\draw (0.85,-1) \wdot;
\end{tikzpicture}
};
\endxy \ = \ 0.
\eea
To prove \cref{dot-on-k-strand}, we first prove the case of $k=2$. For this, we start by computing that

\beq\label{dot-on-2-strand}
\xy
(0,0)*{
\begin{tikzpicture}[color=\clr, scale=.3]
	\draw [ color=\clr, thick, directed=1] (0,-4) to (0,2);
	\node at (0,-4.5) {\scriptsize $2$};
	\node at (0,2.5) {\scriptsize $2$};
	\draw (0,1) \wdot;
\end{tikzpicture}
};
\endxy
 \ \stackrel{\eqref{digon-removal}}{=} \ 
\left(\frac{1}{2}\right)
\xy
(0,0)*{
\begin{tikzpicture}[color=\clr, scale=.3]
	\draw [ color=\clr, thick, directed=1] (0,.75) to (0,2.5);
	\draw [ color=\clr, thick, directed=.55] (0,-2.75) to [out=30,in=330] (0,.75);
	\draw [ color=\clr, thick, directed=.55] (0,-2.75) to [out=150,in=210] (0,.75);
	\draw [ color=\clr, thick, directed=.65] (0,-4) to (0,-2.75);
	\node at (0,-4.5) {\scriptsize $2$};
	\node at (0,3) {\scriptsize $2$};
	\node at (-1.75,-1) {\scriptsize $1$};
	\node at (1.75,-1) {\scriptsize $1$};
	\draw (0,1.5) \wdot;
\end{tikzpicture}
};
\endxy
 \ \stackrel{\eqref{dots-past-merges}}{=} \ 
\frac{1}{2}\left(
\xy
(0,0)*{
\begin{tikzpicture}[color=\clr, scale=.3]
	\draw [ color=\clr, thick, directed=1] (0,.75) to (0,2);
	\draw [ color=\clr, thick, directed=.85] (0,-2.75) to [out=30,in=330] (0,.75);
	\draw [ color=\clr, thick, directed=.85] (0,-2.75) to [out=150,in=210] (0,.75);
	\draw [ color=\clr, thick, directed=.65] (0,-4) to (0,-2.75);
	\node at (0,-4.5) {\scriptsize $2$};
	\node at (0,2.5) {\scriptsize $2$};
	\node at (-1.5,-1) {\scriptsize $1$};
	\node at (1.5,-1) {\scriptsize $1$};
	\draw (-0.875,-1) \wdot;
\end{tikzpicture}
};
\endxy+
\xy
(0,0)*{
\begin{tikzpicture}[color=\clr, scale=.3]
	\draw [ color=\clr, thick, directed=1] (0,.75) to (0,2);
	\draw [ color=\clr, thick, directed=.85] (0,-2.75) to [out=30,in=330] (0,.75);
	\draw [ color=\clr, thick, directed=.85] (0,-2.75) to [out=150,in=210] (0,.75);
	\draw [ color=\clr, thick, directed=.65] (0,-4) to (0,-2.75);
	\node at (0,-4.5) {\scriptsize $2$};
	\node at (0,2.5) {\scriptsize $2$};
	\node at (-1.5,-1) {\scriptsize $1$};
	\node at (1.5,-1) {\scriptsize $1$};
	\draw (0.875,-1) \wdot;
\end{tikzpicture}
};
\endxy\right).
\eeq
Next, we compose \cref{dots-past-merges} on bottom with 
$\xy
(0,0)*{
\begin{tikzpicture}[color=\clr, scale=.1] 
	\draw [ thick, directed=1] (1.5,-2) to (1.5,2.5);
	\draw [ thick, directed=1] (-1,-2) to (-1,2.5);
	\draw (-1,0) \wdot;
	\draw (1.5,0) \wdot;
\end{tikzpicture}
};
\endxy$
 followed by a split to get

\[
\xy
(0,0)*{
\begin{tikzpicture}[color=\clr, scale=.3]
	\draw [ color=\clr, thick, directed=1] (0,.75) to (0,2.5);
	\draw [ color=\clr, thick, directed=.35] (0,-2.75) to [out=30,in=330] (0,.75);
	\draw [ color=\clr, thick, directed=.35] (0,-2.75) to [out=150,in=210] (0,.75);
	\draw [ color=\clr, thick, directed=.65] (0,-4) to (0,-2.75);
	\node at (0,-4.5) {\scriptsize $2$};
	\node at (0,3) {\scriptsize $2$};
	\node at (-1.5,-1) {\scriptsize $1$};
	\node at (1.5,-1) {\scriptsize $1$};
	\draw (-0.85,-1) \wdot;
	\draw (0.85,-1) \wdot;
	\draw (0,1.5) \wdot;
\end{tikzpicture}
};
\endxy \ = \ 
\xy
(0,0)*{
\begin{tikzpicture}[color=\clr, scale=.3]
	\draw [ color=\clr, thick, directed=1] (0,.75) to (0,2);
	\draw [ color=\clr, thick, directed=.35] (0,-2.75) to [out=30,in=330] (0,.75);
	\draw [ color=\clr, thick, directed=.35] (0,-2.75) to [out=150,in=210] (0,.75);
	\draw [ color=\clr, thick, directed=.65] (0,-4) to (0,-2.75);
	\node at (0,-4.5) {\scriptsize $2$};
	\node at (0,2.5) {\scriptsize $2$};
	\node at (-1.5,-1) {\scriptsize $1$};
	\node at (1.5,-1) {\scriptsize $1$};
	\draw (-0.85,-1) \wdot;
	\draw (0.85,-1) \wdot;
	\draw (-0.65,0) \wdot;
\end{tikzpicture}
};
\endxy+
\xy
(0,0)*{
\begin{tikzpicture}[color=\clr, scale=.3]
	\draw [ color=\clr, thick, directed=1] (0,.75) to (0,2);
	\draw [ color=\clr, thick, directed=.35] (0,-2.75) to [out=30,in=330] (0,.75);
	\draw [ color=\clr, thick, directed=.35] (0,-2.75) to [out=150,in=210] (0,.75);
	\draw [ color=\clr, thick, directed=.65] (0,-4) to (0,-2.75);
	\node at (0,-4.5) {\scriptsize $2$};
	\node at (0,2.5) {\scriptsize $2$};
	\node at (-1.5,-1) {\scriptsize $1$};
	\node at (1.5,-1) {\scriptsize $1$};
	\draw (-0.85,-1) \wdot;
	\draw (0.85,-1) \wdot;
	\draw (0.65,0) \wdot;
\end{tikzpicture}
};
\endxy
\]
Using \cref{2-dots-zero} on the left, and superinterchange and \cref{dot-collision} on the right, this becomes
\[
0 \ =
\xy
(0,0)*{
\begin{tikzpicture}[color=\clr, scale=.3]
	\draw [ color=\clr, thick, directed=1] (0,.75) to (0,2);
	\draw [ color=\clr, thick, directed=.35] (0,-2.75) to [out=30,in=330] (0,.75);
	\draw [ color=\clr, thick, directed=.35] (0,-2.75) to [out=150,in=210] (0,.75);
	\draw [ color=\clr, thick, directed=.65] (0,-4) to (0,-2.75);
	\node at (0,-4.5) {\scriptsize $2$};
	\node at (0,2.5) {\scriptsize $2$};
	\node at (-1.5,-1) {\scriptsize $1$};
	\node at (1.5,-1) {\scriptsize $1$};
	\draw (0.875,-1) \wdot;
\end{tikzpicture}
};
\endxy-
\xy
(0,0)*{
\begin{tikzpicture}[color=\clr, scale=.3]
	\draw [ color=\clr, thick, directed=1] (0,.75) to (0,2);
	\draw [ color=\clr, thick, directed=.35] (0,-2.75) to [out=30,in=330] (0,.75);
	\draw [ color=\clr, thick, directed=.35] (0,-2.75) to [out=150,in=210] (0,.75);
	\draw [ color=\clr, thick, directed=.65] (0,-4) to (0,-2.75);
	\node at (0,-4.5) {\scriptsize $2$};
	\node at (0,2.5) {\scriptsize $2$};
	\node at (-1.5,-1) {\scriptsize $1$};
	\node at (1.5,-1) {\scriptsize $1$};
	\draw (-0.875,-1) \wdot;
\end{tikzpicture}
};
\endxy.
\]
Combining the above with \cref{dot-on-2-strand} and symmetry, we have \cref{dot-on-k-strand} in case $k=2.$ For general $k$, we use \cref{dots-past-merges} repeatedly to get
\beq\label{dot-past-k-explosion}
\xy
(0,0)*{
\begin{tikzpicture}[color=\clr, scale=.3]
	\draw [ color=\clr, thick, directed=1] (0,.75) to (0,2.5);
	\draw [ color=\clr, thick, directed=.35] (0,-2.75) to [out=30,in=330] (0,.75);
	\draw [ color=\clr, thick, directed=.35] (0,-2.75) to [out=150,in=210] (0,.75);
	\draw [ color=\clr, thick, directed=.65] (0,-4) to (0,-2.75);
	\node at (0,-4.5) {\scriptsize $k$};
	\node at (0,3) {\scriptsize $k$};
	\node at (-1.5,-1) {\scriptsize $1$};
	\node at (1.5,-1) {\scriptsize $1$};
	\node at (0,-1) { \ $\cdots$};
	\draw (0,1.5) \wdot;
\end{tikzpicture}
};
\endxy=
\xy
(0,0)*{
\begin{tikzpicture}[color=\clr, scale=.3]
	\draw [ color=\clr, thick, directed=1] (0,.75) to (0,2);
	\draw [ color=\clr, thick, directed=.35] (0,-2.75) to [out=30,in=330] (0,.75);
	\draw [ color=\clr, thick, directed=.35] (0,-2.75) to [out=150,in=210] (0,.75);
	\draw [ color=\clr, thick, directed=.65] (0,-4) to (0,-2.75);
	\node at (0,-4.5) {\scriptsize $k$};
	\node at (0,2.5) {\scriptsize $k$};
	\node at (-1.5,-1) {\scriptsize $1$};
	\node at (1.5,-1) {\scriptsize $1$};
	\node at (0,-1) { \ $\cdots$};
	\draw (-0.75,-0.25) \wdot;
\end{tikzpicture}
};
\endxy+\cdots+
\xy
(0,0)*{
\begin{tikzpicture}[color=\clr, scale=.3]
	\draw [ color=\clr, thick, directed=1] (0,.75) to (0,2);
	\draw [ color=\clr, thick, directed=.35] (0,-2.75) to [out=30,in=330] (0,.75);
	\draw [ color=\clr, thick, directed=.35] (0,-2.75) to [out=150,in=210] (0,.75);
	\draw [ color=\clr, thick, directed=.65] (0,-4) to (0,-2.75);
	\node at (0,-4.5) {\scriptsize $k$};
	\node at (0,2.5) {\scriptsize $k$};
	\node at (-1.5,-1) {\scriptsize $1$};
	\node at (1.5,-1) {\scriptsize $1$};
	\node at (0,-1) { \ $\cdots$};
	\draw (0.75,-0.25) \wdot;
\end{tikzpicture}
};
\endxy
\eeq
where the ellipses indicate $k$-strands which have been completely ``exploded'' into $k$ separate $1$-strands. By \cref{associativity} the order this is done does not matter. The sum is over the $k$ different webs with a dot on a unique $1$-strand. By \cref{associativity} and the $k=2$ case, the summands are pairwise equal and we have, for example, 
\beq\label{dot-goes-left}
\xy
(0,0)*{
\begin{tikzpicture}[color=\clr, scale=.3]
	\draw [ color=\clr, thick, directed=1] (0,.75) to (0,2.5);
	\draw [ color=\clr, thick, directed=.35] (0,-2.75) to [out=30,in=330] (0,.75);
	\draw [ color=\clr, thick, directed=.35] (0,-2.75) to [out=150,in=210] (0,.75);
	\draw [ color=\clr, thick, directed=.65] (0,-4) to (0,-2.75);
	\node at (0,-4.5) {\scriptsize $k$};
	\node at (0,3) {\scriptsize $k$};
	\node at (-1.5,-1) {\scriptsize $1$};
	\node at (1.5,-1) {\scriptsize $1$};
	\node at (0,-1) { \ $\cdots$};
	\draw (0,1.5) \wdot;
\end{tikzpicture}
};
\endxy=(k)
\xy
(0,0)*{
\begin{tikzpicture}[color=\clr, scale=.3]
	\draw [ color=\clr, thick, directed=1] (0,.75) to (0,2);
	\draw [ color=\clr, thick, directed=.35] (0,-2.75) to [out=30,in=330] (0,.75);
	\draw [ color=\clr, thick, directed=.35] (0,-2.75) to [out=150,in=210] (0,.75);
	\draw [ color=\clr, thick, directed=.65] (0,-4) to (0,-2.75);
	\node at (0,-4.5) {\scriptsize $k$};
	\node at (0,2.5) {\scriptsize $k$};
	\node at (-1.5,-1) {\scriptsize $1$};
	\node at (1.5,-1) {\scriptsize $1$};
	\node at (0,-1) { \ $\cdots$};
	\draw (-0.75,-0.25) \wdot;
\end{tikzpicture}
};
\endxy
\eeq
where, on the right, only the leftmost 1-strand has a dot. We finish the proof by computing that
\[
\xy
(0,0)*{
\begin{tikzpicture}[color=\clr, scale=.3]
	\draw [ color=\clr, thick, directed=1] (0,-4) to (0,2);
	\node at (0,-4.5) {\scriptsize $k$};
	\node at (0,2.5) {\scriptsize $k$};
	\draw (0,1) \wdot;
\end{tikzpicture}
};
\endxy
 \ \stackrel{\eqref{digon-removal}}{=} \ 
\frac{1}{k!}
\xy
(0,0)*{
\begin{tikzpicture}[color=\clr, scale=.3]
	\draw [ color=\clr, thick, directed=1] (0,.75) to (0,2.5);
	\draw [ color=\clr, thick, directed=.35] (0,-2.75) to [out=30,in=330] (0,.75);
	\draw [ color=\clr, thick, directed=.35] (0,-2.75) to [out=150,in=210] (0,.75);
	\draw [ color=\clr, thick, directed=.65] (0,-4) to (0,-2.75);
	\node at (0,-4.5) {\scriptsize $k$};
	\node at (0,3) {\scriptsize $k$};
	\node at (-1.5,-1) {\scriptsize $1$};
	\node at (1.5,-1) {\scriptsize $1$};
	\node at (0,-1) { \ $\cdots$};
	\draw (0,1.5) \wdot;
\end{tikzpicture}
};
\endxy
 \ \stackrel{\eqref{dot-goes-left}}{=} \ 
\frac{1}{(k-1)!}
\xy
(0,0)*{
\begin{tikzpicture}[color=\clr, scale=.3]
	\draw [ color=\clr, thick, directed=1] (0,.75) to (0,2);
	\draw [ color=\clr, thick, directed=.35] (0,-2.75) to [out=30,in=330] (0,.75);
	\draw [ color=\clr, thick, directed=.35] (0,-2.75) to [out=150,in=210] (0,.75);
	\draw [ color=\clr, thick, directed=.65] (0,-4) to (0,-2.75);
	\node at (0,-4.5) {\scriptsize $k$};
	\node at (0,2.5) {\scriptsize $k$};
	\node at (-1.5,-1) {\scriptsize $1$};
	\node at (1.5,-1) {\scriptsize $1$};
	\node at (0,-1) { \ $\cdots$};
	\draw (-0.75,-0.25) \wdot;
\end{tikzpicture}
};
\endxy
 \ \stackrel{\eqref{digon-removal}}{=} \ 
\xy
(0,0)*{
\begin{tikzpicture}[color=\clr, scale=.3]
	\draw [ color=\clr, thick, directed=1] (0,.75) to (0,2);
	\draw [ color=\clr, thick, directed=.35] (0,-2.75) to [out=30,in=330] (0,.75);
	\draw [ color=\clr, thick, directed=.35] (0,-2.75) to [out=150,in=210] (0,.75);
	\draw [ color=\clr, thick, directed=.65] (0,-4) to (0,-2.75);
	\node at (0,-4.5) {\scriptsize $k$};
	\node at (0,2.5) {\scriptsize $k$};
	\node at (-1.5,-1) {\scriptsize $1$};
	\node at (1.75,-1) { \ \scriptsize $k\!-\!1$};
	\draw (-0.875,-1) \wdot;
\end{tikzpicture}
};
\endxy
\]
and noting that the other side of \cref{dot-on-k-strand} follows by symmetry.
\end{proof}

\begin{lemma}\label{Udot-relations}
For $h,k,l,r,s \geq 0$, 
\begin{enumerate}
\item [(a)] 
\begin{equation}\label{rung-collision}
\xy
(0,0)*{
\begin{tikzpicture}[color=\clr]
	\draw [color=\clr, thick, directed=.15, directed=1, directed=.55] (0,0) to (0,1.75);
	\node at (0,-0.15) {\scriptsize $k$};
	\node at (0,1.9) {\scriptsize $k\!+\!r\!+\!s$ \ };
	\draw [color=\clr, thick, directed=.15, directed=1, directed=.55] (1,0) to (1,1.75);
	\node at (1,-0.15) {\scriptsize $l$};
	\node at (1,1.9) {\scriptsize  \ $l\!-\!r\!-\!s$};
	\draw [color=\clr, thick, directed=.55] (1,0.5) to (0,0.5);
	\node at (0.5,0.25) {\scriptsize$r$};
	\node at (-0.4,0.875) {\scriptsize $k\!+\!r$};
	\draw [color=\clr, thick, directed=.55] (1,1.25) to (0,1.25);
	\node at (0.5,1.5) {\scriptsize$s$};
	\node at (1.4,0.875) {\scriptsize $l\!-\!r$};
\end{tikzpicture}
};
\endxy=\binom{r+s}{s}
\xy
(0,0)*{
\begin{tikzpicture}[color=\clr]
	\draw [color=\clr, thick, directed=.15, directed=1] (0,0) to (0,1.75);
	\node at (0,-0.15) {\scriptsize $k$};
	\node at (0,1.9) {\scriptsize $k\!+\!r\!+\!s$ \ };
	\draw [color=\clr, thick, directed=.15, directed=1] (1,0) to (1,1.75);
	\node at (1,-0.15) {\scriptsize $l$};
	\node at (1,1.9) {\scriptsize  \ $l\!-\!r\!-\!s$};
	\draw [color=\clr, thick, directed=.55] (1,0.75) to (0,0.75);
	\node at (0.5,1) {\scriptsize$r\!+\!s$};
\end{tikzpicture}
};
\endxy,
\end{equation}
\item [(b)]
\begin{equation}\label{square-switch-double-dots}
\xy
(0,0)*{
\begin{tikzpicture}[color=\clr]
	\draw [color=\clr, thick, directed=.15, directed=1, directed=.55] (0,0) to (0,1.75);
	\node at (0,-0.15) {\scriptsize $k$};
	\node at (0,1.9) {\scriptsize $k$};
	\draw [color=\clr, thick, directed=.15, directed=1, directed=.55] (1,0) to (1,1.75);
	\node at (1,-0.15) {\scriptsize $l$};
	\node at (1,1.9) {\scriptsize $l$};
	\draw [color=\clr, thick, directed=.55] (0,0.5) to (1,0.5);
	\node at (0.5,0.25) {\scriptsize$1$};
	\node at (-0.4,0.875) {\scriptsize $k\!-\!1$};
	\draw [color=\clr, thick, directed=.55] (1,1.25) to (0,1.25);
	\node at (0.5,1.5) {\scriptsize$1$};
	\node at (1.4,0.875) {\scriptsize $l\!+\!1$};
	\draw  (0.75,0.5) \wdot;
	\draw  (0.25,1.25) \wdot;
\end{tikzpicture}
};
\endxy+
\xy
(0,0)*{
\begin{tikzpicture}[color=\clr]
	\draw [color=\clr, thick, directed=.15, directed=1, directed=.55] (0,0) to (0,1.75);
	\node at (0,-0.15) {\scriptsize $k$};
	\node at (0,1.9) {\scriptsize $k$};
	\draw [color=\clr, thick, directed=.15, directed=1, directed=.55] (1,0) to (1,1.75);
	\node at (1,-0.15) {\scriptsize $l$};
	\node at (1,1.9) {\scriptsize $l$};
	\draw [color=\clr, thick, directed=.55] (1,0.5) to (0,0.5);
	\node at (0.5,0.25) {\scriptsize$1$};
	\node at (-0.4,0.875) {\scriptsize $k\!+\!1$};
	\draw [color=\clr, thick, directed=.55] (0,1.25) to (1,1.25);
	\node at (0.5,1.5) {\scriptsize$1$};
	\node at (1.4,0.875) {\scriptsize $l\!-\!1$};
	\draw  (0.75,0.5) \wdot;
	\draw  (0.25,1.25) \wdot;
\end{tikzpicture}
};
\endxy=(k+l)
\xy
(0,0)*{
\begin{tikzpicture}[color=\clr, scale=.3] 
	\draw [color=\clr, thick, directed=1] (1,-2.75) to (1,2.5);
	\draw [color=\clr, thick, directed=1] (-1,-2.75) to (-1,2.5);
	\node at (-1,3) {\scriptsize $k$};
	\node at (1,3) {\scriptsize $l$};
	\node at (-1,-3.15) {\scriptsize $k$};
	\node at (1,-3.15) {\scriptsize $l$};
\end{tikzpicture}
};
\endxy \ , 
\end{equation}
\item [(c)]
\begin{equation}\label{double-rungs-3}
\xy
(0,0)*{
\begin{tikzpicture}[color=\clr]
	\draw [color=\clr, thick, directed=.15, directed=1] (-1,0) to (-1,1.75);
	\node at (-1,-0.15) {\scriptsize $h$};
	\node at (-1,1.9) {\scriptsize $h\!+\!2$};
	\draw [color=\clr, thick, directed=.15, directed=1, directed=.55] (0,0) to (0,1.75);
	\node at (0,-0.15) {\scriptsize $k$};
	\node at (0,1.9) {\scriptsize $k\!-\!1$};
	\draw [color=\clr, thick, directed=.15, directed=1] (1,0) to (1,1.75);
	\node at (1,-0.15) {\scriptsize $l$};
	\node at (1,1.9) {\scriptsize $l\!-\!1$};
	\draw [color=\clr, thick, directed=.55] (1,0.5) to (0,0.5);
	\node at (0.5,0.25) {\scriptsize$1$};
	\draw [color=\clr, thick, directed=.55] (0,1.25) to (-1,1.25);
	\node at (-0.5,1.5) {\scriptsize$2$};
\end{tikzpicture}
};
\endxy-
\xy
(0,0)*{
\begin{tikzpicture}[color=\clr]
	\draw [color=\clr, thick, directed=.15, directed=1, directed=.55] (-1,0) to (-1,1.75);
	\node at (-1,-0.15) {\scriptsize $h$};
	\node at (-1,1.9) {\scriptsize $h\!+\!2$};
	\draw [color=\clr, thick, directed=.15, directed=1] (0,0) to (0,1.75);
	\node at (0,-0.15) {\scriptsize $k$};
	\node at (0,1.9) {\scriptsize $k\!-\!1$};
	\draw [color=\clr, thick, directed=.15, directed=1] (1,0) to (1,1.75);
	\node at (1,-0.15) {\scriptsize $l$};
	\node at (1,1.9) {\scriptsize $l\!-\!1$};
	\draw [color=\clr, thick, directed=.55] (0,0.5) to (-1,0.5);
	\node at (-0.5,0.25) {\scriptsize$1$};
	\draw [color=\clr, thick, directed=.55] (1,0.875) to (0,0.875);
	\node at (0.5,1.125) {\scriptsize$1$};
	\draw [color=\clr, thick, directed=.55] (0,1.25) to (-1,1.25);
	\node at (-0.5,1.5) {\scriptsize$1$};
\end{tikzpicture}
};
\endxy+
\xy
(0,0)*{
\begin{tikzpicture}[color=\clr]
	\draw [color=\clr, thick, directed=.15, directed=1] (-1,0) to (-1,1.75);
	\node at (-1,-0.15) {\scriptsize $h$};
	\node at (-1,1.9) {\scriptsize $h\!+\!2$};
	\draw [color=\clr, thick, directed=.15, directed=1, directed=.55] (0,0) to (0,1.75);
	\node at (0,-0.15) {\scriptsize $k$};
	\node at (0,1.9) {\scriptsize $k\!-\!1$};
	\draw [color=\clr, thick, directed=.15, directed=1] (1,0) to (1,1.75);
	\node at (1,-0.15) {\scriptsize $l$};
	\node at (1,1.9) {\scriptsize $l\!-\!1$};
	\draw [color=\clr, thick, directed=.55] (1,1.25) to (0,1.25);
	\node at (0.5,1.5) {\scriptsize$1$};
	\draw [color=\clr, thick, directed=.55] (0,0.5) to (-1,0.5);
	\node at (-0.5,0.25) {\scriptsize$2$};
\end{tikzpicture}
};
\endxy=0,
\end{equation}
\item [(d)]
\begin{equation}\label{double-rungs-4}
\xy
(0,0)*{
\begin{tikzpicture}[color=\clr]
	\draw [color=\clr, thick, directed=.15, directed=1] (-1,0) to (-1,1.75);
	\node at (-1,-0.15) {\scriptsize $h$};
	\node at (-1,1.9) {\scriptsize $h\!+\!2$};
	\draw [color=\clr, thick, directed=.15, directed=1] (0,0) to (0,1.75);
	\node at (0,-0.15) {\scriptsize $k$};
	\node at (0,1.9) {\scriptsize $k\!-\!1$};
	\draw [color=\clr, thick, directed=.15, directed=1] (1,0) to (1,1.75);
	\node at (1,-0.15) {\scriptsize $l$};
	\node at (1,1.9) {\scriptsize $l\!-\!1$};
	\draw [color=\clr, thick, directed=.55] (1,0.5) to (0,0.5);
	\node at (0.5,0.25) {\scriptsize$1$};
	\draw [color=\clr, thick, directed=.55] (0,0.875) to (-1,0.875);
	\node at (-0.5,0.65) {\scriptsize$1$};
	\draw [color=\clr, thick, directed=.55] (0,1.25) to (-1,1.25);
	\node at (-0.5,1.5) {\scriptsize$1$};
	\draw  (-0.75,1.25) \wdot;
\end{tikzpicture}
};
\endxy -
\xy
(0,0)*{
\begin{tikzpicture}[color=\clr]
	\draw [color=\clr, thick, directed=.15, directed=1, directed=.55] (-1,0) to (-1,1.75);
	\node at (-1,-0.15) {\scriptsize $h$};
	\node at (-1,1.9) {\scriptsize $h\!+\!2$};
	\draw [color=\clr, thick, directed=.15, directed=1] (0,0) to (0,1.75);
	\node at (0,-0.15) {\scriptsize $k$};
	\node at (0,1.9) {\scriptsize $k\!-\!1$};
	\draw [color=\clr, thick, directed=.15, directed=1] (1,0) to (1,1.75);
	\node at (1,-0.15) {\scriptsize $l$};
	\node at (1,1.9) {\scriptsize $l\!-\!1$};
	\draw [color=\clr, thick, directed=.55] (0,0.5) to (-1,0.5);
	\node at (-0.5,0.25) {\scriptsize$1$};
	\draw [color=\clr, thick, directed=.55] (1,0.875) to (0,0.875);
	\node at (0.5,1.125) {\scriptsize$1$};
	\draw [color=\clr, thick, directed=.55] (0,1.25) to (-1,1.25);
	\node at (-0.5,1.5) {\scriptsize$1$};
	\draw  (-0.75,1.25) \wdot;
\end{tikzpicture}
};
\endxy-
\xy
(0,0)*{
\begin{tikzpicture}[color=\clr]
	\draw [color=\clr, thick, directed=.15, directed=1, directed=.55] (-1,0) to (-1,1.75);
	\node at (-1,-0.15) {\scriptsize $h$};
	\node at (-1,1.9) {\scriptsize $h\!+\!2$};
	\draw [color=\clr, thick, directed=.15, directed=1] (0,0) to (0,1.75);
	\node at (0,-0.15) {\scriptsize $k$};
	\node at (0,1.9) {\scriptsize $k\!-\!1$};
	\draw [color=\clr, thick, directed=.15, directed=1] (1,0) to (1,1.75);
	\node at (1,-0.15) {\scriptsize $l$};
	\node at (1,1.9) {\scriptsize $l\!-\!1$};
	\draw [color=\clr, thick, directed=.55] (0,0.5) to (-1,0.5);
	\node at (-0.5,0.25) {\scriptsize$1$};
	\draw [color=\clr, thick, directed=.55] (1,0.875) to (0,0.875);
	\node at (0.5,1.125) {\scriptsize$1$};
	\draw [color=\clr, thick, directed=.55] (0,1.25) to (-1,1.25);
	\node at (-0.5,1.5) {\scriptsize$1$};
	\draw  (-0.25,0.5) \wdot;
\end{tikzpicture}
};
\endxy+
\xy
(0,0)*{
\begin{tikzpicture}[color=\clr]
	\draw [color=\clr, thick, directed=.15, directed=1] (-1,0) to (-1,1.75);
	\node at (-1,-0.15) {\scriptsize $h$};
	\node at (-1,1.9) {\scriptsize $h\!+\!2$};
	\draw [color=\clr, thick, directed=.15, directed=1] (0,0) to (0,1.75);
	\node at (0,-0.15) {\scriptsize $k$};
	\node at (0,1.9) {\scriptsize $k\!-\!1$};
	\draw [color=\clr, thick, directed=.15, directed=1] (1,0) to (1,1.75);
	\node at (1,-0.15) {\scriptsize $l$};
	\node at (1,1.9) {\scriptsize $l\!-\!1$};
	\draw [color=\clr, thick, directed=.55] (0,0.5) to (-1,0.5);
	\node at (-0.5,0.25) {\scriptsize$1$};
	\draw [color=\clr, thick, directed=.55] (0,0.875) to (-1,0.875);
	\node at (-0.5,1.125) {\scriptsize$1$};
	\draw [color=\clr, thick, directed=.55] (1,1.25) to (0,1.25);
	\node at (0.5,1.5) {\scriptsize$1$};
	\draw  (-0.25,0.5) \wdot;
\end{tikzpicture}
};
\endxy=0.
\end{equation}
\end{enumerate}
In addition, we have the equations obtained by reversing all rung orientations\footnote{Relabelling the tops of diagrams as needed to make a valid web.} of the ladders in \cref{rung-collision}.  We also have the equations obtained from  \cref{double-rungs-3,double-rungs-4} by reversing all rung orientations, reflecting across the vertical line through the middle arrow, and both reversing and reflecting.  Hence \cref{double-rungs-3} and \cref{double-rungs-4} each represent four relations. 
\end{lemma}

\begin{proof}
We sketch the proof and leave the details to reader. The proof of \cref{rung-collision} follows from \cref{associativity} and \cref{digon-removal}. The proof of \cref{square-switch-double-dots} is similar to \cite[Lemma 2.10(b)]{TVW}. Its proof involves applications of \cref{square-switch} on the edges labeled $l+1$ and $k+1$ in the webs on the left, followed by two applications of \cref{digon-removal} and the \cref{dumbbell-relation}.

The proof of \cref{double-rungs-3} and \cref{double-rungs-4} are similar to \cite[Lemma 2.10(c)]{TVW}. The former involves using \cref{dumbbell-relation}  on the parallel ladder rungs in the middle web, followed by \cref{associativity},  \cref{additional-relations}\cref{2-dots-zero}, \cref{square-switch}, and two applications of \cref{digon-removal}. The proof of the latter is similar, but it also requires \cref{double-rungs-2}.
\end{proof}

\subsection{Crossings and the Sergeev Algebra}\label{SS:crossings}

We introduce the following additional diagram as a shorthand.  Define the upward crossing morphism in $\End_{\qwebs}(\up_{1}^2)$ by
\[
\xy
(0,0)*{
\bt[color=\clr, scale=1.25]
	\draw[thick, directed=1] (0,0) to (0.5,0.5);
	\draw[thick, directed=1] (0.5,0) to (0,0.5);
	\node at (0,-0.15) {\scriptsize $1$};
	\node at (0,0.65) {\scriptsize $1$};
	\node at (0.5,-0.15) {\scriptsize $1$};
	\node at (0.5,0.65) {\scriptsize $1$};
\et
};
\endxy
 \ := \ 
\xy
(0,0)*{
\begin{tikzpicture}[color=\clr, scale=.3]
	\draw [ thick, directed=.75] (0,0.25) to (0,1.25);
	\draw [ thick, directed=1] (0,1.25) to [out=30,in=270] (1,2.5);
	\draw [ thick, directed=1] (0,1.25) to [out=150,in=270] (-1,2.5); 
	\draw [ thick, directed=.65] (1,-1) to [out=90,in=330] (0,0.25);
	\draw [ thick, directed=.65] (-1,-1) to [out=90,in=210] (0,0.25);
	\node at (-1,3) {\scriptsize $1$};
	\node at (1,3) {\scriptsize $1$};
	\node at (-1,-1.5) {\scriptsize $1$};
	\node at (1,-1.5) {\scriptsize $1$};
	\node at (-0.75,0.75) {\scriptsize $2$};
\end{tikzpicture}
};
\endxy
 \ - \ 
 \xy
(0,0)*{
\begin{tikzpicture}[color=\clr, scale=.3]
	\draw [thick, directed=1] (-2,-1) to (-2,2.5);
	\draw [thick, directed=1] (-4,-1) to (-4,2.5);
	\node at (-2,-1.5) {\scriptsize $1$};
	\node at (-2,3) {\scriptsize $1$};
	\node at (-4,-1.5) {\scriptsize $1$};
	\node at (-4,3) {\scriptsize $1$};
\end{tikzpicture}
};
\endxy \ .
\]  More generally, we introduce the following notation.

\begin{definition}\label{Sergeev-webs}
Fix $k\in\Z_{>0}$. For $i=1, \dotsc , k$ and $j = 1, \dotsc , k-1$, define the morphisms $c_{i}, s_{j} \in\End_{\qwebs}(\up_{1}^k)$ by
\[
c_i=
\xy
(0,0)*{
\bt[color=\clr, scale=1.25]
	\draw[thick, directed=1] (-0.75,0) to (-0.75,0.5);
	\draw[thick, directed=1] (-0.25,0) to (-0.25,0.5);
	\draw[thick, directed=1] (0.25,0) to (0.25,0.5);
	\draw[thick, directed=1] (0.75,0) to (0.75,0.5);
	\draw[thick, directed=1] (0,0) to (0,0.5);
	\node at (0,-0.15) {\scriptsize $1$};
	\node at (0,0.65) {\scriptsize $1$};
	\node at (-0.25,-0.15) {\scriptsize $1$};
	\node at (-0.25,0.65) {\scriptsize $1$};
	\node at (-0.75,-0.15) {\scriptsize $1$};
	\node at (-0.75,0.65) {\scriptsize $1$};
	\node at (0.25,-0.15) {\scriptsize $1$};
	\node at (0.25,0.65) {\scriptsize $1$};
	\node at (.75,-0.15) {\scriptsize $1$};
	\node at (.75,0.65) {\scriptsize $1$};
	\node at (-0.5,0.25) { \ $\cdots$};
	\node at (0.5,0.25) { \ $\cdots$};
	\draw (0,0.25) \wdot;
\et
};
\endxy \ ,
\quad s_j=
\xy
(0,0)*{
\bt[color=\clr, scale=1.25]
	\draw[thick, directed=1] (-0.75,0) to (-0.75,0.5);
	\draw[thick, directed=1] (-0.25,0) to (-0.25,0.5);
	\draw[thick, directed=1] (0.75,0) to (0.75,0.5);
	\draw[thick, directed=1] (1.25,0) to (1.25,0.5);
	\draw[thick, directed=1] (0,0) to (0.5,0.5);
	\draw[thick, directed=1] (0.5,0) to (0,0.5);
	\node at (0,-0.15) {\scriptsize $1$};
	\node at (0,0.65) {\scriptsize $1$};
	\node at (0.5,-0.15) {\scriptsize $1$};
	\node at (0.5,0.65) {\scriptsize $1$};
	\node at (-0.25,-0.15) {\scriptsize $1$};
	\node at (-0.25,0.65) {\scriptsize $1$};
	\node at (-0.75,-0.15) {\scriptsize $1$};
	\node at (-0.75,0.65) {\scriptsize $1$};
	\node at (0.75,-0.15) {\scriptsize $1$};
	\node at (0.75,0.65) {\scriptsize $1$};
	\node at (1.25,-0.15) {\scriptsize $1$};
	\node at (1.25,0.65) {\scriptsize $1$};
	\node at (-0.5,0.25) { \ $\cdots$};
	\node at (1,0.25) { \ $\cdots$};
\et
};
\endxy
\]  where the dot on $c_{i}$ is on the $i$th strand and $s_{j}$ is the crossing of the $j$th and $(j+1)$st strands.
\end{definition}

Drawing these morphisms as crossings is justified by the following lemma.

\begin{lemma}\label{L:srelations} For any $k \geq 2$ the following relations hold in $\qwebs$ for all admissible $i,j$:
\begin{enumerate}
\item $s_{j}^{2}=1$;
\item $s_{i}s_{j}=s_{j}s_{i}$ if $|i-j| >1$;
\item $s_{i}s_{i+1}s_{i}=s_{i + 1}s_{i}s_{i+1}$;
\item $s_{i}c_{j}=c_{j}s_{i}$ if $|i-j| >1$;
\item $s_{i}c_{i}=c_{i+1}s_{i}$; 
\end{enumerate}

\end{lemma}

\begin{proof}
Statements (2) and (4) follow immediately from the super-interchange law, while the rest may be verified by direct calculations involving \cref{digon-removal}, \cref{square-switch}, and \cref{square-switch-dots}.
\end{proof}
 Comparing the previous lemma to the defining relations of the Sergeev algebra given in \cref{SS:Sergeevdualities} we see there is a homomorphism of superalgebras 
\begin{equation}\label{E:Sergeevhomomorphism}
\xi_{k} :\Ser_{k} \to \End_{\qwebs}(\up_{1}^{k})
\end{equation}
for any $k \geq 1$ which sends the generators of $\Ser_{k}$ to the morphisms in $\End_{\qwebs}(\up_{1}^{k})$ of the same name.  In \cref{description-of-all-ones} we will see that this map is an isomorphism. Meanwhile, for $x \in \Ser_{k}$ we abuse notation by writing $x \in \End_{\qwebs}(\up_{1}^{k})$ for $\xi_{k}(x)$.

In what follows we allow ourselves to draw merges and splits of multiple strands.  For example, define
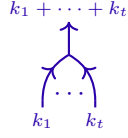
\begin{equation*}
\xy
(0,0)*{
\bt[color=\clr, scale=.35]
	\draw [ thick, directed=1, color=\clr ] (4,4) to (4,5.25);
	\draw [ thick, directed=.65, color=\clr ] (5,2.25) to [out=90,in=330] (4,4);
	\draw [ thick, directed=.65, color=\clr ] (3,2.25) to [out=90,in=210] (4,4);
	\draw [thick, color=\clr] (3,2) to (3,2.25);
	\draw [thick, color=\clr] (5,2) to (5,2.25);
	\node at (4.1, 2.5) { $\cdots$};
	\node at (4,5.75) {\scriptsize $k_{1}+\dots + k_{t}$};
	\node at (3,1.5) {\scriptsize $k_{1}$};
	\node at (5,1.5) {\scriptsize $k_{t}$};
\et
};
\endxy 
\end{equation*} to be the vertical concatenation of $t-1$ merge diagrams.  By \cref{associativity} the resulting morphism is independent of how this is done.  We define the split into multiple strands similarly.

\begin{lemma}\label{L:untwist-permutation}
Let $k\in\Z_{>0}$ and let $\sigma \in \Sigma_{k}$. Then, 
\begin{equation}\label{E:untwist-permutation}
\xy
(0,0)*{
\bt[color=\clr, scale=.35]
	\draw [color=\clr,  thick, directed=1] (4,9) to (4,10);
	\draw [color=\clr,  thick, directed=0.75] (5,7.25) to [out=90,in=330] (4,9);
	\draw [color=\clr,  thick, directed=0.75] (3,7.25) to [out=90,in=210] (4,9);
	\draw [color=\clr,  thick ] (2.5,7.25) rectangle (5.5,4.25);
	\draw [color=\clr,  thick ] (3,4.25) to (3,3.25);
	\draw [color=\clr,  thick ] (5,4.25) to (5,3.25);
	\node at (4.1, 7.5) { $\cdots$};
	\node at (4.1, 3.75) { $\cdots$};
	\node at (4, 5.75) { $\sigma$};
	\node at (4,10.5) {\scriptsize $k$};
	\node at (3,2.75) {\scriptsize $1$};
	\node at (5,2.75) {\scriptsize $1$};
\et
};
\endxy  \ = \ 
\xy
(0,0)*{
\bt[color=\clr, scale=.35]
	\draw [ thick, directed=1] (4,9) to (4,10);
	\draw [ thick, directed=0.75] (5,7.25) to [out=90,in=330] (4,9);
	\draw [ thick, directed=0.75] (3,7.25) to [out=90,in=210] (4,9);
%	\draw [ thick ] (3,7.25) to (3,3.25);
%	\draw [ thick ] (5,7.25) to (5,3.25);
	\node at (4.1, 7.75) { $\cdots$};
	\node at (4,10.5) {\scriptsize $k$};
	\node at (3,6.75) {\scriptsize $1$};
	\node at (5,6.75) {\scriptsize $1$};
\et
};
\endxy \ .
\end{equation}
\end{lemma}

\begin{proof}
First, note that
\beq\label{untwist-crossing}
\xy
(0,0)*{
\bt[color=\clr, scale=.35]
	\draw [ thick, directed=1] (0,2.75) to (0,3.5);
	\draw [ thick, directed=0.75] (1,1.75) to [out=90,in=330] (0,2.75);
	\draw [ thick, directed=0.75] (-1,1.75) to [out=90,in=210] (0,2.75);
	\draw [thick] (-1,-1) to (1,1.75);
	\draw [thick] (1,-1) to (-1,1.75);
	\node at (0,3.95) {\scriptsize $2$};
	\node at (-1,-1.45) {\scriptsize $1$};
	\node at (1,-1.45) {\scriptsize $1$};
\et
};
\endxy \ = \ 
\xy
(0,0)*{
\begin{tikzpicture}[color=\clr, scale=.35]
	\draw [ thick, ] (0,0) to (0,.75);
	\draw [ thick, ] (0,.75) to [out=30,in=270] (1,1.75);
	\draw [ thick, ] (0,.75) to [out=150,in=270] (-1,1.75); 
	\draw [ thick, directed=0.75] (1,-1) to [out=90,in=330] (0,0);
	\draw [ thick, directed=0.75] (-1,-1) to [out=90,in=210] (0,0);
	\draw [ thick, directed=1] (0,2.75) to (0,3.5);
	\draw [ thick, directed=0.75] (1,1.75) to [out=90,in=330] (0,2.75);
	\draw [ thick, directed=0.75] (-1,1.75) to [out=90,in=210] (0,2.75);
	\node at (-1.5,1.75) {\scriptsize $1$};
	\node at (1.5,1.75) {\scriptsize $1$};
	\node at (-1,-1.45) {\scriptsize $1$};
	\node at (1,-1.45) {\scriptsize $1$};
%	\node at (-0.5,0.2) {\scriptsize $2$};
	\node at (0,3.95) {\scriptsize $2$};
\end{tikzpicture}
};
\endxy \ - \ 
\xy
(0,0)*{
\bt[color=\clr, scale=.35]
	\draw [ thick, directed=1] (0,2.75) to (0,3.5);
	\draw [ thick, directed=0.75] (1,1.75) to [out=90,in=330] (0,2.75);
	\draw [ thick, directed=0.75] (-1,1.75) to [out=90,in=210] (0,2.75);
	\draw [thick] (-1,-1) to (-1,1.75);
	\draw [thick] (1,-1) to (1,1.75);
	\node at (0,3.95) {\scriptsize $2$};
	\node at (-1,-1.45) {\scriptsize $1$};
	\node at (1,-1.45) {\scriptsize $1$};
\et
};
\endxy \ = \ (2)
\xy
(0,0)*{
\bt[color=\clr, scale=.35]
	\draw [ thick, directed=1] (4,9) to (4,10);
	\draw [ thick, directed=0.75] (5,7.25) to [out=90,in=330] (4,9);
	\draw [ thick, directed=0.75] (3,7.25) to [out=90,in=210] (4,9);
%	\draw [ thick ] (3,5) to (3,7.25);
%	\draw [ thick ] (5,5) to (5,7.25);
	\node at (4,10.5) {\scriptsize $2$};
	\node at (3,6.75) {\scriptsize $1$};
	\node at (5,6.75) {\scriptsize $1$};
\et
};
\endxy \ - \ 
\xy
(0,0)*{
\bt[color=\clr, scale=.35]
	\draw [ thick, directed=1] (4,9) to (4,10);
	\draw [ thick, directed=0.75] (5,7.25) to [out=90,in=330] (4,9);
	\draw [ thick, directed=0.75] (3,7.25) to [out=90,in=210] (4,9);
%	\draw [ thick ] (3,5) to (3,7.25);
%	\draw [ thick ] (5,5) to (5,7.25);
	\node at (4,10.5) {\scriptsize $2$};
	\node at (3,6.75) {\scriptsize $1$};
	\node at (5,6.75) {\scriptsize $1$};
\et
};
\endxy \ = \ 
\xy
(0,0)*{
\bt[color=\clr, scale=.35]
	\draw [ thick, directed=1] (4,9) to (4,10);
	\draw [ thick, directed=0.75] (5,7.25) to [out=90,in=330] (4,9);
	\draw [ thick, directed=0.75] (3,7.25) to [out=90,in=210] (4,9);
%	\draw [ thick ] (3,5) to (3,7.25);
%	\draw [ thick ] (5,5) to (5,7.25);
	\node at (4,10.5) {\scriptsize $2$};
	\node at (3,6.75) {\scriptsize $1$};
	\node at (5,6.75) {\scriptsize $1$};
\et
};
\endxy \ .
\eeq
The first equality is by definition and the second follows from \cref{digon-removal}. Since every permutation is a product of simple transpositions, a straightforward calculation using \cref{untwist-crossing} and \cref{associativity} proves the statement in general.
\end{proof}

\subsection{The clasp idempotents}\label{SS:ClaspIdempotents}

\begin{definition}\label{clasps}  For $k\in\Z_{> 1}$, the $k$-th \emph{clasp} $Cl_{k} \in\End_{\qwebs}(\up_{1}^{k})$ is given by
\[
Cl_k
=\frac{1}{k!}
\xy
(0,0)*{
\begin{tikzpicture}[color=\clr, scale=.3]
	\draw [ thick, directed=.55] (0,-1) to (0,.75);
	\draw [ thick, directed=1] (0,.75) to [out=30,in=270] (1,2.5);
	\draw [ thick, directed=1] (0,.75) to [out=150,in=270] (-1,2.5); 
	\draw [ thick, directed=.65] (1,-2.75) to [out=90,in=330] (0,-1);
	\draw [ thick, directed=.65] (-1,-2.75) to [out=90,in=210] (0,-1);
	\node at (-1,3) {\scriptsize $1$};
	\node at (0.1,1.75) {$\cdots$};
	\node at (1,3) {\scriptsize $1$};
	\node at (-1,-3.3) {\scriptsize $1$};
	\node at (0.1,-2.35) {$\cdots$};
	\node at (1,-3.3) {\scriptsize $1$};
	\node at (0.75,-0.25) {\scriptsize $k$};
\end{tikzpicture}
};
\endxy .
\] In addition, $Cl_{1} \in \End_{\qwebs}(\up_{1})$ is understood to be the identity.  A calculation using \cref{digon-removal} shows $Cl_k$ is an idempotent for all $k \in \Z_{\geq 1}$.
\end{definition}
The following lemma shows clasps admit recursion formulas similar to those of the \emph{Jones-Wenzl projectors} in the Temperley Lieb algebra (e.g., see \cite{W}).

\begin{lemma}\label{clasp-recursion}
For $k\in\Z_{>1}$,
\begin{equation}\label{E:clasp1}
Cl_k
=
\xy
(0,0)*{
\begin{tikzpicture}[color=\clr, scale=.3]
	\draw [ thick, ] (-2,-7) to (-2,-6);
	\draw [ thick, ] (0,-7) to (0,-6);
	\draw [ thick, ] (2,-7) to (2,-6);
	\draw [ thick, directed=1] (-2,-2) to (-2,-1);
	\draw [ thick, directed=1] (0,-2) to (0,-1);
	\draw [ thick, directed=1] (2,-2) to (2,-1);
	\draw [ thick, directed=1] (4,-7) to (4,-1);
	\draw [ thick] (-2.3,-6) rectangle (2.3,-2);
	\node at (-2,-7.45) {\scriptsize $1$};
	\node at (0,-7.45) {\scriptsize $1$};
	\node at (2,-7.45) {\scriptsize $1$};
	\node at (4,-7.45) {\scriptsize $1$};
	\node at (-2,-.45) {\scriptsize $1$};
	\node at (0,-.45) {\scriptsize $1$};
	\node at (2,-.45) {\scriptsize $1$};
	\node at (4,-.45) {\scriptsize $1$};
	\node at (-1,-1.75) { \ $\cdots$};
	\node at (-1,-6.5) { \ $\cdots$};
	\node at (0,-4) { $Cl_{k-1}$};
\end{tikzpicture}
};
\endxy
+\frac{k-1}{k}
\xy
(0,0)*{
\begin{tikzpicture}[color=\clr, scale=.3]
	\draw [ thick] (-2,-9) to (-2,-8);
	\draw [ thick] (0,-9) to (0,-8);
	\draw [ thick] (2,-9) to (2,-8);
	\draw [ thick] (4,-9) to (4,-5.85);
	\draw [ thick, directed=1] (-2,-0) to (-2,1);
	\draw [ thick, directed=1] (0,0) to (0,1);
	\draw [ thick, directed=1] (2,0) to (2,1);
	\draw [ thick, directed=1] (4,-2.15) to (4,1);
	\draw [ thick, directed=0.75] (-2,-5.85) to (-2,-2.15);
	\draw [ thick, directed=0.75] (0,-5.85) to (0,-2.15);
	\draw [ thick, directed=0.75] (3,-4.5) to (3,-3.5);
	\draw [ thick] (-2.3,-8) rectangle (2.3,-5.85);
	\draw [ thick] (-2.3,-2.15) rectangle (2.3,0);
	\draw [ thick] (3,-3.5) to [out=30,in=270] (4,-2.15);
	\draw [ thick] (3,-3.5) to [out=150,in=270] (2,-2.15); 
	\draw [ thick] (4,-5.85) to [out=90,in=330] (3,-4.5);
	\draw [ thick] (2,-5.85) to [out=90,in=210] (3,-4.5);
	\draw
	 (2.25,-2.9) \wdot
	 (3.75,-2.9) \wdot
	 (2.25,-5.1) \wdot
	 (3.75,-5.1) \wdot;
	\node at (-2,-9.45) {\scriptsize $1$};
	\node at (0,-9.45) {\scriptsize $1$};
	\node at (2,-9.45) {\scriptsize $1$};
	\node at (4,-9.45) {\scriptsize $1$};
	\node at (-2,1.55) {\scriptsize $1$};
	\node at (0,1.55) {\scriptsize $1$};
	\node at (2,1.55) {\scriptsize $1$};
	\node at (4,1.55) {\scriptsize $1$};
	\node at (3.65,-4) {\scriptsize $2$};
	\node at (-2.5,-4) {\scriptsize $1$};
	\node at (0.5,-4) {\scriptsize $1$};
	\node at (-1,0.25) { \ $\cdots$};
	\node at (-1,-4) { \ $\cdots$};
	\node at (-1,-8.75) { \ $\cdots$};
	\node at (0,-1.125) {\small $Cl_{k-1}$};
	\node at (0,-6.975) {\small $Cl_{k-1}$};
\end{tikzpicture}
};
\endxy, \tag{a}
\end{equation}
and
\begin{equation}\label{E:clasp2}
Cl_k
= \frac{k-1}{k}
\xy
(0,0)*{
\begin{tikzpicture}[color=\clr, scale=.3]
	\draw [ thick] (-2,-9) to (-2,-8);
	\draw [ thick] (0,-9) to (0,-8);
	\draw [ thick] (2,-9) to (2,-8);
	\draw [ thick] (4,-9) to (4,-5.85);
	\draw [ thick, directed=1] (-2,-0) to (-2,1);
	\draw [ thick, directed=1] (0,0) to (0,1);
	\draw [ thick, directed=1] (2,0) to (2,1);
	\draw [ thick, directed=1] (4,-2.15) to (4,1);
	\draw [ thick, directed=0.75] (-2,-5.85) to (-2,-2.15);
	\draw [ thick, directed=0.75] (0,-5.85) to (0,-2.15);
	\draw [ thick, directed=0.75] (3,-4.5) to (3,-3.5);
	\draw [ thick] (-2.3,-8) rectangle (2.3,-5.85);
	\draw [ thick] (-2.3,-2.15) rectangle (2.3,0);
	\draw [ thick] (3,-3.5) to [out=30,in=270] (4,-2.15);
	\draw [ thick] (3,-3.5) to [out=150,in=270] (2,-2.15); 
	\draw [ thick, directed=0.75] (4,-5.85) to [out=90,in=330] (3,-4.5);
	\draw [ thick, directed=0.75] (2,-5.85) to [out=90,in=210] (3,-4.5);
	\node at (-2,-9.45) {\scriptsize $1$};
	\node at (0,-9.45) {\scriptsize $1$};
	\node at (2,-9.45) {\scriptsize $1$};
	\node at (4,-9.45) {\scriptsize $1$};
	\node at (-2,1.55) {\scriptsize $1$};
	\node at (0,1.55) {\scriptsize $1$};
	\node at (2,1.55) {\scriptsize $1$};
	\node at (4,1.55) {\scriptsize $1$};
	\node at (3.65,-4) {\scriptsize $2$};
	\node at (-2.5,-4) {\scriptsize $1$};
	\node at (0.5,-4) {\scriptsize $1$};
	\node at (-1,0.25) { \ $\cdots$};
	\node at (-1,-4) { \ $\cdots$};
	\node at (-1,-8.75) { \ $\cdots$};
	\node at (0,-1.125) {\small $Cl_{k-1}$};
	\node at (0,-6.975) {\small $Cl_{k-1}$};
\end{tikzpicture}
};
\endxy
-\frac{k-2}{k}
\xy
(0,0)*{
\begin{tikzpicture}[color=\clr, scale=.3]
	\draw [ thick, ] (-2,-7) to (-2,-6);
	\draw [ thick, ] (0,-7) to (0,-6);
	\draw [ thick, ] (2,-7) to (2,-6);
	\draw [ thick, directed=1] (-2,-2) to (-2,-1);
	\draw [ thick, directed=1] (0,-2) to (0,-1);
	\draw [ thick, directed=1] (2,-2) to (2,-1);
	\draw [ thick, directed=1] (4,-7) to (4,-1);
	\draw [ thick] (-2.3,-6) rectangle (2.3,-2);
	\node at (-2,-7.45) {\scriptsize $1$};
	\node at (0,-7.45) {\scriptsize $1$};
	\node at (2,-7.45) {\scriptsize $1$};
	\node at (4,-7.45) {\scriptsize $1$};
	\node at (-2,-.45) {\scriptsize $1$};
	\node at (0,-.45) {\scriptsize $1$};
	\node at (2,-.45) {\scriptsize $1$};
	\node at (4,-.45) {\scriptsize $1$};
	\node at (-1,-1.75) { \ $\cdots$};
	\node at (-1,-6.5) { \ $\cdots$};
	\node at (0,-4) { $Cl_{k-1}$};
\end{tikzpicture}
};
\endxy . \tag{b}
\end{equation}
\end{lemma}

\begin{proof}
The proof of the first recursion identical to \cite[Lemma 2.13]{RT}. See also \cite[Lemma 2.12]{TVW}.  The second recursion follows from the first by applying \cref{dumbbell-relation}  to the rightmost webs. 
\end{proof}
An inductive argument using the recursion formulas yields the following closed formula for the clasps.

\begin{lemma}\label{L:claspsum} For $k\geq 1$ we have
\begin{equation}\label{clasp-sum}
Cl_k \ = \ \frac{1}{k!}\sum_{\sigma\in\Sigma_k}
\xy
(0,0)*{
\bt[color=\clr, scale=.35]
	\draw [ thick ] (2.5,7.25) rectangle (5.5,4.25);
	\draw [ thick ] (3,4.25) to (3,3.25);
	\draw [ thick ] (5,4.25) to (5,3.25);
	\draw [ thick, directed=1] (3,7.25) to (3,8.25);
	\draw [ thick, directed=1 ] (5,7.25) to (5,8.25);
	\node at (4.1, 7.5) { $\cdots$};
	\node at (4.1, 3.75) { $\cdots$};
	\node at (4, 5.75) { $\sigma$};
	\node at (3,2.75) {\scriptsize $1$};
	\node at (5,2.75) {\scriptsize $1$};
	\node at (3,8.75) {\scriptsize $1$};
	\node at (5,8.75) {\scriptsize $1$};
\et
};
\endxy \ .
\end{equation}
\end{lemma}

\subsection{Symmetry on  \texorpdfstring{$\qwebsupdown$}{Oriented Webs}}

We now introduce braidings between arbitrary objects of $\qwebs$ and show that this provides a symmetric braiding.  We first define the braiding between two generating objects.

\begin{definition}\label{D:Upbraiding} For $k,l\in\Z_{>0}$ define
\[\beta_{\up_{k}, \up_{l}}=
\xy %%%%%%%%%%%%%%%%%%%%%%%%%%%%%% Arbitrary upward crossing
(0,0)*{
\bt[color=\clr, scale=1.25]
	\draw[thick, directed=1] (0,0) to (0.5,0.5);
	\draw[thick, directed=1] (0.5,0) to (0,0.5);
	\node at (0,-0.15) {\scriptsize $k$};
	\node at (0,0.65) {\scriptsize $l$};
	\node at (0.5,-0.15) {\scriptsize $l$};
	\node at (0.5,0.65) {\scriptsize $k$};
\et
};
\endxy =\frac{1}{k!\,l!}
\xy
(0,0)*{
\bt[color=\clr, scale=.35]
	\draw [ thick, directed=1] (0, .75) to (0,1.5);
	\draw [ thick, directed=0.75] (1,-1) to [out=90,in=330] (0,.75);
	\draw [ thick, directed=0.75] (-1,-1) to [out=90,in=210] (0,.75);
	\draw [ thick, directed=1] (4, .75) to (4,1.5);
	\draw [ thick, directed=0.75] (5,-1) to [out=90,in=330] (4,.75);
	\draw [ thick, directed=0.75] (3,-1) to [out=90,in=210] (4,.75);
	\draw [ thick, directed=0.75] (0,-6.5) to (0,-5.75);
	\draw [ thick, ] (0,-5.75) to [out=30,in=270] (1,-4);
	\draw [ thick, ] (0,-5.75) to [out=150,in=270] (-1,-4); 
	\draw [ thick, directed=0.75] (4,-6.5) to (4,-5.75);
	\draw [ thick, ] (4,-5.75) to [out=30,in=270] (5,-4);
	\draw [ thick, ] (4,-5.75) to [out=150,in=270] (3,-4); 
	\draw [ thick ] (5,-1) to (1,-4);
	\draw [ thick ] (3,-1) to (-1,-4);
	\draw [ thick ] (1,-1) to (5,-4);
	\draw [ thick ] (-1,-1) to (3,-4);
	\node at (2.6, -0.5) {\scriptsize $1$};
	\node at (5.4, -0.5) {\scriptsize $1$};
	\node at (1.4, -0.5) {\scriptsize $1$};
	\node at (-1.4, -0.5) {\scriptsize $1$};
	\node at (2.6, -4.5) {\scriptsize $1$};
	\node at (5.4, -4.5) {\scriptsize $1$};
	\node at (1.4, -4.5) {\scriptsize $1$};
	\node at (-1.4, -4.5) {\scriptsize $1$};
	\node at (0.1, -0.65) { $\cdots$};
	\node at (4.1, -0.65) { $\cdots$};
	\node at (0.1, -4.6) { $\cdots$};
	\node at (4.1, -4.6) { $\cdots$};
%	\node at (2.1,-2.6) { $\cdots$};
%	\node at (2,-2.2) { $\vdots$};
	\node at (0,2) {\scriptsize $l$};
	\node at (4,2) {\scriptsize $k$};
	\node at (0,-7) {\scriptsize $k$};
	\node at (4,-7) {\scriptsize $l$};
\et
};
\endxy \ .
\]
\end{definition}
The following lemma shows that $\beta_{\up_{k}, \up_{l}}$ is its own inverse.  It also shows that if one web can be obtained from another by isotopies and sliding dots, merges, and splits along strands, then they are equal in $\qwebs$.

\begin{lemma}\label{L:braidingforups}
 The following relations hold in $\qwebs$ for all $h,k,l\in\Z_{>0}$. 
\begin{enumerate}
\item [(a.)]
\[
\xy
(0,0)*{
\bt[color=\clr, scale=1.25]
	\draw[thick, directed=1
	] (0,0) to (0.5,0.5);
	\draw[thick, directed=1
	] (0.5,0) to (0,0.5);
	\draw[thick, ] (0,-0.5) to (0.5,0);
	\draw[thick, ] (0.5,-0.5) to (0,0);
	\node at (0,-0.65) {\scriptsize $k$};
	\node at (0,.65) {\scriptsize $k$};
	\node at (0.5,-.65) {\scriptsize $l$};
	\node at (0.5,.65) {\scriptsize $l$};
\et
};
\endxy=
\xy
(0,0)*{
\bt[color=\clr, scale=1.25]
	\draw[thick, directed=1
	] (0.5,-0.5) to (0.5,0.5);
	\draw[thick, rdirected=0.05
	] (0,0.5) to (0,-0.5);
	\node at (0,-.65) {\scriptsize $k$};
	\node at (0,.65) {\scriptsize $k$};
	\node at (0.5,-.65) {\scriptsize $l$};
	\node at (0.5,.65) {\scriptsize $l$};
\et
};
\endxy ,
\]
\item [(b.)] 
\[
\xy
(0,0)*{
\bt[color=\clr, scale=1.25]
	\draw[thick, directed=1] (0,0) to (1,1) to (1,1.5);
	\draw[thick, directed=1] (0.5,0) to (0,0.5) to (0,1) to (0.5,1.5);
	\draw[thick, directed=1] (1,0) to (1,0.5) to (0.5,1) to (0,1.5);
	\node at (0,-0.15) {\scriptsize $h$};
	\node at (0,1.65) {\scriptsize $l$};
	\node at (0.5,-0.15) {\scriptsize $k$};
	\node at (0.5,1.65) {\scriptsize $k$};
	\node at (1,-0.15) {\scriptsize $l$};
	\node at (1,1.65) {\scriptsize $h$};
\et
};
\endxy=
\xy
(0,0)*{
\bt[color=\clr, scale=1.25]
	\draw[thick, directed=1] (0,0) to (0,0.5) to (1,1.5);
	\draw[thick, directed=1] (0.5,0) to (1,0.5) to (1,1) to (0.5,1.5);
	\draw[thick, directed=1] (1,0) to (0,1) to (0,1.5);
	\node at (0,-0.15) {\scriptsize $h$};
	\node at (0,1.65) {\scriptsize $l$};
	\node at (0.5,-0.15) {\scriptsize $k$};
	\node at (0.5,1.65) {\scriptsize $k$};
	\node at (1,-0.15) {\scriptsize $l$};
	\node at (1,1.65) {\scriptsize $h$};
\et
};
\endxy ,
\]
\item [(c.)]
\[
\xy
(0,0)*{
\bt[color=\clr, scale=.35]
	\draw [ thick, directed=1] (0, .75) to (0,1.5) to (2,3.5);
	\draw [ thick, directed=.65] (1,-1) to [out=90,in=330] (0,.75);
	\draw [ thick, directed=.65] (-1,-1) to [out=90,in=210] (0,.75);
	\draw [ thick, directed=1] (2,-1) to (2,1.5) to (0,3.5);
	\node at (2, 4) {\scriptsize $h\! +\! k$};
	\node at (-1,-1.5) {\scriptsize $h$};
	\node at (1,-1.5) {\scriptsize $k$};
	\node at (2,-1.5) {\scriptsize $l$};
	\node at (0,4) {\scriptsize $l$};
\et
};
\endxy =
\xy
(0,0)*{
\bt[color=\clr, scale=.35]
	\draw [ thick, directed=1] (0, .75) to (0,1.5);
	\draw [ thick, directed=.65] (1,-1) to [out=90,in=330] (0,.75);
	\draw [ thick, directed=.65] (-1,-1) to [out=90,in=210] (0,.75);
	\draw [ thick, ] (-3,-3) to (-1,-1);
	\draw [ thick, ] (-1,-3) to (1,-1);
	\draw [ thick, directed=1] (1,-3) to (-2,-1) to (-2,1.5);
	\node at (0, 2) {\scriptsize $h\! +\! k$};
	\node at (-1,-3.5) {\scriptsize $k$};
	\node at (-3,-3.5) {\scriptsize $h$};
	\node at (-2,2) {\scriptsize $l$};
	\node at (1,-3.5) {\scriptsize $l$};
\et
};
\endxy \ ,\quad\quad
\xy
(0,0)*{\reflectbox{
\bt[color=\clr, scale=.35]
	\draw [ thick, directed=1] (0, .75) to (0,1.5) to (2,3.5);
	\draw [ thick, directed=.65] (1,-1) to [out=90,in=330] (0,.75);
	\draw [ thick, directed=.65] (-1,-1) to [out=90,in=210] (0,.75);
	\draw [ thick, directed=1] (2,-1) to (2,1.5) to (0,3.5);
	\node at (2, 4) {\reflectbox{\scriptsize $h\! +\! k$}};
	\node at (-1,-1.5) {\reflectbox{\scriptsize $k$}};
	\node at (1,-1.5) {\reflectbox{\scriptsize $h$}};
	\node at (2,-1.5) {\reflectbox{\scriptsize $l$}};
	\node at (0,4) {\reflectbox{\scriptsize $l$}};
\et
}};
\endxy=
\xy
(0,0)*{\reflectbox{
\bt[color=\clr, scale=.35]
	\draw [ thick, directed=1] (0, .75) to (0,1.5);
	\draw [ thick, directed=.65] (1,-1) to [out=90,in=330] (0,.75);
	\draw [ thick, directed=.65] (-1,-1) to [out=90,in=210] (0,.75);
	\draw [ thick, ] (-3,-3) to (-1,-1);
	\draw [ thick, ] (-1,-3) to (1,-1);
	\draw [ thick, directed=1] (1,-3) to (-2,-1) to (-2,1.5);
	\node at (0, 2) {\reflectbox{\scriptsize $h\! +\! k$}};
	\node at (-1,-3.5) {\reflectbox{\scriptsize $h$}};
	\node at (-3,-3.5) {\reflectbox{\scriptsize $k$}};
	\node at (-2,2) {\reflectbox{\scriptsize $l$}};
	\node at (1,-3.5) {\reflectbox{\scriptsize $l$}};
\et
}};
\endxy ,
\]
\[
\xy
(0,0)*{\rotatebox{180}{
\bt[color=\clr, scale=.35]
	\draw [ thick, rdirected=.6] (0, .75) to (0,1.5);
	\draw [ thick, ] (1,-1) to [out=90,in=330] (0,.75);
	\draw [ thick, ] (-1,-1) to [out=90,in=210] (0,.75);
	\draw [ thick, rdirected=0.1] (-3,-3) to (-1,-1);
	\draw [ thick, rdirected=0.1] (-1,-3) to (1,-1);
	\draw [ thick, rdirected=.05] (1,-3) to (-2,-1) to (-2,1.5);
	\node at (0, 2) {\rotatebox{180}{\scriptsize $h\! +\! k$}};
	\node at (-1,-3.5) {\rotatebox{180}{\scriptsize $h$}};
	\node at (-3,-3.5) {\rotatebox{180}{\scriptsize $k$}};
	\node at (-2,2) {\rotatebox{180}{\scriptsize $l$}};
	\node at (1,-3.5) {\rotatebox{180}{\scriptsize $l$}};
\et
}};
\endxy=
\xy
(0,0)*{\rotatebox{180}{
\bt[color=\clr, scale=.35]
	\draw [ thick, rdirected=.15] (0, .75) to (0,1.5) to (2,3.5);
	\draw [ thick, rdirected=.1] (1,-1) to [out=90,in=330] (0,.75);
	\draw [ thick, rdirected=.1] (-1,-1) to [out=90,in=210] (0,.75);
	\draw [ thick, rdirected=.05] (2,-1) to (2,1.5) to (0,3.5);
	\node at (2, 4) {\rotatebox{180}{\scriptsize $h\! +\! k$}};
	\node at (-1,-1.5) {\rotatebox{180}{\scriptsize $k$}};
	\node at (1,-1.5) {\rotatebox{180}{\scriptsize $h$}};
	\node at (2,-1.5) {\rotatebox{180}{\scriptsize $l$}};
	\node at (0,4) {\rotatebox{180}{\scriptsize $l$}};
\et
}};
\endxy \ ,\quad\quad
\xy
(0,0)*{\rotatebox{180}{\reflectbox{
\bt[color=\clr, scale=.35]
	\draw [ thick, rdirected=.6] (0, .75) to (0,1.5);
	\draw [ thick, ] (1,-1) to [out=90,in=330] (0,.75);
	\draw [ thick, ] (-1,-1) to [out=90,in=210] (0,.75);
	\draw [ thick, rdirected=0.1] (-3,-3) to (-1,-1);
	\draw [ thick, rdirected=0.1] (-1,-3) to (1,-1);
	\draw [ thick, rdirected=.05] (1,-3) to (-2,-1) to (-2,1.5);
	\node at (0, 2) {\rotatebox{180}{\reflectbox{\scriptsize $h\! +\! k$}}};
	\node at (-1,-3.5) {\rotatebox{180}{\reflectbox{\scriptsize $k$}}};
	\node at (-3,-3.5) {\rotatebox{180}{\reflectbox{\scriptsize $h$}}};
	\node at (-2,2) {\rotatebox{180}{\reflectbox{\scriptsize $l$}}};
	\node at (1,-3.5) {\rotatebox{180}{\reflectbox{\scriptsize $l$}}};
\et
}}};
\endxy=
\xy
(0,0)*{\rotatebox{180}{\reflectbox{
\bt[color=\clr, scale=.35]
	\draw [ thick, rdirected=.15] (0, .75) to (0,1.5) to (2,3.5);
	\draw [ thick, rdirected=.1] (1,-1) to [out=90,in=330] (0,.75);
	\draw [ thick, rdirected=.1] (-1,-1) to [out=90,in=210] (0,.75);
	\draw [ thick, rdirected=.05] (2,-1) to (2,1.5) to (0,3.5);
	\node at (2, 4) {\rotatebox{180}{\reflectbox{\scriptsize $h\! +\! k$}}};
	\node at (-1,-1.5) {\rotatebox{180}{\reflectbox{\scriptsize $h$}}};
	\node at (1,-1.5) {\rotatebox{180}{\reflectbox{\scriptsize $k$}}};
	\node at (2,-1.5) {\rotatebox{180}{\reflectbox{\scriptsize $l$}}};
	\node at (0,4) {\rotatebox{180}{\reflectbox{\scriptsize $l$}}};
\et
}}};
\endxy ,
\]

\item[(d.)] 
\[
\xy
(0,0)*{
\bt[color=\clr, scale=1.25]
	\draw[thick, directed=1] (0,0) to (0.5,1);
	\draw[thick, directed=1] (0.5,0) to (0,1);
	\node at (0,-0.15) {\scriptsize $k$};
	\node at (0,1.15) {\scriptsize $l$};
	\node at (0.5,-0.15) {\scriptsize $l$};
	\node at (0.5,1.15) {\scriptsize $k$};
	\draw
	 (0.15,0.3) \wdot;
\et
};
\endxy=
\xy
(0,0)*{
\bt[color=\clr, scale=1.25]
	\draw[thick, directed=1] (0,0) to (0.5,1);
	\draw[thick, directed=1] (0.5,0) to (0,1);
	\node at (0,-0.15) {\scriptsize $k$};
	\node at (0,1.15) {\scriptsize $l$};
	\node at (0.5,-0.15) {\scriptsize $l$};
	\node at (0.5,1.15) {\scriptsize $k$};
	\draw
	 (0.35,0.7) \wdot;
\et
};
\endxy \ ,\quad\quad
\xy
(0,0)*{\reflectbox{
\bt[color=\clr, scale=1.25]
	\draw[thick, directed=1] (0,0) to (0.5,1);
	\draw[thick, directed=1] (0.5,0) to (0,1);
	\node at (0,-0.15) {\reflectbox{\scriptsize $l$}};
	\node at (0,1.15) {\reflectbox{\scriptsize $k$}};
	\node at (0.5,-0.15) {\reflectbox{\scriptsize $k$}};
	\node at (0.5,1.15) {\reflectbox{\scriptsize $l$}};
	\draw
	 (0.15,0.3) \wdot;
\et
}};
\endxy=
\xy
(0,0)*{\reflectbox{
\bt[color=\clr, scale=1.25]
	\draw[thick, directed=1] (0,0) to (0.5,1);
	\draw[thick, directed=1] (0.5,0) to (0,1);
	\node at (0,-0.15) {\reflectbox{\scriptsize $l$}};
	\node at (0,1.15) {\reflectbox{\scriptsize $k$}};
	\node at (0.5,-0.15) {\reflectbox{\scriptsize $k$}};
	\node at (0.5,1.15) {\reflectbox{\scriptsize $l$}};
	\draw
	 (0.35,0.7) \wdot;
\et
}};
\endxy .
\]
\end{enumerate}
\end{lemma}

\begin{proof}
\bea
\xy
(0,0)*{
\bt[color=\clr, scale=1.25]
	\draw[thick, directed=1] (0,0) to (0.5,0.5);
	\draw[thick, directed=1] (0.5,0) to (0,0.5);
	\draw[thick, ] (0,-0.5) to (0.5,0);
	\draw[thick, ] (0.5,-0.5) to (0,0);
	\node at (0,-0.65) {\scriptsize $k$};
	\node at (0,.65) {\scriptsize $k$};
	\node at (0.5,-.65) {\scriptsize $l$};
	\node at (0.5,.65) {\scriptsize $l$};
\et
};
\endxy &=& \left(\frac{1}{k!\,l!}\right)^2
\xy
(0,0)*{
\bt[color=\clr, scale=.35]
	\draw [ thick, directed=0.7] (0, .75) to (0,2.5);
	\draw [ thick, directed=0.75] (1,-1) to [out=90,in=330] (0,.75);
	\draw [ thick, directed=0.75] (-1,-1) to [out=90,in=210] (0,.75);
	\draw [ thick, directed=0.7] (4, .75) to (4,2.5);
	\draw [ thick, directed=0.75] (5,-1) to [out=90,in=330] (4,.75);
	\draw [ thick, directed=0.75] (3,-1) to [out=90,in=210] (4,.75);
	\draw [ thick, directed=0.75] (0,-6.75) to (0,-5.75);
	\draw [ thick, ] (0,-5.75) to [out=30,in=270] (1,-4);
	\draw [ thick, ] (0,-5.75) to [out=150,in=270] (-1,-4); 
	\draw [ thick, directed=0.75] (4,-6.75) to (4,-5.75);
	\draw [ thick, ] (4,-5.75) to [out=30,in=270] (5,-4);
	\draw [ thick, ] (4,-5.75) to [out=150,in=270] (3,-4); 
	\draw [ thick ] (5,-1) to (1,-4);
	\draw [ thick ] (3,-1) to (-1,-4);
	\draw [ thick ] (1,-1) to (5,-4);
	\draw [ thick ] (-1,-1) to (3,-4);
	\node at (0.1, -0.75) { $\cdots$};
	\node at (4.1, -0.75) { $\cdots$};
	\node at (0.1, -4.5) { $\cdots$};
	\node at (4.1, -4.5) { $\cdots$};
	\node at (0,10.5) {\scriptsize $k$};
	\node at (4,10.5) {\scriptsize $l$};
	\node at (0,-7.25) {\scriptsize $k$};
	\node at (4,-7.25) {\scriptsize $l$};
	\node at (2.6, -4.5) {\scriptsize $1$};
	\node at (5.4, -4.5) {\scriptsize $1$};
	\node at (1.4, -4.5) {\scriptsize $1$};
	\node at (-1.4, -4.5) {\scriptsize $1$};
	\node at (2.6, 7.75) {\scriptsize $1$};
	\node at (5.4, 7.75) {\scriptsize $1$};
	\node at (1.4, 7.75) {\scriptsize $1$};
	\node at (-1.4, 7.75) {\scriptsize $1$};
	\draw [ thick, directed=1] (0, 9) to (0,10);
	\draw [ thick, directed=0.75] (1,7.25) to [out=90,in=330] (0,9);
	\draw [ thick, directed=0.75] (-1,7.25) to [out=90,in=210] (0,9);
	\draw [ thick, directed=1] (4,9) to (4,10);
	\draw [ thick, directed=0.75] (5,7.25) to [out=90,in=330] (4,9);
	\draw [ thick, directed=0.75] (3,7.25) to [out=90,in=210] (4,9);
	\draw [ thick, ] (0,2.5) to [out=30,in=270] (1,4.25);
	\draw [ thick, ] (0,2.5) to [out=150,in=270] (-1,4.25); 
	\draw [ thick, ] (4,2.5) to [out=30,in=270] (5,4.25);
	\draw [ thick, ] (4,2.5) to [out=150,in=270] (3,4.25); 
	\draw [ thick ] (5,7.25) to (1,4.25);
	\draw [ thick ] (3,7.25) to (-1,4.25);
	\draw [ thick ] (1,7.25) to (5,4.25);
	\draw [ thick ] (-1,7.25) to (3,4.25);
	\node at (0.1, 7.5) { $\cdots$};
	\node at (4.1, 7.5) { $\cdots$};
	\node at (0.1, 3.75) { $\cdots$};
	\node at (4.1, 3.75) { $\cdots$};
	\node at (0.75, 1.75) {\scriptsize $l$};
	\node at (4.75, 1.75) {\scriptsize $k$};
\et
};
\endxy \stackrel{\eqref{clasp-sum}}{=} \left(\frac{1}{k!\,l!}\right)^2\sum_{\substack{\sigma\in\Sigma_l\\ \tau\in\frakS_k}} 
\xy
(0,0)*{
\bt[color=\clr, scale=.35]
	\draw [ thick, directed=0.75] (0,-6.75) to (0,-5.75);
	\draw [ thick, ] (0,-5.75) to [out=30,in=270] (1,-4);
	\draw [ thick, ] (0,-5.75) to [out=150,in=270] (-1,-4); 
	\draw [ thick, directed=0.75] (4,-6.75) to (4,-5.75);
	\draw [ thick, ] (4,-5.75) to [out=30,in=270] (5,-4);
	\draw [ thick, ] (4,-5.75) to [out=150,in=270] (3,-4); 
	\draw [ thick ] (5,-1) to (1,-4);
	\draw [ thick ] (3,-1) to (-1,-4);
	\draw [ thick ] (1,-1) to (5,-4);
	\draw [ thick ] (-1,-1) to (3,-4);
	\draw [ thick ] (5,-1) to (5,0);
	\draw [ thick ] (3,-1) to (3,0);
	\draw [ thick ] (1,-1) to (1,0);
	\draw [ thick ] (-1,-1) to (-1,0);
	\node at (0.1, -0.75) { $\cdots$};
	\node at (4.1, -0.75) { $\cdots$};
	\node at (0.1, -4.5) { $\cdots$};
	\node at (4.1, -4.5) { $\cdots$};
	\node at (0,10.5) {\scriptsize $k$};
	\node at (4,10.5) {\scriptsize $l$};
	\node at (0,-7.25) {\scriptsize $k$};
	\node at (4,-7.25) {\scriptsize $l$};
	\node at (2.6, -4.5) {\scriptsize $1$};
	\node at (5.4, -4.5) {\scriptsize $1$};
	\node at (1.4, -4.5) {\scriptsize $1$};
	\node at (-1.4, -4.5) {\scriptsize $1$};
	\node at (2.6, 7.75) {\scriptsize $1$};
	\node at (5.4, 7.75) {\scriptsize $1$};
	\node at (1.4, 7.75) {\scriptsize $1$};
	\node at (-1.4, 7.75) {\scriptsize $1$};
	\draw [ thick, directed=1] (0, 9) to (0,10);
	\draw [ thick, directed=0.75] (1,7.25) to [out=90,in=330] (0,9);
	\draw [ thick, directed=0.75] (-1,7.25) to [out=90,in=210] (0,9);
	\draw [ thick, directed=1] (4,9) to (4,10);
	\draw [ thick, directed=0.75] (5,7.25) to [out=90,in=330] (4,9);
	\draw [ thick, directed=0.75] (3,7.25) to [out=90,in=210] (4,9);
	\draw [ thick ] (5,7.25) to (1,4.25);
	\draw [ thick ] (3,7.25) to (-1,4.25);
	\draw [ thick ] (1,7.25) to (5,4.25);
	\draw [ thick ] (-1,7.25) to (3,4.25);
	\draw [ thick ] (1,4.25) to (1,3.25);
	\draw [ thick ] (-1,4.25) to (-1,3.25);
	\draw [ thick ] (3,4.25) to (3,3.25);
	\draw [ thick ] (5,4.25) to (5,3.25);
	\draw [ thick ] (-1.5,3.25) rectangle (1.5,0);
	\draw [ thick ] (2.5,3.25) rectangle (5.5,0);
	\node at (0.1, 7.5) { $\cdots$};
	\node at (4.1, 7.5) { $\cdots$};
	\node at (0.1, 3.75) { $\cdots$};
	\node at (4.1, 3.75) { $\cdots$};
	\node at (0, 1.75) { $\sigma$};
	\node at (4, 1.75) { $\tau$};
\et
};
\endxy\\ 
& \stackrel{\text{Lemma}~\ref{L:braidingforups}}{=} & \left(\frac{1}{k!\,l!}\right)^2\sum_{\substack{\sigma\in\frakS_l\\ \tau\in\frakS_k}}
\xy
(0,0)*{
\bt[color=\clr, scale=.35]
	\draw [ thick, directed=0.75] (0,-6.75) to (0,-5.75);
	\draw [ thick, ] (0,-5.75) to [out=30,in=270] (1,-4);
	\draw [ thick, ] (0,-5.75) to [out=150,in=270] (-1,-4); 
	\draw [ thick, directed=0.75] (4,-6.75) to (4,-5.75);
	\draw [ thick, ] (4,-5.75) to [out=30,in=270] (5,-4);
	\draw [ thick, ] (4,-5.75) to [out=150,in=270] (3,-4); 
	\draw [ thick ] (5,-1) to (1,-4);
	\draw [ thick ] (3,-1) to (-1,-4);
	\draw [ thick ] (1,-1) to (5,-4);
	\draw [ thick ] (-1,-1) to (3,-4);
	\draw [ thick ] (5,-1) to (5,0);
	\draw [ thick ] (3,-1) to (3,0);
	\draw [ thick ] (1,-1) to (1,0);
	\draw [ thick ] (-1,-1) to (-1,0);
	\node at (0.1, -0.75) { $\cdots$};
	\node at (4.1, -0.75) { $\cdots$};
	\node at (0.1, -4.5) { $\cdots$};
	\node at (4.1, -4.5) { $\cdots$};
	\node at (0,10.5) {\scriptsize $k$};
	\node at (4,10.5) {\scriptsize $l$};
	\node at (0,-7.25) {\scriptsize $k$};
	\node at (4,-7.25) {\scriptsize $l$};
	\node at (2.6, -4.5) {\scriptsize $1$};
	\node at (5.4, -4.5) {\scriptsize $1$};
	\node at (1.4, -4.5) {\scriptsize $1$};
	\node at (-1.4, -4.5) {\scriptsize $1$};
	\node at (2.6, 7.75) {\scriptsize $1$};
	\node at (5.4, 7.75) {\scriptsize $1$};
	\node at (1.4, 7.75) {\scriptsize $1$};
	\node at (-1.4, 7.75) {\scriptsize $1$};
	\draw [ thick, directed=1] (0, 9) to (0,10);
	\draw [ thick, directed=0.75] (1,7.25) to [out=90,in=330] (0,9);
	\draw [ thick, directed=0.75] (-1,7.25) to [out=90,in=210] (0,9);
	\draw [ thick, directed=1] (4,9) to (4,10);
	\draw [ thick, directed=0.75] (5,7.25) to [out=90,in=330] (4,9);
	\draw [ thick, directed=0.75] (3,7.25) to [out=90,in=210] (4,9);
	\draw [ thick ] (5,3.25) to (1,0);
	\draw [ thick ] (3,3.25) to (-1,0);
	\draw [ thick ] (1,3.25) to (5,0);
	\draw [ thick ] (-1,3.25) to (3,0);
	\draw [ thick ] (1,4.25) to (1,3.25);
	\draw [ thick ] (-1,4.25) to (-1,3.25);
	\draw [ thick ] (3,4.25) to (3,3.25);
	\draw [ thick ] (5,4.25) to (5,3.25);
	\draw [ thick ] (-1.5,7.25) rectangle (1.5,4.25);
	\draw [ thick ] (2.5,7.25) rectangle (5.5,4.25);
	\node at (0.1, 7.5) { $\cdots$};
	\node at (4.1, 7.5) { $\cdots$};
	\node at (0.1, 3.75) { $\cdots$};
	\node at (4.1, 3.75) { $\cdots$};
	\node at (0, 5.75) { $\tau$};
	\node at (4, 5.75) { $\sigma$};
\et
};
\endxy\stackrel{\eqref{E:untwist-permutation}}{=}\left(\frac{1}{k!\,l!}\right)^2\sum_{\substack{\sigma\in\frakS_l\\ \tau\in\frakS_k}} 
\xy
(0,0)*{
\bt[color=\clr, scale=.35]
	\draw [ thick, directed=0.75] (0,-6.75) to (0,-5.75);
	\draw [ thick, ] (0,-5.75) to [out=30,in=270] (1,-4);
	\draw [ thick, ] (0,-5.75) to [out=150,in=270] (-1,-4); 
	\draw [ thick, directed=0.75] (4,-6.75) to (4,-5.75);
	\draw [ thick, ] (4,-5.75) to [out=30,in=270] (5,-4);
	\draw [ thick, ] (4,-5.75) to [out=150,in=270] (3,-4); 
	\draw [ thick ] (5,-1) to (1,-4);
	\draw [ thick ] (3,-1) to (-1,-4);
	\draw [ thick ] (1,-1) to (5,-4);
	\draw [ thick ] (-1,-1) to (3,-4);
	\draw [ thick ] (5,-1) to (5,0);
	\draw [ thick ] (3,-1) to (3,0);
	\draw [ thick ] (1,-1) to (1,0);
	\draw [ thick ] (-1,-1) to (-1,0);
	\node at (0.1, -0.75) { $\cdots$};
	\node at (4.1, -0.75) { $\cdots$};
	\node at (0.1, -4.5) { $\cdots$};
	\node at (4.1, -4.5) { $\cdots$};
	\node at (0,6.75) {\scriptsize $k$};
	\node at (4,6.75) {\scriptsize $l$};
	\node at (0,-7.25) {\scriptsize $k$};
	\node at (4,-7.25) {\scriptsize $l$};
	\node at (2.6, -4.5) {\scriptsize $1$};
	\node at (5.4, -4.5) {\scriptsize $1$};
	\node at (1.4, -4.5) {\scriptsize $1$};
	\node at (-1.4, -4.5) {\scriptsize $1$};
	\node at (2.6, 3.75) {\scriptsize $1$};
	\node at (5.4, 3.75) {\scriptsize $1$};
	\node at (1.4, 3.75) {\scriptsize $1$};
	\node at (-1.4, 3.75) {\scriptsize $1$};
	\draw [ thick, directed=1] (0, 5.25) to (0,6.25);
	\draw [ thick, directed=0.75] (1,3.5) to [out=90,in=330] (0,5.25);
	\draw [ thick, directed=0.75] (-1,3.5) to [out=90,in=210] (0,5.25);
	\draw [ thick, directed=1] (4,5.25) to (4,6.25);
	\draw [ thick, directed=0.75] (5,3.5) to [out=90,in=330] (4,5.25);
	\draw [ thick, directed=0.75] (3,3.5) to [out=90,in=210] (4,5.25);
	\draw [ thick ] (5,3.5) to (1,0);
	\draw [ thick ] (3,3.5) to (-1,0);
	\draw [ thick ] (1,3.5) to (5,0);
	\draw [ thick ] (-1,3.5) to (3,0);
	\node at (0.1, 3.75) { $\cdots$};
	\node at (4.1, 3.75) { $\cdots$};
\et
};
\endxy\\
 & \stackrel{\text{Lemma}~\ref{L:braidingforups}}{=} & \left(\frac{1}{k!\,l!}\right)^2\sum_{\substack{\sigma\in\frakS_l\\ \tau\in\frakS_k}} 
\xy
(0,0)*{
\begin{tikzpicture}[color=\clr, scale=.3]
	\draw [ thick, directed=1] (0,.75) to (0,2);
	\draw [ thick, directed=.85] (0,-2.75) to [out=30,in=330] (0,.75);
	\draw [ thick, directed=.85] (0,-2.75) to [out=150,in=210] (0,.75);
	\draw [ thick, directed=.65] (0,-4) to (0,-2.75);
	\node at (0,-4.5) {\scriptsize $k$};
	\node at (0,2.5) {\scriptsize $k$};
	\node at (-1.5,-1) {\scriptsize $1$};
	\node at (1.5,-1) {\scriptsize $1$};
	\node at (0,-1) { \ $\cdots$};
\end{tikzpicture}
};
\endxy
\xy
(0,0)*{
\begin{tikzpicture}[color=\clr, scale=.3]
	\draw [ thick, directed=1] (0,.75) to (0,2);
	\draw [ thick, directed=.85] (0,-2.75) to [out=30,in=330] (0,.75);
	\draw [ thick, directed=.85] (0,-2.75) to [out=150,in=210] (0,.75);
	\draw [ thick, directed=.65] (0,-4) to (0,-2.75);
	\node at (0,-4.5) {\scriptsize $l$};
	\node at (0,2.5) {\scriptsize $l$};
	\node at (-1.5,-1) {\scriptsize $1$};
	\node at (1.5,-1) {\scriptsize $1$};
	\node at (0,-1) { \ $\cdots$};
\end{tikzpicture}
};
\endxy
\stackrel{\eqref{digon-removal}}{=} 
\xy
(0,0)*{
\bt[color=\clr, scale=1.25]
	\draw[thick, directed=1] (0.5,-0.5) to (0.5,0.5);
	\draw[thick, rdirected=0.05] (0,0.5) to (0,-0.5);
	\node at (0,-.65) {\scriptsize $k$};
	\node at (0,.65) {\scriptsize $k$};
	\node at (0.5,-.65) {\scriptsize $l$};
	\node at (0.5,.65) {\scriptsize $l$};
\et
};
\endxy \ .
\eea
We leave the remaining identities as an exercise for the reader.
\end{proof}

%Let $r >1$ be fixed and let $\lambda = (\lambda_{1},\dotsc , \lambda_{r})$ be an $r$-tuple of positive integers.  Given $\sigma \in \Sigma_{r}$, let $\lambda^{\sigma}= (\lambda_{\sigma(1), \dotsc  , \lambda_{\sigma(r)}}$. In light of the previous lemma,there is a well defined morphism $\beta$
We now define the braiding for arbitrary objects in $\qwebs$.

\begin{definition}\label{D:generalupbraiding} Let $\ob{a}=(a_{1}, \dotsc , a_{r})$ and $\ob{b}=(b_{1}, \dotsc , b_{s})$  be two tuples of nonnegative integers and let $\up_{\ob{a}}$ and $\up_{\ob{b}}$ be the corresponding objects of $\qwebs$.  Define 
\[
\beta_{\up_{\ob{a}},\up_{\ob{b}}} =  %%%%%%%%%%%%%%%%%%%%%%%%%%%%%%%%%%%%%%%%%%%% Arbitrary braiding
\xy
(0,0)*{
\bt[color=\clr, scale=.35]
	\draw [ thick, directed=1] (0,-1) to (0,0) to (5,5) to (5,6);
	\draw [ thick, directed=1] (3,-1) to (3,0) to (8,5) to (8,6);
	\draw [ thick, directed=1] (5,-1) to (5,0) to (0,5) to (0,6);
	\draw [ thick, directed=1] (8,-1) to (8,0) to (3,5) to (3,6);
	\node at (1.5,-0.5) {$\cdots$};
	\node at (1.5,5.5) {$\cdots$};
	\node at (6.5,-0.5) {$\cdots$};
	\node at (6.5,5.5) {$\cdots$};
	\node at (0,-1.5) {\scriptsize $\ob{a}_1$};
	\node at (3,-1.5) {\scriptsize $\ob{a}_r$};
	\node at (5,-1.5) {\scriptsize $\ob{b}_1$};
	\node at (8,-1.5) {\scriptsize $\ob{b}_{s}$};
	\node at (0,6.5) {\scriptsize $\ob{b}_1$};
	\node at (3,6.5) {\scriptsize $\ob{b}_{s}$};
	\node at (5,6.5) {\scriptsize $\ob{a}_1$};
	\node at (8,6.5) {\scriptsize $\ob{a}_r$};
\et
};
\endxy
\]
\end{definition}
With these and \cref{L:braidingforups}, the following is immediate. 
\begin{theorem}\label{T:braiding}  The morphisms $\beta_{\up_{\ob{a}}, \up_{\ob{b}}}$ define a symmetric braiding on the monoidal supercategory $\qwebs$.
\end{theorem}

\section{Functors  \texorpdfstring{$\Pi_m$ and $\Psi$}{Pim and Psim}, and  \texorpdfstring{$\End_{\qwebs}(\up_{1}^k)$}{Endups}}

We next relate the combinatorial category $\qwebs$ to the supercategory $\bUdot (\fqt (m))$ and to the representations of $\fq (n)$. In what follows, recall if a diagram has an edge labelled by a negative integer then the morphism given by that diagram is understood to be zero and an edge labelled by zero in a web can always be erased (or added). 

\subsection{The Functor  \texorpdfstring{$\Pi_m$}{Pim}}

\iffalse there is a functor $\Pi_m$ and a monoidal functor of supercategories, $\Psi$ such that the following diagram commutes for all $m$:
%
\begin{equation}\label{diagram}
\xy
(0,0)*{
\begin{tikzpicture}
\node (S) at (0,0) {$\bUdot(\fqt(m))_{\geq0}$};
\node (P) at (3,0) {$\mods$};
\node (W) at (0,-2) {$\qwebs$};
\path[->, thick]
(S) edge [above] node {$\Phi_m$} (P)
(S) edge [left] node {$\Pi_m$} (W)
(W) edge [below] node {\quad $\Psi$} (P);
\end{tikzpicture}
};
\endxy.
\end{equation}
%
For $\lambda=(\lambda_1,\dots,\lambda_m)\in\Z^m_{\geq 0}$, recall that $\up_{\lambda}$ denotes the word $\up_{\lambda_{1}} \dotsb \up_{\lambda_{m}}$.
\fi

\begin{proposition}\label{pi-functor}
For every $m\geq 1$ there exists a functor of monoidal supercategories 
\[
\Pi_m: \bUdot(\fq(m))_{\geq 0} \to\qwebs
\]
given on objects $\lambda = \sum_{i=1}^{m} \lambda_{i}\varepsilon_{i} \in X(T)_{\geq 0}$ by 
\[
\Pi_m \left(\lambda \right) = \up_{\lambda} = \up_{\lambda_{1}} \up_{\lambda_{2}}\dotsb \up_{\lambda_{m}},
\]
and on the divided powers of the generating morphisms by 
\begin{equation*}
\Pi_m(e_i^{(j)}1_{\lambda})=
\xy
(0,0)*{
\begin{tikzpicture}[color=\clr ]
	\draw[ thick, directed=.25, directed=1] (0,0) to (0,1.5);
	\node at (0,-0.15) {\scriptsize $\lambda_{i}$};
	\node at (0,1.65) {\scriptsize $\lambda_{i}\!+\!j$};
	\draw[ thick, directed=.25, directed=1] (1,0) to (1,1.5);
	\node at (1,-0.15) {\scriptsize $\lambda_{i+1}$};
	\node at (1,1.65) {\scriptsize $\lambda_{i+1}\!-\!j$};
	\draw[ thick, directed=.55] (1,0.75) to (0,0.75);
	\node at (0.5,1) {\scriptsize$j$};
	\draw [thick, directed=1] (2,0) to (2,1.5);
	\draw [thick, directed=1] (2.75,0) to (2.75,1.5);
	\draw [thick, directed=1] (-1.75,0) to (-1.75,1.5);
	\draw [thick, directed=1] (-1,0) to (-1,1.5);
	\node at (-1.75,-0.15) {\scriptsize $\lambda_1$};
	\node at (-1,-0.15) {\scriptsize $\lambda_{i-1}$};
	\node at (-1.75,1.65) {\scriptsize $\lambda_1$};
	\node at (-1,1.65) {\scriptsize $\lambda_{i-1}$};
	\node at (-1.35,0.75) {$\cdots$};
	\node at (2,-0.15) {\scriptsize $\lambda_{i+2}$};
	\node at (2.75,-0.15) {\scriptsize $\lambda_m$};
	\node at (2,1.65) {\scriptsize $\lambda_{i+2}$};
	\node at (2.75,1.65) {\scriptsize $\lambda_m$};
	\node at (2.4,0.75) {$\cdots$};
\end{tikzpicture}
};
\endxy \ ,
\end{equation*}
\begin{equation*}
\Pi_m(f_i^{(j)}1_{\lambda})=
\xy
(0,0)*{
\begin{tikzpicture}[color=\clr ]
	\draw[ thick, directed=.25, directed=1] (0,0) to (0,1.5);
	\node at (0,-0.15) {\scriptsize $\lambda_{i}$};
	\node at (0,1.65) {\scriptsize $\lambda_{i}\!-\!j$};
	\draw[ thick, directed=.25, directed=1] (1,0) to (1,1.5);
	\node at (1,-0.15) {\scriptsize $\lambda_{i+1}$};
	\node at (1,1.65) {\scriptsize $\lambda_{i+1}\!+\!j$};
	\draw[ thick, directed=.55] (0,0.75) to (1,0.75);
	\node at (0.5,1) {\scriptsize$j$};
	\draw [thick, directed=1] (2,0) to (2,1.5);
	\draw [thick, directed=1] (2.75,0) to (2.75,1.5);
	\draw [thick, directed=1] (-1.75,0) to (-1.75,1.5);
	\draw [thick, directed=1] (-1,0) to (-1,1.5);
	\node at (-1.75,-0.15) {\scriptsize $\lambda_1$};
	\node at (-1,-0.15) {\scriptsize $\lambda_{i-1}$};
	\node at (-1.75,1.65) {\scriptsize $\lambda_1$};
	\node at (-1,1.65) {\scriptsize $\lambda_{i-1}$};
	\node at (-1.35,0.75) {$\cdots$};
	\node at (2,-0.15) {\scriptsize $\lambda_{i+2}$};
	\node at (2.75,-0.15) {\scriptsize $\lambda_m$};
	\node at (2,1.65) {\scriptsize $\lambda_{i+2}$};
	\node at (2.75,1.65) {\scriptsize $\lambda_m$};
	\node at (2.4,0.75) {$\cdots$};
\end{tikzpicture}
};
\endxy \ ,
\end{equation*}
\begin{equation*}
\Pi_m(e_{\overline{i}}^{(j)}1_{\lambda})=
\xy
(0,0)*{
\begin{tikzpicture}[color=\clr ]
	\draw[ thick, directed=.25, directed=1] (0,0) to (0,1.5);
	\node at (0,-0.15) {\scriptsize $\lambda_{i}$};
	\node at (0,1.65) {\scriptsize $\lambda_{i}\!+\!j$};
	\draw[ thick, directed=.25, directed=1] (1,0) to (1,1.5);
	\node at (1,-0.15) {\scriptsize $\lambda_{i+1}$};
	\node at (1,1.65) {\scriptsize $\lambda_{i+1}\!-\!j$};
	\draw[ thick, directed=.55] (1,0.75) to (0,0.75);
	\node at (0.5,1) {\scriptsize$j$};
	\draw  (0.25,0.75) \wdot;
	\draw [thick, directed=1] (2,0) to (2,1.5);
	\draw [thick, directed=1] (2.75,0) to (2.75,1.5);
	\draw [thick, directed=1] (-1.75,0) to (-1.75,1.5);
	\draw [thick, directed=1] (-1,0) to (-1,1.5);
	\node at (-1.75,-0.15) {\scriptsize $\lambda_1$};
	\node at (-1,-0.15) {\scriptsize $\lambda_{i-1}$};
	\node at (-1.75,1.65) {\scriptsize $\lambda_1$};
	\node at (-1,1.65) {\scriptsize $\lambda_{i-1}$};
	\node at (-1.35,0.75) {$\cdots$};
	\node at (2,-0.15) {\scriptsize $\lambda_{i+2}$};
	\node at (2.75,-0.15) {\scriptsize $\lambda_m$};
	\node at (2,1.65) {\scriptsize $\lambda_{i+2}$};
	\node at (2.75,1.65) {\scriptsize $\lambda_m$};
	\node at (2.4,0.75) {$\cdots$};
\end{tikzpicture}
};
\endxy \ ,
\end{equation*}
\begin{equation*}
\Pi_m(f_{\overline{i}}^{(j)}1_{\lambda})=
\xy
(0,0)*{
\begin{tikzpicture}[color=\clr ]
	\draw[ thick, directed=.25, directed=1] (0,0) to (0,1.5);
	\node at (0,-0.15) {\scriptsize $\lambda_{i}$};
	\node at (0,1.65) {\scriptsize $\lambda_{i}\!-\!j$};
	\draw[ thick, directed=.25, directed=1] (1,0) to (1,1.5);
	\node at (1,-0.15) {\scriptsize $\lambda_{i+1}$};
	\node at (1,1.65) {\scriptsize $\lambda_{i+1}\!+\!j$};
	\draw[ thick, directed=.55] (0,0.75) to (1,0.75);
	\node at (0.5,1) {\scriptsize$j$};
	\draw  (0.75,0.75) \wdot;
	\draw [thick, directed=1] (2,0) to (2,1.5);
	\draw [thick, directed=1] (2.75,0) to (2.75,1.5);
	\draw [thick, directed=1] (-1.75,0) to (-1.75,1.5);
	\draw [thick, directed=1] (-1,0) to (-1,1.5);
	\node at (-1.75,-0.15) {\scriptsize $\lambda_1$};
	\node at (-1,-0.15) {\scriptsize $\lambda_{i-1}$};
	\node at (-1.75,1.65) {\scriptsize $\lambda_1$};
	\node at (-1,1.65) {\scriptsize $\lambda_{i-1}$};
	\node at (-1.35,0.75) {$\cdots$};
	\node at (2,-0.15) {\scriptsize $\lambda_{i+2}$};
	\node at (2.75,-0.15) {\scriptsize $\lambda_m$};
	\node at (2,1.65) {\scriptsize $\lambda_{i+2}$};
	\node at (2.75,1.65) {\scriptsize $\lambda_m$};
	\node at (2.4,0.75) {$\cdots$};
\end{tikzpicture}
};
\endxy \ ,
\end{equation*}
\begin{equation*}
\Pi_m(h_{\overline{i}}1_{\lambda})=
\xy
(0,0)*{
\begin{tikzpicture}[scale=1.25, color=\clr ]
	\draw [thick, directed=1] (0,0) to (0,1);
	\draw [thick, directed=1] (0.5,0) to (0.5,1);
	\node at (0,-0.15) {\scriptsize $\lambda_1$};
	\node at (0.5,-0.15) {\scriptsize $\lambda_{i-1}$};
	\node at (0,1.15) {\scriptsize $\lambda_1$};
	\node at (0.5,1.15) {\scriptsize $\lambda_{i-1}$};
	\node at (0.275,0.5) {$\cdots$};
	\draw[thick, directed=1] (1,0) to (1,1);
	\node at (1,-0.15) {\scriptsize $\lambda_{i}$};
	\node at (1,1.15) {\scriptsize $\lambda_{i}$};
	\draw (1,0.5) \wdot;
	\draw [thick, directed=1] (1.5,0) to (1.5,1);
	\draw [thick, directed=1] (2,0) to (2,1);
	\node at (1.5,-0.15) {\scriptsize $\lambda_{i+1}$};
	\node at (2,-0.15) {\scriptsize $\lambda_m$};
	\node at (1.5,1.15) {\scriptsize $\lambda_{i+1}$};
	\node at (2,1.15) {\scriptsize $\lambda_m$};
	\node at (1.775,0.5) {$\cdots$};
\end{tikzpicture}
};
\endxy \ .
\]
Moreover, $\Pi_m$ is a full functor onto the full subcategory of $\qwebs$ consisting of objects $\left\{\up_{\ob{a}} \mid \ob{a} \in \Z_{\geq 0}^{m} \right\}$.
\end{proposition}
\begin{proof}
To show $\Pi_m$ is well-defined it suffices to verify the defining relations of $\bUdot(\fqt(m))_{\geq 0}$ hold in $\qwebs$. This follows by direct calculations using the defining relations of $\qwebs$ along with the identities proven in \cref{SS:AdditionalRelations}.

That the image of $\Pi_m$ lies in the given subcategory is immediate.  That the functor is full amounts to the fact that the images of the generating morphisms of $\bUdot(\fqt(m))_{\geq0}$ also generate the morphisms of $\qwebs$.  Since every merge and split is a special case of a ladder, it is straightforward to verify that one can obtain every possible merge, split, and dot (possibly tensored by identity strands on the left and/or right) which lies in this subcategory. Since every web which is a morphism in this subcategory is a composition of such diagrams, it follows that $\Pi_{m}$ is full.
\end{proof}

\begin{remark}\label{R:Compatiblity2}  Recall the functors $\Theta_{m',m}$ from \cref{R:Compatibility}. The $\Pi$ functors are compatible in the sense that $\Pi_{m'} \circ \Theta_{m',m}$ and $\Pi_{m}$ are isomorphic functors for all positive integers $m' \geq m$.
\end{remark}

\subsection{The Functor  \texorpdfstring{$\Psi$}{Psim}}

As shown below, there is also a functor $\Psi_{n}: \qwebs\to\mods$.

\begin{proposition}\label{psi-functor}
For every $n \geq 1$ there exists an essentially surjective functor of monoidal supercategories, 
\[
\Psi_{n}: \qwebs\to\mods,
\]
given on objects by 
\[
\Psi_{n}\left(\up_{\lambda} \right)=S_{\lambda} = S^{\lambda}(V_{n}) = S^{\lambda_{1}}(V_{n}) \otimes\dotsb \otimes S^{\lambda_{t}}(V_{n}),
\] where $\lambda = (\lambda_{1}, \dotsc , \lambda_{t})$.  On morphisms the functor is defined by sending the dot, merge, and split, respectively, to the following maps:
\begin{align*}
S^{k}(V_{n}) &\to S^{k}(V_{n}), &  v_{1}\dotsb v_{k} &  \mapsto \sum_{t=1}^{k} (-1)^{\p{v_{1}}+\dotsb +\p{v_{t-1}}} v_{1} \dotsb c(v_{t}) \dotsb v_{k}; \\
S^{k}(V_{n}) \otimes S^{l}(V_{n}) &\to S^{k+l}(V_{n}),  &        v_{1}\dotsb v_{k}\otimes w_{1}\dotsb w_{l}          &\mapsto v_{1}\dotsb v_{k}w_{1}\dotsb w_{l};\\
 S^{k+l}(V_{n}) &\to S^{k}(V_{n}) \otimes S^{l}(V_{n}), &      v_{1} \dotsb v_{k+l}        & \mapsto \sum (-1)^{\varepsilon_{I,J}} v_{i_{1}}\dotsb v_{i_{k}} \otimes v_{j_{1}}\dotsb v_{j_{l}}.
\end{align*}
The last sum is over all $I=\{i_{1}< \dotsb < i_{k} \}$ and $ J=\{j_{1}< \dotsb < j_{l} \}$ for which $I \cup J = \left\{ 1, \dotsb ,k+l  \right\}$.  The element   $\varepsilon_{I,J} \in \Z_{2}$ is defined by the formula $v_{i_{1}}\dotsb v_{i_{k}}v_{j_{1}}\dotsb v_{j_{l}} = (-1)^{\varepsilon_{I,J}} v_{1}\dotsb v_{k+l}$.
\end{proposition}

\begin{proof}  The interested reader can verify that these maps satisfy the defining relations of $\qwebs$.  That the functor is essentially surjective is clear.
\end{proof}

One can verify $\beta_{\uparrow_{k}, \uparrow_{l}}$ is sent to the graded flip map $ S^{k}(V_{n}) \otimes S^{l}(V_{n}) \to S^{l}(V_{n}) \otimes S^{k}(V_{n})$ given by  $v_{1}\dotsb v_{k} \otimes w_{1}\dotsb w_{l} \mapsto (-1)^{(\p{v_{1}}+\dotsb +\p{v_{k}})(\p{w_{1}}+\dotsb +\p{w_{l}})} w_{1}\dotsb w_{l}  \otimes v_{1}\dotsb v_{k}$.  Also, we could have defined the functor $\Psi_{n}$ using the functors $\Phi_{m,n}$. %Then well-definedness of $\Psi_{n}$ amounts to the fact $\Pi_{m}$ is a well-defined functor for all $m$.

\begin{remark}\label{R:AltDefofPsi} Recall the functors $\Phi_{m,n}$ from \cref{Phi-up}. Then,
\[\Psi_{n}\left(
\xy
(0,0)*{
\begin{tikzpicture}[scale=1.25, color=\clr ]
	\draw[thick, directed=1] (1,0) to (1,1);
	\node at (1,-0.15) {\scriptsize $k$};
	\node at (1,1.15) {\scriptsize $k$};
	\draw (1,0.5) \wdot ;
\end{tikzpicture}
};
\endxy\right)=\Phi_{1,n}(\htilde_{\overline{1}}1_{(k)}),\quad\quad\Psi_{n}\left(
\xy
(0,0)*{
\begin{tikzpicture}[scale=.35, color=\clr ]
	\draw [ thick, directed=1] (0, .75) to (0,2);
	\draw [ thick, directed=.65] (1,-1) to [out=90,in=330] (0,.75);
	\draw [ thick, directed=.65] (-1,-1) to [out=90,in=210] (0,.75);
	\node at (0, 2.5) {\scriptsize $k\! +\! l$};
	\node at (-1,-1.5) {\scriptsize $k$};
	\node at (1,-1.5) {\scriptsize $l$};
\end{tikzpicture}
};
\endxy\right)=\Phi_{2,n}(\etilde_1^{(l)}1_{(k,l)}),\quad\quad\Psi_{n}\left(
\xy
(0,0)*{
\begin{tikzpicture}[scale=.35, color=\clr ]
	\draw [ thick, directed=0.65] (0,-0.5) to (0,.75);
	\draw [ thick, directed=1] (0,.75) to [out=30,in=270] (1,2.5);
	\draw [ thick, directed=1] (0,.75) to [out=150,in=270] (-1,2.5); 
	\node at (0, -1) {\scriptsize $k\! +\! l$};
	\node at (-1,3) {\scriptsize $k$};
	\node at (1,3) {\scriptsize $l$};
\end{tikzpicture}
};
\endxy\right)=\Phi_{2,n}(\ftilde_1^{(l)}1_{(k+l,0)}).
\]  In particular, for any $m,n \geq 1$ we have $\Psi_{n} \circ \Pi_{m} = \Phi_{m,n}$.  The well-definedness of $\Psi_{n}$ can also be deduced from the well-definedness of the functors $\Phi_{m,n}$ and the compatibility given in \cref{R:Compatibility}. 
\end{remark}

% To prove $\Psi_{n}$ is well-defined we must verify that the images of the relations \cref{associativity,digon-removal,dot-collision,dots-past-merges,dumbbell-relation,square-switch,square-switch-dots,double-rungs-1,double-rungs-2} hold in $\mods$.  Set $m=3$ and apply $\Pi_{3}$ to the defining relations of $\bUdot (\fqt (3))$.  This yields the defining relations of $\qwebs$.  But since $\Psi_{n}$ is defined using $\Phi_{1,n}$ and $\Phi_{2,n}$, the compatibility given in \cref{R:Compatiblity2} along with the fact that $\Phi_{3}$ is a well defined functor implies that $\Psi_{n}$ is also well defined. 

\subsection{Description of \texorpdfstring{$\End_{\qwebs}(\up_{1}^k)$}{End(upk)}}\label{description-of-all-ones}

We next describe the endomorphism algebras $\End_{\qwebs}(\up_{1}^k)$.  In particular, as promised in the discussion after \cref{L:srelations}, $\End_{\qwebs}(\up_{1}^{k})$ will turn out to be isomorphic to the Sergeev algebra.

\begin{lemma}\label{sergeev-generators}
The superalgebra $\End_{\qwebs}(\up_{1}^k)$ is generated by $c_1,\dots, c_k, s_1,\dots,s_{k-1}$.
\end{lemma}

\begin{proof}
It suffices to prove that every web diagram $w\in\End_{\qwebs}(\up^k)$ can be written as a linear combination of webs containing only upward crossings and dotted $1$-strands.  Given $w$, we may assume without loss of generality that every dot in $w$ is on a $1$-strand by \cref{dot-on-k-strand}. Next, every merge and split in $w$ can have its edges expanded into $1$-strands using \cref{digon-removal}.  For example, here is an expanded merge:
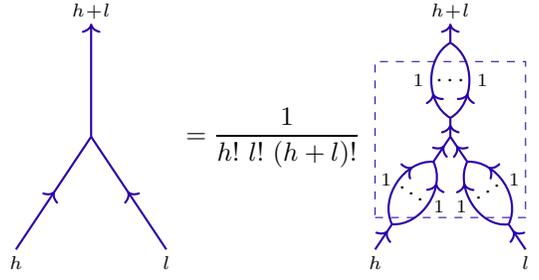
\[
\xy
(0,0)*{
\begin{tikzpicture}[scale=.5]
	\draw [color=\clr,  thick, directed=1] (-1,3) to (-1,6);
	\draw [color=\clr,  thick, rdirected=.55] (-1,3) to (-3,0);
	\draw [color=\clr,  thick, rdirected=.55] (-1,3) to (1,0);
	\node at (-1,6.35) {\scriptsize $h\! +\! l$};
	\node at (-3,-.35) {\scriptsize $h$};
	\node at (1,-.35) {\scriptsize $l$};
\end{tikzpicture}
};
\endxy
=
\frac{1}{h! \ l! \ (h+l)!}
\xy
(0,0)*{
\begin{tikzpicture}[scale=.5]
	\draw [color=\clr,  thick, directed=1] (-1,5.5) to (-1,6);
	\draw [color=\clr,  thick, directed=.65] (-1,3) to (-1,3.5);
	\draw [color=\clr,  thick, directed=.35] (-1,3.5) to [out=150,in=210] (-1,5.5);
	\draw [color=\clr,  thick, directed=.35] (-1,3.5) to [out=30,in=330] (-1,5.5);
	\draw [color=\clr,  thick, rdirected=.65] (-1,3) to (-1.45,2.333);
	\draw [color=\clr,  thick, rdirected=.55] (-2.55,.67) to (-3,0);
	\draw [color=\clr,  thick, directed=.75] (-2.55,.67) to [out=116.3,in=176.3] (-1.45,2.333);
	\draw [color=\clr,  thick, directed=.75] (-2.55,.67) to [out=-3.7,in=296.3] (-1.45,2.333);
	\draw [color=\clr,  thick, rdirected=.65] (-1,3) to (-.55,2.333);
	\draw [color=\clr,  thick, rdirected=.55] (0.55,.67) to (1,0);
	\draw [color=\clr,  thick, directed=.75] (0.55,.67) to [out=187.7,in=243.7] (-.55,2.333);
	\draw [color=\clr,  thick, directed=.75] (0.55,.67) to [out=63.7,in=363.7] (-.55,2.333);
	\draw [color=\clr, dashed] (-3,0.85) rectangle (1,5);
	\node at (-1,6.35) {\scriptsize $h\! +\! l$};
	\node at (-3,-.35) {\scriptsize $h$};
	\node at (1,-.35) {\scriptsize $l$};
	\node at (-1.85,4.5) {\scriptsize $1$};
	\node at (-0.15,4.5) {\scriptsize $1$};
	\node at (-0.95,4.5) {\small $\cdots$};
	\node at (0.05,1.55) { \rotatebox{33.7}{\small $\cdots$}};
	\node at (-0.7,1.15) {\scriptsize $1$};
	\node at (0.7,1.85) {\scriptsize $1$};
	\node at (-2,1.45) { \rotatebox{-33.7}{\small $\cdots$}};
	\node at (-1.3,1.15) {\scriptsize $1$};
	\node at (-2.7,1.85) {\scriptsize $1$};
\end{tikzpicture}
};
\endxy.
\]
By \cref{associativity}, the web enclosed by the dashed rectangle above is $(h+l)!\,Cl_{h+l}$.  By \cref{L:claspsum} each such clasp idempotent can be rewritten as a sum of upward crossings.  The result follows. 
\end{proof}

%We remark that the proof of Lemma \ref{sergeev-generators} shows that every web $w\in\End_{\qwebs}(1^k)$ is a scalar multiple of a composition of clasp idempotents and dotted $1$-strands. By Lemma \ref{sergeev-generators}, we may now define superalgebra homomorphisms $\End_{\qwebs}(1^k)\to A$ by defining them on $C_1,\dots,C_k,S_1,\dots,S_{k-1}$, and know that homomorphisms $A\to\End_{\qwebs}(1^k)$ containing these elements in their image are surjective.

\begin{proposition}\label{P:sergeev-isomorphism}  For every $k\in\Z_{>0}$ the map 
\[
\xi_k:\Ser_{k}  \to  \End_{\qwebs}(\up_{1}^k)
\] of \cref{E:Sergeevhomomorphism} is a superalgebra isomorphism.  Moreover, the map obtained by composing this homomorphism with $\Psi_{n}:\End_{\qwebs}(\up_{1}^{k}) \to \End_{\fq (n)}\left(V^{\otimes k} \right)$ coincides with the map $\psi$ given in \cref{T:sergeev-duality}.
\end{proposition}

\begin{proof} As discussed regarding \cref{E:Sergeevhomomorphism}, $\xi_{k}$ gives a well-defined homomorphism.  By \cref{sergeev-generators}, $\xi_k$ is surjective.  

We now prove injectivity.  Fix $n > k$ and let $1^k=(1,\dots,1)\in\Z^k$. Since $\s^{1^k} = V_n^{\otimes k}$, the functor $\Psi_{n}$ induces a map of superalgebras $\End_{\qwebs}(\up_{1}^k) \to \End_{\q(n)}(V_n^{\otimes k})$ which we call by the same name.  Taken together with the superalgebra map $\psi: \Ser_k \to \End_{\q(n)}(V_n^{\otimes k})$ from \cref{T:sergeev-duality} we have the following diagram of superalgebra maps.
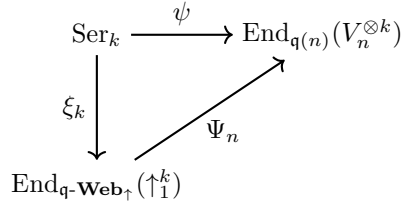
\[
\xy
(0,0)*{
\begin{tikzpicture}
\node (S) at (0,0) {$\Ser_{k}$};
\node (P) at (3,0) {$\End_{\q(n)}(V_n^{\otimes k})$};
\node (W) at (0,-2) {$\End_{\qwebs}(\up_{1}^k)$};
\path[->, thick]
(S) edge [above] node {$\psi$} (P)
(S) edge [left] node {$\xi_k$} (W)
(W) edge [below] node {\quad $\Psi_{n}$} (P);
\end{tikzpicture}
};
\endxy \ .
\]  A direct calculation on generators verifies this diagram commmutes.  The injectivity of $\xi_{k}$ follows from the fact that $\psi$ is an isomorphism by \cref{T:sergeev-duality} (since $n > k$).
\end{proof}

\subsection{The fullness of \texorpdfstring{$\Psi_{n}$}{Psin} }\label{S:fullness}

\begin{theorem}\label{T:fullness} For every $n \geq 1$, the functor $\Psi_{n}: \qwebs \to \mods$ is full.
\end{theorem}

\begin{proof}  Let $\ob{a}=(a_{1}, \dotsc , a_{r})$ and $\ob{b}=(b_{1}, \dotsc , b_{s})$ be objects of $\qwebs$.  We first observe that $\Hom_{\qwebs}(\ob{a},\ob{b})=0$ unless $|\ob{a}|=|\ob{b}|$.  Likewise, by weight considerations $\Hom_{q(n)}(S^{\ob{a}}(V_{n}), S^{\ob{b}}(V_{n}))=0$ unless $|\ob{a}|=|\ob{b}|$.  Thus we can assume $|\ob{a}|=|\ob{b}|$ in what follows.

There is a map of superspaces 
\[
\alpha= \alpha_{\ob{a}, \ob{b}}: \Hom_{\qwebs}(\ob{a}, \ob{b}) \to \Hom_{\qwebs}\left( \up_{1}^{|\ob{a}|}, \up_{1}^{|\ob{b}|}\right)
\]
given by using merges and splits to ``explode'' all the boundary strands:
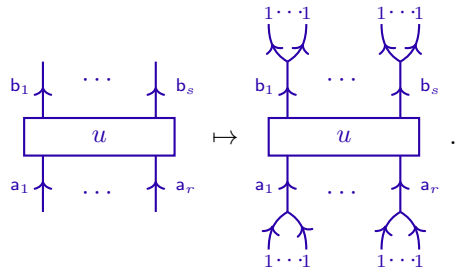
\[
\xy
(0,0)*{
\begin{tikzpicture} [scale=1, color=\clr]
	\draw[ thick] (-1.5,1.25) rectangle (0.5,.75);
	\draw[ thick, directed=0.55] (-1.25,0) to (-1.25,.75);
	\draw[ thick, directed=0.55] (-1.25,1.25) to (-1.25,2);
	\draw[ thick, directed=0.55] (.25,0) to (.25,.75);
	\draw[ thick, directed=0.55] (.25,1.25) to (.25,2);
	\node at (-0.5,1.75) {$\cdots$};
	\node at (-0.5,.25) {$\cdots$};
	\node at (-0.5,1) {$u$};
	\node at (-1.55,1.675) {\scriptsize $\ob{b}_1$};
	\node at (-1.55,0.325) {\scriptsize $\ob{a}_1$};
	\node at (0.65,1.675) {\scriptsize $\ob{b}_{s}$};
	\node at (0.65,0.325) {\scriptsize $\ob{a}_r$};
\end{tikzpicture}
};
\endxy
\mapsto
\xy
(0,0)*{
\begin{tikzpicture} [scale=1, color=\clr]
	\draw[ thick] (-1.5,1.25) rectangle (0.5,.75);
	\draw[ thick, directed=0.55] (-1.25,0) to (-1.25,.75);
	\draw[ thick, directed=0.55] (-1.25,1.25) to (-1.25,2);
	\draw[ thick, directed=0.55] (.25,0) to (.25,.75);
	\draw[ thick, directed=0.55] (.25,1.25) to (.25,2);
	\node at (-0.5,1.75) {$\cdots$};
	\node at (-0.5,.25) {$\cdots$};
	\node at (-0.5,1) {$u$};
	\node at (-1.55,1.675) {\scriptsize $\ob{b}_1$};
	\node at (-1.55,0.325) {\scriptsize $\ob{a}_1$};
	\node at (0.65,1.675) {\scriptsize $\ob{b}_{s}$};
	\node at (0.65,0.325) {\scriptsize $\ob{a}_r$};
	\draw [ thick, directed=.55] (-1,-.5) to [out=90,in=330] (-1.25,0);
	\draw [ thick, rdirected=.65] (-1.25,0) to [out=210,in=90] (-1.5,-.5);
	\draw [ thick, rdirected=.65] (0.25,0) to [out=330,in=90] (0.5,-.5);
	\draw [ thick, rdirected=.65] (0.25,0) to [out=210,in=90] (0,-.5);
	\draw [ thick, directed=.55] (-1.25,2) to [out=30,in=270] (-1,2.5);
	\draw [ thick, directed=.55] (-1.25,2) to [out=150,in=270] (-1.5,2.5);
	\draw [ thick, directed=.55] (0.25,2) to [out=30,in=270] (0.5,2.5);
	\draw [ thick, directed=.55] (0.25,2) to [out=150,in=270] (0,2.5);
	\node at (-1.25,2.65) {\small \  $\cdots$};
	\node at (-1.25,-.65) {\small  \ $\cdots$};
	\node at (0.25,2.65) {\small  \ $\cdots$};
	\node at (0.25,-.65) {\small  \ $\cdots$};
	\node at (-1,2.65) {\scriptsize $1$};
	\node at (-1.5,2.65) {\scriptsize $1$};
	\node at (-1,-.65) {\scriptsize $1$};
	\node at (-1.5,-.65) {\scriptsize $1$};
	\node at (0,2.65) {\scriptsize $1$};
	\node at (0.5,2.65) {\scriptsize $1$};
	\node at (0,-.65) {\scriptsize $1$};
	\node at (0.5,-.65) {\scriptsize $1$};
\end{tikzpicture}
};
\endxy.
\]  This map is an embedding.  Up to a scalar the left inverse is given by applying a complementary set of splits and merges to rejoin the strands and applying relation \cref{digon-removal}.  Since the map is given by composing and tensoring diagrams, via the functor there is corresponding map, $\tilde{\alpha}$, such that the following commutative diagram of superspace maps,

\begin{equation}\label{E:PsiBox}
\begin{tikzcd}
\Hom_{\qwebs}(\ob{a}, \ob{b})  \arrow[r, hook, "\alpha"] \arrow[d,  "\Psi_{n}"] & \Hom_{\qwebs}\left( \up_{1}^{|\ob{a}|}, \up_{1}^{|\ob{b}|}\right) \arrow[d,  "\Psi_{n}"] \\
\Hom_{\qwebs}(S^{\ob{a}}(V_{n}),S^{\ob{b}}(V_{n}))\arrow[r, hook, "\tilde{\alpha}"] & \Hom_{\qwebs}\left( V_{n}^{\otimes |\ob{a}|}, V_{n}^{\otimes |\ob{b}|}\right).
\end{tikzcd}
\end{equation}
Furthermore, $e_{\ob{a}}= Cl_{a_{1}} \otimes \dotsb \otimes Cl_{a_{r}}, e_{\ob{b}} = Cl_{b_{1}} \otimes \dotsb \otimes Cl_{b_{s}} \in \End_{\qwebs}(\uparrow_{1}^{|\ob{a}|})$  are idempotents, the image of $\alpha$ is precisely $e_{\ob{b}}\Hom_{\qwebs}(\uparrow^{\ob{a}}, \uparrow^{\ob{b}})e_{\ob{a}}$, and the image of $\tilde{\alpha}$ is $\Psi_{n}(e_{\ob{b}})\Hom_{q(n)}(V^{\ob{a}}, V^{\ob{b}})e_{\ob{a}}\Psi_{n}(e_{\ob{a}})$.  By \cref{P:sergeev-isomorphism} and \cref{T:sergeev-duality} the $\Psi_{n}$ on the right side of \cref{E:PsiBox} is surjective.  This along with a diagram chase implies $\Psi_{n}$ is surjective on the left side of the diagram.  Together with the discussion of the first paragraph it follows $\Psi_{n}$ is a full functor.
\end{proof}

\section{Oriented Webs}\label{S:OrientedWebs}
In this section we introduce oriented webs.  On the representation theory side this corresponds to including the duals of symmetric powers of the natural supermodule for $\fq (n)$.

\subsection{Oriented Webs}

\begin{definition}\label{D:orientedwebs}
The category $\qwebsupdown$ is the monoidal supercategory with generating objects 
\[
\{\up_{k}, \down_k \mid k \in \Z_{\geq 0}\}
\]
and generating morphisms
 
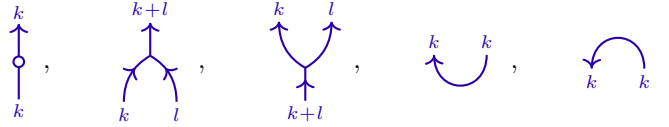
\[
\xy
(0,0)*{
\begin{tikzpicture}[color=\clr, scale=1]
	\draw[color=\clr, thick, directed=1] (1,0) to (1,1);
	\node at (1,-0.15) {\scriptsize $k$};
	\node at (1,1.15) {\scriptsize $k$};
	\draw (1,0.5) \wdot;
\end{tikzpicture}
};
\endxy \ ,\quad\quad
\xy
(0,0)*{
\begin{tikzpicture}[color=\clr, scale=.35]
	\draw [color=\clr,  thick, directed=1] (0, .75) to (0,2);
	\draw [color=\clr,  thick, directed=.65] (1,-1) to [out=90,in=330] (0,.75);
	\draw [color=\clr,  thick, directed=.65] (-1,-1) to [out=90,in=210] (0,.75);
	\node at (0, 2.5) {\scriptsize $k\! +\! l$};
	\node at (-1,-1.5) {\scriptsize $k$};
	\node at (1,-1.5) {\scriptsize $l$};
\end{tikzpicture}
};
\endxy \ ,\quad\quad
\xy
(0,0)*{
\begin{tikzpicture}[color=\clr, scale=.35]
	\draw [color=\clr,  thick, directed=.65] (0,-0.5) to (0,.75);
	\draw [color=\clr,  thick, directed=1] (0,.75) to [out=30,in=270] (1,2.5);
	\draw [color=\clr,  thick, directed=1] (0,.75) to [out=150,in=270] (-1,2.5); 
	\node at (0, -1) {\scriptsize $k\! +\! l$};
	\node at (-1,3) {\scriptsize $k$};
	\node at (1,3) {\scriptsize $l$};
\end{tikzpicture}
};
\endxy \ , \quad \quad 
\xy
(0,0)*{
\bt[scale=.35, color=\clr]
	\draw [ thick, looseness=2, directed=0.99] (1,2.5) to [out=270,in=270] (-1,2.5);
	\node at (-1,3) {\scriptsize $k$};
	\node at (1,3) {\scriptsize $k$};
\et
};
\endxy \ ,\quad\quad
\xy
(0,0)*{
\bt[scale=.35, color=\clr]
	\draw [ thick, looseness=2, directed=0.99] (1,2.5) to [out=90,in=90] (-1,2.5);
	\node at (-1,2) {\scriptsize $k$};
	\node at (1,2) {\scriptsize $k$};
\et
};
\endxy \ ,
\]
for all $k,l\in\Z_{>0}$.  We call these the \emph{dot}, \emph{merge}, \emph{split},  \emph{cup}, and \emph{cap}, respectively.  The parity is given by declaring the dot to be odd and the other generating morphisms to be even.

The relations imposed on the generators of $\qwebsupdown$ are the relations \cref{associativity,digon-removal,dot-collision,dots-past-merges,dumbbell-relation,square-switch,square-switch-dots,double-rungs-1,double-rungs-2} of $\qwebs$, declaring the morphism given in \cref{E:leftwardcrossing} is invertible, and that the equations \cref{straighten-zigzag,delete-bubble} hold.

The first relation is the following straightening rules for cups and caps:
\begin{equation}\label{straighten-zigzag}
\xy
(0,0)*{\reflectbox{\rotatebox{180}{
\bt[scale=1.25, color=\clr]
	\draw [ thick, looseness=2, ] (1,0) to [out=90,in=90] (0.5,0);
	\draw [ thick, looseness=2, ] (1.5,0) to [out=270,in=270] (1,0);
	\draw[ thick, directed=1] (0.5,0) to (0.5,-0.5);
	\draw[ thick, ] (1.5,0) to (1.5,0.5);
	\node at (0.5,-0.65) {\scriptsize \reflectbox{\rotatebox{180}{$k$}}};
	\node at (1.5,0.65) {\scriptsize \reflectbox{\rotatebox{180}{$k$}}};
\et
}}};
\endxy=
\xy
(0,0)*{
\bt[scale=1.25, color=\clr]
	\draw[thick, directed=1] (1,0) to (1,1);
	\node at (1,-0.15) {\scriptsize $k$};
	\node at (1,1.15) {\scriptsize $k$};
\et
};
\endxy \ \quad\quad\text{and}\quad\quad
\xy
(0,0)*{
\bt[scale=1.25, color=\clr]
	\draw [ thick, looseness=2, ] (1,0) to [out=90,in=90] (0.5,0);
	\draw [ thick, looseness=2, ] (1.5,0) to [out=270,in=270] (1,0);
	\draw[ thick, directed=1] (0.5,0) to (0.5,-0.5);
	\draw[ thick, ] (1.5,0) to (1.5,0.5);
	\node at (0.5,-0.65) {\scriptsize $k$};
	\node at (1.5,0.65) {\scriptsize $k$};
\et
};
\endxy=
\xy
(0,0)*{
\bt[scale=1.25,  color=\clr]
	\draw[thick, directed=1] (1,1) to (1,0);
	\node at (1,-0.15) {\scriptsize $k$};
	\node at (1,1.15) {\scriptsize $k$};
\et
};
\endxy \ .
\end{equation}

To state the second relation, we first define the \emph{leftward crossing} in $\qwebsupdown$ for all $k, l \geq 1$  by 
\begin{equation}\label{E:leftwardcrossing}
\xy
(0,0)*{
\bt[scale=1.25, color=\clr]
	\draw[thick, rdirected=0.1] (0,0) to (0.5,0.5);
	\draw[thick, directed=1] (0.5,0) to (0,0.5);
	\node at (0,-0.15) {\scriptsize $k$};
	\node at (0,0.65) {\scriptsize $l$};
	\node at (0.5,-0.15) {\scriptsize $l$};
	\node at (0.5,0.65) {\scriptsize $k$};
\et
};
\endxy := \ 
\xy
(0,0)*{
\bt[scale=1.25, color=\clr]
	\draw[thick, ] (0,0) to (0.5,0.5);
	\draw [ thick, directed=1] (0.5,0.5) to (0.5,1);
	\draw[thick, ] (0.5,0) to (0,0.5);
	\draw [ thick, looseness=2, ] (1,0) to [out=270,in=270] (0.5,0);
	\draw [thick, ] (1,1) to (1,0);
	\draw [ thick, ] (0,-0.5) to (0,0);
	\draw [ thick, looseness=2, ] (0,0.5) to [out=90,in=90] (-0.5,0.5);
	\draw [ thick, directed=1] (-0.5,0.5) to (-0.5,-0.5);
	\node at (-0.5,-0.65) {\scriptsize $k$};
	\node at (0,-0.65) {\scriptsize $l$};
	\node at (0.5,1.15) {\scriptsize $l$};
	\node at (1,1.15) {\scriptsize $k$};
\et
};
\endxy \ .
\end{equation}  We impose on $\websupdown$ the requirement that every leftward crossing is invertible. In other words, for every $k,l \geq 1$ there is another generating morphism of type $\up_{k}\down_{l} \to \down_{l}\up_{k}$ (which we draw as a rightward crossing) which is a two-sided inverse to leftward crossing under composition.  Note that dots freely move through leftward crossings and, consequently, through rightward crossings.  

Using the rightward crossing, we define the rightward cup and cap by

\begin{equation}\label{rightward-cup-and-cap}
\xy
(0,0)*{
\bt[scale=.35, color=\clr]
%	\draw [ thick, looseness=2, ] (1,2.5) to [out=270,in=270] (-1,2.5);
	\draw [ thick, looseness=2, rdirected=-0.95] (1,4.5) to [out=90,in=90] (-1,4.5);
%	\draw [ thick, directed=1.01 ] (-1,2.5) to (1,4.5);
%	\draw [ thick, rdirected=-.99 ] (1,2.5) to (-1,4.5);
	\node at (-1.75,4.75) {\scriptsize $k$};
\et
};
\endxy :=
\xy
(0,0)*{
\bt[scale=.35, color=\clr]
%	\draw [ thick, looseness=2, ] (1,2.5) to [out=270,in=270] (-1,2.5);
	\draw [ thick, looseness=2, directed=0.99] (1,4.5) to [out=90,in=90] (-1,4.5);
	\draw [ thick, directed=1 ] (-1,2.5) to (1,4.5);
	\draw [ thick, rdirected=-.95 ] (1,2.5) to (-1,4.5);
	\node at (-1.75,4.75) {\scriptsize $k$};
\et
};
\endxy \ , \quad\quad
\xy
(0,0)*{
\bt[scale=.35, color=\clr]
	\draw [ thick, looseness=2, rdirected=-0.95 ] (1,2.5) to [out=270,in=270] (-1,2.5);
%	\draw [ thick, looseness=2, rdirected=-0.95] (1,4.5) to [out=90,in=90] (-1,4.5);
%	\draw [ thick, directed=1.01 ] (-1,2.5) to (1,4.5);
%	\draw [ thick, rdirected=-.99 ] (1,2.5) to (-1,4.5);
	\node at (-1.75,2.5) {\scriptsize $k$};
\et
};
\endxy :=
\xy
(0,0)*{
\bt[scale=.35, color=\clr]
	\draw [ thick, looseness=2, directed=0.99 ] (1,2.5) to [out=270,in=270] (-1,2.5);
%	\draw [ thick, looseness=2, directed=0.99] (1,4.5) to [out=90,in=90] (-1,4.5);
	\draw [ thick, directed=1 ] (-1,2.5) to (1,4.5);
	\draw [ thick, rdirected=-0.90] (1,2.5) to (-1,4.5);
	\node at (-1.75,4.75) {\scriptsize $k$};
\et
};
\endxy \ ,
\end{equation}
 The interested reader can verify they satisfy the right-handed versions of \cref{straighten-zigzag}.  The final relation we impose on $\qwebsupdown$ is that the following holds for all $k \geq 1$:

\begin{equation}\label{delete-bubble}
\xy
(0,0)*{
\bt[scale=.35, color=\clr]
	\draw [ thick, looseness=2, ] (1,2.5) to [out=270,in=270] (-1,2.5);
	\draw [ thick, looseness=2, directed=0.95] (1,2.5) to [out=90,in=90] (-1,2.5);
%	\draw [ thick, ] (-1,2.5) to (1,4.5);
%	\draw [ thick, ] (1,2.5) to (-1,4.5);
	\draw (1,2.5) \wdot  ;
	\node at (-1.75,2.5) {\scriptsize $k$};
\et
};
\endxy=0.
\end{equation}

\end{definition}

\subsection{Additional Relations}\label{SS:OrientedAdditionalRelations}

We define the downward dot by 

\begin{equation}\label{downward-dot}
\xy
(0,0)*{
\bt[scale=1.25, color=\clr]
	\draw[thick, directed=1] (1,1) to (1,0);
	\node at (1,-0.15) {\scriptsize $k$};
	\node at (1,1.15) {\scriptsize $k$};
	\draw (1,0.5) \wdot;
\et
};
\endxy:=
\xy
(0,0)*{
\bt[scale=1.25, color=\clr]
	\draw [ thick, looseness=2, ] (1,0.25) to [out=90,in=90] (0.5,0.25);
	\draw [ thick, looseness=2, ] (1.5,-0.25) to [out=270,in=270] (1,-0.25);
	\draw [ thick, ] (1,-0.25) to (1,0.25);
	\draw[ thick, directed=1] (0.5,0.25) to (0.5,-0.75);
	\draw[ thick, ] (1.5,-0.25) to (1.5,0.75);
	\draw (1,0) \wdot;
	\node at (0.5,-0.9) {\scriptsize $k$};
	\node at (1.5,0.9) {\scriptsize $k$};
\et
};
\endxy \ .
\end{equation}
The following additional relations hold in $\qwebsupdown$.

\begin{lemma}\label{L:downward relations} For any $k \geq 1$, the following relations hold.
\begin{itemize}
\item [(a.)]
\begin{align}\label{dot-past-cup}
\xy
(0,0)*{
\bt[scale=.35, color=\clr]
	\draw [ thick, looseness=2, ] (1,2.5) to [out=270,in=270] (-1,2.5);
	\draw [ thick, directed=1] (-1,2.5) to (-1,3);
	\draw [ thick, ] (1,2.5) to (1,3);
	\node at (-1,3.5) {\scriptsize $k$};
	\node at (1,3.5) {\scriptsize $k$};
	\draw (1,2.25) \wdot  ;
\et
};
\endxy=
\xy
(0,0)*{
\bt[scale=.35, color=\clr]
	\draw [ thick, looseness=2, ] (1,2.5) to [out=270,in=270] (-1,2.5);
	\draw [ thick, directed=1] (-1,2.5) to (-1,3);
	\draw [ thick, ] (1,2.5) to (1,3);
	\node at (-1,3.5) {\scriptsize $k$};
	\node at (1,3.5) {\scriptsize $k$};
	\draw (-1,2.25) \wdot  ;
\et
};
\endxy \ ,& \quad\quad
\xy
(0,0)*{
\bt[scale=.35, color=\clr]
	\draw [ thick, looseness=2,] (1,2.5) to [out=90,in=90] (-1,2.5);
	\draw [ thick, directed=1] (-1,2.5) to (-1,1.5);
	\draw [ thick, ] (1,2.5) to (1,1.5);
	\node at (-1,1) {\scriptsize $k$};
	\node at (1,1) {\scriptsize $k$};
	\draw (1,2.75) \wdot  ;
\et
};
\endxy=
\xy
(0,0)*{
\bt[scale=.35, color=\clr]
	\draw [ thick, looseness=2,] (1,2.5) to [out=90,in=90] (-1,2.5);
	\draw [ thick, directed=1] (-1,2.5) to (-1,1.5);
	\draw [ thick, ] (1,2.5) to (1,1.5);
	\node at (-1,1) {\scriptsize $k$};
	\node at (1,1) {\scriptsize $k$};
	\draw (-1,2.75) \wdot  ;
\et
};
\endxy, \\
\xy
(0,0)*{
\bt[scale=.35, color=\clr]
	\draw [ thick, looseness=2, ] (1,2.5) to [out=270,in=270] (-1,2.5);
	\draw [ thick, ] (-1,2.5) to (-1,3);
	\draw [ thick, directed=1] (1,2.5) to (1,3);
	\node at (-1,3.5) {\scriptsize $k$};
	\node at (1,3.5) {\scriptsize $k$};
	\draw (1,2.25) \wdot  ;
\et
};
\endxy=
\xy
(0,0)*{
\bt[scale=.35, color=\clr]
	\draw [ thick, looseness=2, ] (1,2.5) to [out=270,in=270] (-1,2.5);
	\draw [ thick, ] (-1,2.5) to (-1,3);
	\draw [ thick, directed=1] (1,2.5) to (1,3);
	\node at (-1,3.5) {\scriptsize $k$};
	\node at (1,3.5) {\scriptsize $k$};
	\draw (-1,2.25) \wdot  ;
\et
};
\endxy \ ,& \quad\quad
\xy
(0,0)*{
\bt[scale=.35, color=\clr]
	\draw [ thick, looseness=2,] (1,2.5) to [out=90,in=90] (-1,2.5);
	\draw [ thick, ] (-1,2.5) to (-1,1.5);
	\draw [ thick, directed=1] (1,2.5) to (1,1.5);
	\node at (-1,1) {\scriptsize $k$};
	\node at (1,1) {\scriptsize $k$};
	\draw (1,2.75) \wdot  ;
\et
};
\endxy=
\xy
(0,0)*{
\bt[scale=.35, color=\clr]
	\draw [ thick, looseness=2,] (1,2.5) to [out=90,in=90] (-1,2.5);
	\draw [ thick, ] (-1,2.5) to (-1,1.5);
	\draw [ thick, directed=1] (1,2.5) to (1,1.5);
	\node at (-1,1) {\scriptsize $k$};
	\node at (1,1) {\scriptsize $k$};
	\draw (-1,2.75) \wdot  ;
\et
};
\endxy,
\end{align}

\item [(b.)] 
\begin{align}\label{other-bubbles}
\xy
(0,0)*{
\bt[scale=.35, color=\clr]
	\draw [ thick, looseness=2, ] (1,2.5) to [out=270,in=270] (-1,2.5);
	\draw [ thick, looseness=2, rdirected=0.95] (1,2.5) to [out=90,in=90] (-1,2.5);
%	\draw [ thick, ] (-1,2.5) to (1,4.5);
%	\draw [ thick, ] (1,2.5) to (-1,4.5);
	\draw (1,2.5) \wdot  ;
	\node at (-1.75,2.5) {\scriptsize $k$};
\et
};
\endxy
=
\xy
(0,0)*{
\bt[scale=.35, color=\clr]
	\draw [ thick, looseness=2, ] (1,2.5) to [out=270,in=270] (-1,2.5);
	\draw [ thick, looseness=2, directed=0.95] (1,2.5) to [out=90,in=90] (-1,2.5);
%	\draw [ thick, ] (-1,2.5) to (1,4.5);
%	\draw [ thick, ] (1,2.5) to (-1,4.5);
	\draw (1,2.5) \wdot  ;
	\node at (-1.75,2.5) {\scriptsize $k$};
\et
};
\endxy=0, & \quad\quad
\xy
(0,0)*{
\bt[scale=.35, color=\clr]
	\draw [ thick, looseness=2, ] (1,2.5) to [out=270,in=270] (-1,2.5);
	\draw [ thick, looseness=2, rdirected=0.95] (1,2.5) to [out=90,in=90] (-1,2.5);
%	\draw [ thick, ] (-1,2.5) to (1,4.5);
%	\draw [ thick, ] (1,2.5) to (-1,4.5);
%	\draw (1,2.5) \wdot  ;
	\node at (-1.75,2.5) {\scriptsize $k$};
\et
};
\endxy
=
\xy
(0,0)*{
\bt[scale=.35, color=\clr]
	\draw [ thick, looseness=2, ] (1,2.5) to [out=270,in=270] (-1,2.5);
	\draw [ thick, looseness=2, directed=0.95] (1,2.5) to [out=90,in=90] (-1,2.5);
%	\draw [ thick, ] (-1,2.5) to (1,4.5);
%	\draw [ thick, ] (1,2.5) to (-1,4.5);
%	\draw (1,2.5) \wdot  ;
	\node at (-1.75,2.5) {\scriptsize $k$};
\et
};
\endxy=0.
\end{align}
\end{itemize}

\end{lemma}

\begin{proof}

For the first relation of \eqref{dot-past-cup}, we compute:

\[
\xy
(0,0)*{
\bt[scale=.35, color=\clr]
	\draw [ thick, looseness=2, ] (1,2.5) to [out=270,in=270] (-1,2.5);
	\draw [ thick, directed=1] (-1,2.5) to (-1,3);
	\draw [ thick, ] (1,2.5) to (1,3);
	\node at (-1,3.5) {\scriptsize $k$};
	\node at (1,3.5) {\scriptsize $k$};
	\draw (1,2.25) \wdot  ;
\et
};
\endxy\stackrel{\eqref{downward-dot}}{=}
\xy
(0,0)*{
\bt[scale=.35, color=\clr]
	\draw [ thick, looseness=2, ] (1,2.5) to [out=270,in=270] (-1,2.5);
	\draw [ thick, directed=1] (-1,2.5) to (-1,4.5);
	\draw [ thick, ] (1,2.5) to (1,3);
	\draw [thick, looseness=2, ] (1,3) to [out=90, in=90] (3,3);
	\draw [thick, ] (3,3) to (3,2.5);
	\draw [thick, looseness=2, ] (3,2.5) to [out=270, in=270] (5,2.5);
	\draw [thick, ] (5,2.5) to (5,4.5);
	\node at (-1,5) {\scriptsize $k$};
	\node at (5,5) {\scriptsize $k$};
	\draw (3,2.75) \wdot  ;
\et
};
\endxy\stackrel{\eqref{straighten-zigzag}}{=}
\xy
(0,0)*{
\bt[scale=.35, color=\clr]
	\draw [ thick, looseness=2, ] (1,2.5) to [out=270,in=270] (-1,2.5);
	\draw [ thick, directed=1] (-1,2.5) to (-1,3);
	\draw [ thick, ] (1,2.5) to (1,3);
	\node at (-1,3.5) {\scriptsize $k$};
	\node at (1,3.5) {\scriptsize $k$};
	\draw (-1,2.25) \wdot  ;
\et
};
\endxy.
\]

Proofs of the other relations in (a.) are similar. The first relation of (b.) follows from \ref{L:braidingforups}(e) Lemma 4.5.2 (e.) and \cref{delete-bubble}, while the second follows from
\[
\xy
(0,0)*{
\bt[scale=.35, color=\clr]
	\draw [ thick, looseness=2, ] (1,2.5) to [out=270,in=270] (-1,2.5);
	\draw [ thick, looseness=2, rdirected=0.95] (1,2.5) to [out=90,in=90] (-1,2.5);
%	\draw [ thick, ] (-1,2.5) to (1,4.5);
%	\draw [ thick, ] (1,2.5) to (-1,4.5);
%	\draw (1,2.5) \wdot  ;
	\node at (-1.75,2.5) {\scriptsize $k$};
\et
};
\endxy\stackrel{\eqref{dot-collision}}{=}\left(\frac{1}{k}\right)
\xy
(0,0)*{
\bt[scale=.35, color=\clr]
	\draw [ thick, looseness=2, ] (1,2.5) to [out=270,in=270] (-1,2.5);
	\draw [ thick, looseness=2, rdirected=0.95] (1,2.5) to [out=90,in=90] (-1,2.5);
%	\draw [ thick, ] (-1,2.5) to (1,4.5);
%	\draw [ thick, ] (1,2.5) to (-1,4.5);
	\draw (0.9,2.95) \wdot;
	\draw (0.9,2) \wdot;
	\node at (-1.75,2.5) {\scriptsize $k$};
\et
};
\endxy\stackrel{\eqref{dot-past-cup}}{=}\left(\frac{1}{k}\right)
\xy
(0,0)*{
\bt[scale=.35, color=\clr]
	\draw [ thick, looseness=2, ] (1,2.5) to [out=270,in=270] (-1,2.5);
	\draw [ thick, looseness=2, rdirected=0.95] (1,2.5) to [out=90,in=90] (-1,2.5);
%	\draw [ thick, ] (-1,2.5) to (1,4.5);
%	\draw [ thick, ] (1,2.5) to (-1,4.5);
	\draw (-1,2.25) \wdot;
	\draw (0.9,2) \wdot;
	\node at (-1.75,2.5) {\scriptsize $k$};
\et
};
\endxy\stackrel{\eqref{super-interchange}}{=}\left(-\frac{1}{k}\right)
\xy
(0,0)*{
\bt[scale=.35, color=\clr]
	\draw [ thick, looseness=2, ] (1,2.5) to [out=270,in=270] (-1,2.5);
	\draw [ thick, looseness=2, rdirected=0.95] (1,2.5) to [out=90,in=90] (-1,2.5);
%	\draw [ thick, ] (-1,2.5) to (1,4.5);
%	\draw [ thick, ] (1,2.5) to (-1,4.5);
	\draw (-1,2.25) \wdot;
	\draw (0.9,2.95) \wdot;
	\node at (-1.75,2.5) {\scriptsize $k$};
\et
};
\endxy\stackrel{\eqref{dot-past-cup}}{=}\left(-\frac{1}{k}\right)
\xy
(0,0)*{
\bt[scale=.35, color=\clr]
	\draw [ thick, looseness=2, ] (1,2.5) to [out=270,in=270] (-1,2.5);
	\draw [ thick, looseness=2, rdirected=0.95] (1,2.5) to [out=90,in=90] (-1,2.5);
%	\draw [ thick, ] (-1,2.5) to (1,4.5);
%	\draw [ thick, ] (1,2.5) to (-1,4.5);
	\draw (0.9,2.95) \wdot;
	\draw (0.9,2) \wdot;
	\node at (-1.75,2.5) {\scriptsize $k$};
\et
};
\endxy\stackrel{\eqref{dot-collision}}{=}-
\xy
(0,0)*{
\bt[scale=.35, color=\clr]
	\draw [ thick, looseness=2, ] (1,2.5) to [out=270,in=270] (-1,2.5);
	\draw [ thick, looseness=2, rdirected=0.95] (1,2.5) to [out=90,in=90] (-1,2.5);
%	\draw [ thick, ] (-1,2.5) to (1,4.5);
%	\draw [ thick, ] (1,2.5) to (-1,4.5);
%	\draw (1,2.5) \wdot  ;
	\node at (-1.75,2.5) {\scriptsize $k$};
\et
};
\endxy
\]
and the fact that clockwise and counterclockwise bubbles with no dots are the same web.

\end{proof}

Using leftward cups and caps we can define downward merges, splits, and crossings as follows.
\begin{equation}\label{rainbow-merge}
\xy
(0,0)*{
\bt[scale=.35, color=\clr]
	\draw [ thick, directed=1] (0,0.75) to (0,-0.5);
	\draw [ thick, rdirected=.45] (0,.75) to [out=30,in=270] (1,2.5);
	\draw [ thick, rdirected=.45] (0,.75) to [out=150,in=270] (-1,2.5); 
	\node at (0, -1) {\scriptsize $k\! +\! l$};
	\node at (-1,3) {\scriptsize $k$};
	\node at (1,3) {\scriptsize $l$};
\et
};
\endxy \ := 
\xy
(0,0)*{
\bt[scale=.35, color=\clr]
	\draw [ thick, ] (0, .75) to (0,1.5);
	\draw [ thick, looseness=2, ] (0,1.5) to [out=90,in=90] (-2,1.5);
	\draw [ thick, directed=1] (-2, 1.5) to (-2,-4);
	\draw [ thick, directed=.65] (1,-1) to [out=90,in=330] (0,.75);
	\draw [ thick, directed=.65] (-1,-1) to [out=90,in=210] (0,.75);
	\draw [ thick, looseness=2, ] (1,-1) to [out=270,in=270] (3,-1);
	\draw [ thick, looseness=1.5, ] (-1,-1) to [out=270,in=270] (5,-1);
	\draw [ thick, ] (3,-1) to (3,3);
	\draw [ thick, ] (5,-1) to (5,3);
	\node at (-2, -4.5) {\scriptsize $k\! +\! l$};
	\node at (3,3.5) {\scriptsize $k$};
	\node at (5,3.5) {\scriptsize $l$};
\et
};
\endxy, \ \quad\quad
\xy
(0,0)*{
\bt[scale=.35, color=\clr]
	\draw [ thick, rdirected=.55] (0, .75) to (0,2);
	\draw [ thick, rdirected=0.1] (1,-1) to [out=90,in=330] (0,.75);
	\draw [ thick, rdirected=0.1] (-1,-1) to [out=90,in=210] (0,.75);
	\node at (0, 2.5) {\scriptsize $k\! +\! l$};
	\node at (-1,-1.5) {\scriptsize $k$};
	\node at (1,-1.5) {\scriptsize $l$};
\et
};
\endxy  :=
\xy
(0,0)*{
\bt[scale=.35, color=\clr]
	\draw [ thick, directed=0.75] (0,0) to (0,.75);
	\draw [ thick, looseness=2, ] (0,0) to [out=270,in=270] (2,0);
	\draw [ thick, ] (2,0) to (2,5);
	\draw [ thick, ] (0,.75) to [out=30,in=270] (1,2.5);
	\draw [ thick, ] (0,.75) to [out=150,in=270] (-1,2.5); 
	\draw [ thick, looseness=2, ] (-1,2.5) to [out=90,in=90] (-3,2.5);
	\draw [ thick, looseness=1.5, ] (1,2.5) to [out=90,in=90] (-5,2.5);
	\draw [ thick, directed=1] (-3,2.5) to (-3,-1);
	\draw [ thick, directed=1] (-5,2.5) to (-5,-1);
	\node at (2, 5.5) {\scriptsize $k\! +\! l$};
	\node at (-3,-1.5) {\scriptsize $l$};
	\node at (-5,-1.5) {\scriptsize $k$};
\et
};
\endxy,
\end{equation}
\begin{equation*}
\xy
(0,0)*{
\bt[scale=1.25, color=\clr]
	\draw[thick, rdirected=0.1] (0,0) to (0.5,0.5);
	\draw[thick, rdirected=0.1] (0.5,0) to (0,0.5);
	\node at (0,-0.15) {\scriptsize $k$};
	\node at (0,0.65) {\scriptsize $l$};
	\node at (0.5,-0.15) {\scriptsize $l$};
	\node at (0.5,0.65) {\scriptsize $k$};
\et
};
\endxy := \ 
\xy
(0,0)*{
\bt[scale=1.25, color=\clr]
	\draw[thick, ] (0,0) to (0.5,0.5);
	\draw[thick, ] (0.5,0) to (0,0.5);
	\draw [ thick, looseness=2, ] (1,0) to [out=270,in=270] (0.5,0);
	\draw [ thick, looseness=1.5, ] (1.5,0) to [out=270,in=270] (0,0);
	\draw [thick, ] (1,1.25) to (1,0);
	\draw [thick, ] (1.5,1.25) to (1.5,0);
	\draw [ thick, looseness=2, ] (0,0.5) to [out=90,in=90] (-0.5,0.5);
	\draw [ thick, looseness=1.5, ] (0.5,0.5) to [out=90,in=90] (-1,0.5);
	\draw [ thick, directed=1] (-0.5,0.5) to (-0.5,-0.75);
	\draw [ thick, directed=1] (-1,0.5) to (-1,-0.75);
	\node at (-1,-0.9) {\scriptsize $k$};
	\node at (-0.5,-0.9) {\scriptsize $l$};
	\node at (1,1.4) {\scriptsize $l$};
	\node at (1.5,1.4) {\scriptsize $k$};
\et
};
\endxy.
\end{equation*}

By applying cups and caps to  \cref{associativity,digon-removal,dot-collision,dots-past-merges,dumbbell-relation,square-switch,square-switch-dots,double-rungs-1,double-rungs-2} one obtains a parallel set of relations on downward oriented diagrams.  We freely use these downward relations in what follows. When computing these the reader is advised to keep in mind the effect of the super interchange law when diagrams have multiple dots.  For example, \cref{dot-collision} becomes

\begin{equation}\label{reverse-dot-collision}
\xy
(0,0)*{
\begin{tikzpicture}[scale=.3, color=\clr] 
	\draw [ thick, rdirected=.05] (1,-2.75) to (1,2.5);
	\node at (1,3) {\scriptsize $k$};
	\node at (1,-3.25) {\scriptsize $k$};
	\draw (1,0.5) \wdot ;
	\draw (1,-0.5) \wdot ;
\end{tikzpicture}
};
\endxy=(-k)
\xy
(0,0)*{
\begin{tikzpicture}[scale=.3, color=\clr] 
	\draw [ thick, rdirected=.05] (1,-2.75) to (1,2.5);
	\node at (1,3) {\scriptsize $k$};
	\node at (1,-3.25) {\scriptsize $k$};
\end{tikzpicture}
};
\endxy.
\end{equation}

\subsection{The braiding on \texorpdfstring{$\qwebsupdown$}{Oriented Webs}}\label{SS:BraidingonOrientedWebs} Set the notation $\uddoublearrow$ for when an arrow has indeterminate orientation; for example, $\uddoublearrow_{k}$ could be either $\up_{k}$ or $\down_{k}$. More generally, for any tuple of nonnegative integers $\ob{a}=(a_{1}, \dotsc , a_{r})$ we write $\uddoublearrow_{\ob{a}} = \uddoublearrow_{a_{1}}\uddoublearrow_{a_{2}}\dotsb \uddoublearrow_{a_{r}}$. Recall from \cref{D:Upbraiding} the crossing of two upward strands of arbitrary thickness.  In \cref{E:leftwardcrossing} the leftward crossing was defined as was its inverse, the rightward crossing.  The downward crossing is defined in \cref{rainbow-merge}.  A straightforward set of calculations using the relations established so far verifies the relations given in \cref{L:braidingforups} hold for all possible strand orientations and all possible $h,k,l \in \Z_{\geq 1}$.  Direct calculations also yield the following lemma.
\begin{lemma}\label{} The following relations hold in $\qwebsupdown$ for all possible strand orientations and all possible $h,k,l \in \Z_{\geq 1}$.
\begin{itemize}
\item [(a.)] 
\[
\xy
(0,0)*{
\bt[color=\clr, scale=.35]
	\draw [ thick, looseness=2, ] (1,0) to [out=270,in=270] (-1,0);
	\draw [ thick, ] (1,0) to (3,2);
	\draw [ thick, ] (-1,0) to (1,2);
	\draw [ thick, ] (2,-1.5) to (2,0) to (-1,2);
	\node at (1,2.5) {\scriptsize $k$};
	\node at (3,2.5) {\scriptsize $k$};
	\node at (2,-2) {\scriptsize $l$};
	\node at (-1,2.5) {\scriptsize $l$};
\et
};
\endxy=
\xy
(0,0)*{
\bt[color=\clr, scale=.35]
	\draw [ thick, looseness=2, ] (1,0) to [out=270,in=270] (-1,0);
	\draw [ thick, ] (-2,-3.5) to (-2,0);
	\node at (-1,0.5) {\scriptsize $k$};
	\node at (1,0.5) {\scriptsize $k$};
	\node at (-2,-4) {\scriptsize $l$};
	\node at (-2,0.5) {\scriptsize $l$};
\et
};
\endxy \ ,\quad\quad
\xy
(0,0)*{\reflectbox{
\bt[color=\clr, scale=.35]
	\draw [ thick, looseness=2, ] (1,0) to [out=270,in=270] (-1,0);
	\draw [ thick, ] (1,0) to (3,2);
	\draw [ thick, ] (-1,0) to (1,2);
	\draw [ thick, ] (2,-1.5) to (2,0) to (-1,2);
	\node at (1,2.5) {\reflectbox{\scriptsize $k$}};
	\node at (3,2.5) {\reflectbox{\scriptsize $k$}};
	\node at (2,-2) {\reflectbox{\scriptsize $l$}};
	\node at (-1,2.5) {\reflectbox{\scriptsize $l$}};
\et
}};
\endxy=
\xy
(0,0)*{\reflectbox{
\bt[color=\clr, scale=.35]
	\draw [ thick, looseness=2, ] (1,0) to [out=270,in=270] (-1,0);
	\draw [ thick, ] (-2,-3.5) to (-2,0);
	\node at (-1,0.5) {\reflectbox{\scriptsize $k$}};
	\node at (1,0.5) {\reflectbox{\scriptsize $k$}};
	\node at (-2,-4) {\reflectbox{\scriptsize $l$}};
	\node at (-2,0.5) {\reflectbox{\scriptsize $l$}};
\et
}};
\endxy \ ,
\]
\[
\xy
(0,0)*{
\bt[color=\clr, scale=.35]
	\draw [ thick, looseness=2, ] (1,0) to [out=90,in=90] (-1,0);
	\draw [ thick, ] (2,0) to (2,3.5);
	\node at (-1,-0.5) {\scriptsize $k$};
	\node at (1,-0.5) {\scriptsize $k$};
	\node at (2,-0.5) {\scriptsize $l$};
	\node at (2,4) {\scriptsize $l$};
\et
};
\endxy=
\xy
(0,0)*{
\bt[color=\clr, scale=.35]
	\draw [ thick, looseness=2, ] (1,0) to [out=90,in=90] (-1,0);
	\draw [ thick, ] (-1,0) to (-3,-2) ;
	\draw [ thick, ] (-1,-2) to (1,0);
	\draw [ thick, ] (1,-2) to (-2,0) to (-2,1.5);
	\node at (-3,-2.5) {\scriptsize $k$};
	\node at (-1,-2.5) {\scriptsize $k$};
	\node at (1,-2.5) {\scriptsize $l$};
	\node at (-2,2) {\scriptsize $l$};
\et
};
\endxy \ ,\quad\quad
\xy
(0,0)*{\reflectbox{
\bt[color=\clr, scale=.35]
	\draw [ thick, looseness=2, ] (1,0) to [out=90,in=90] (-1,0);
	\draw [ thick, ] (2,0) to (2,3.5);
	\node at (-1,-0.5) {\reflectbox{\scriptsize $k$}};
	\node at (1,-0.5) {\reflectbox{\scriptsize $k$}};
	\node at (2,-0.5) {\reflectbox{\scriptsize $l$}};
	\node at (2,4) {\reflectbox{\scriptsize $l$}};
\et
}};
\endxy=
\xy
(0,0)*{\reflectbox{
\bt[color=\clr, scale=.35]
	\draw [ thick, looseness=2, ] (1,0) to [out=90,in=90] (-1,0);
	\draw [ thick, ] (-1,0) to (-3,-2) ;
	\draw [ thick, ] (-1,-2) to (1,0);
	\draw [ thick, ] (1,-2) to (-2,0) to (-2,1.5);
	\node at (-3,-2.5) {\reflectbox{\scriptsize $k$}};
	\node at (-1,-2.5) {\reflectbox{\scriptsize $k$}};
	\node at (1,-2.5) {\reflectbox{\scriptsize $l$}};
	\node at (-2,2) {\reflectbox{\scriptsize $l$}};
\et
}};
\endxy .
\]
\end{itemize}
\end{lemma}

Furthermore, for arbitrary objects $\ob{a}, \ob{b}$ in $\qwebsupdown$, one can use the appropriately oriented version of \cref{D:generalupbraiding} to define the braiding isomorphism $\beta_{\uddoublearrow_{\ob{a}}, \uddoublearrow_{b}} : \uddoublearrow_{\ob{a}} \otimes \uddoublearrow_{\ob{b}} \to \uddoublearrow_{\ob{b}} \otimes \uddoublearrow_{\ob{a}}$.  These morphisms make $\qwebsupdown$ into a symmetric braided monoidal supercategory.  

These relations also imply that the right and left duals of morphisms coincide. That is, if one instead uses rightward cups and caps to define rightward or downward crossings, downward dots, etc., then these coincide with the morphisms defined here via leftward cups and caps.  Calculations also show that the rightward crossing defined using the upward crossing and rightward cups and caps is the inverse of the leftward crossing and, hence, equals the rightward crossing.  Finally, one can verify strands of arbitrary thickness and orientation satisfy the Reidemeister moves.

\subsection{Isomorphisms between morphism spaces in \texorpdfstring{$\qwebsupdown $}{Oriented Webs} }\label{SS:Useful isomorphisms}

% In particular, there are unique shortest length $\sigma^{L}_{\gamma}, \sigma^{R}_{\gamma} \in \Sigma_{t}$ such that $\gamma \cdot \sigma^{L}_{\gamma}$ has all up arrows to the left of all down arrows, and  $\gamma \cdot \sigma^{R}_{\gamma}$ has all up arrows to the right of all down arrows. 
We now introduce isomorphisms between various morphism spaces which will be useful in what follows.  The symmetric group on $t$ letters, $\Sigma_{t}$, acts by place permutation on the set of all objects of $\qwebsupdown$ which are the tensor product of exactly $t$ generating objects. For example, if $\ob{a} = \uddoublearrow_{a_{1}} \uddoublearrow_{a_{2}} \dotsb \uddoublearrow_{a_{t}}$ and $\sigma \in \Sigma_{t}$, then $\sigma \cdot \ob{a}  = \uddoublearrow_{a_{\sigma(1)}} \uddoublearrow_{a_{\sigma(2)}} \dotsb \uddoublearrow_{a_{\sigma(t)}}$.  Moreover, for any object$\ob{a}$ and $\sigma \in \Sigma_{t}$ the braiding morphisms in $\qwebsupdown$ give a morphism,
\begin{equation*}
d_{\sigma} = d_{\sigma, \ob{a}} \in \Hom_{\qwebsupdown}(\ob{a}, \sigma \cdot \ob{a}),
\end{equation*}
which is an invertible morphism in $\qwebsupdown$ with $d_{\sigma}^{-1}= d_{\sigma^{-1}}$.  More generally, given objects $\ob{a}$ and $\ob{b}$ which are the tensor product of $t$ and $u$ generators, respectively, and any $\sigma \in \Sigma_{t}$ and $\tau \in \Sigma_{u}$, then the map $ w \mapsto d_{\tau} \circ w  \circ d_{\sigma^{-1}}$ defines a superspace isomorphism 
\begin{equation}\label{E:twistingstrandisomorphism}
\Hom_{\qwebsupdown}\left(\ob{a}, \ob{b} \right) \xrightarrow{\cong} \Hom_{\qwebsupdown}(\sigma \cdot \ob{a}, \tau \cdot \ob{b}).
\end{equation} In addition, given objects of the form $\ob{a} = \up_{a_{1}}\dotsb \up_{a_{r}}\down_{a_{r+1}}\dotsb \down_{a_{t}}$ and $\ob{b} = \down_{b_{1}}\dotsb \down_{b_{s}}\up_{b_{s+1}}\dotsb \up_{b_{u}}$, then there is a superspace map,
\begin{equation}\label{E:cupcapisomorphism}
\Hom_{\qwebsupdown}(\ob{a}, \ob{b}) \xrightarrow{\cong} \Hom_{\qwebsupdown}\left( \up_{b_{s}}\dotsb \up_{b_{1}}\up_{a_{1}}\dotsb \up_{a_{r}}, \up_{b_{s+1}}\dotsb \up_{b_{u}}\up_{a_{t}}\dotsb \up_{a_{r+1}} \right),
\end{equation} given on diagrams by
\[
\xy
(0,0)*{
\begin{tikzpicture}[scale=.3, color=\clr]
	\draw [ thick] (-2.3,-8) rectangle (4.3,-6);
	\node at (1,-7) {$w$};
	\draw [ thick, directed=0.75] (-2,-9.5) to (-2,-8);
	\draw [ thick, directed=0.75] (0,-9.5) to (0,-8);
	\draw [ thick, rdirected=0.75] (2,-9.5) to (2,-8);
	\draw [ thick, rdirected=0.75] (4,-9.5) to (4,-8);
	\node at (2,-10.25) {\scriptsize $a_{r+1}$};
	\node at (4,-10.25) {\scriptsize $a_t$};
	\node at (-2,-10.25) {\scriptsize $a_1$};
	\node at (-1,-9.125) {\,$\cdots$};
	\node at (3,-9.125) {\,$\cdots$};
	\node at (0,-10.25) {\scriptsize $a_r$};
	\draw [ thick, directed=0.75] (2,-6) to (2,-4.5);
	\draw [ thick, directed=0.75] (4,-6) to (4,-4.5);
	\draw [ thick, rdirected=0.75] (-2,-6) to (-2,-4.5);
	\draw [ thick, rdirected=0.75] (0,-6) to (0,-4.5);
	\node at (2,-3.75) {\scriptsize $b_{s+1}$};
	\node at (3,-5.625) {\,$\cdots$};
	\node at (-1,-5.625) {\,$\cdots$};
	\node at (4,-3.75) {\scriptsize $b_u$};
	\node at (-2,-3.75) {\scriptsize $b_1$};
	\node at (0,-3.75) {\scriptsize $b_s$};
\end{tikzpicture}
};
\endxy\mapsto
\xy
(0,0)*{
\begin{tikzpicture}[scale=.3, color=\clr]
	\draw [ thick] (-2.3,-8) rectangle (4.3,-6);
	\draw [ thick, looseness=2, directed=0.5] (-4,-6) to [out=90,in=90] (-2,-6);
	\draw [ thick, looseness=1.75, directed=0.5] (-6,-6) to [out=90,in=90] (0,-6);
	\node at (1,-7) {$w$};
	\draw [ thick, looseness=2, rdirected=0.5] (6,-8) to [out=270,in=270] (4,-8);
	\draw [ thick, looseness=1.75, rdirected=0.5] (8,-8) to [out=270,in=270] (2,-8);
	\draw [ thick, directed=0.4] (-6,-11) to (-6,-6);
	\draw [ thick, directed=0.4] (-4,-11) to (-4,-6);
	\draw [ thick, directed=0.65] (-2,-11) to (-2,-8);
	\draw [ thick, directed=0.65] (0,-11) to (0,-8);
	\node at (-6,-11.75) {\scriptsize $b_s$};
	\node at (-5,-10) {\,$\cdots$};
	\node at (-5,-5.5) { \ $\cdots$};
	\node at (-4,-11.75) {\scriptsize $b_1$};
	\node at (-2,-11.75) {\scriptsize $a_1$};
	\node at (-1,-10) {\,$\cdots$};
	\node at (-1,-5.5) {\,$\cdots$};
	\node at (0,-11.75) {\scriptsize $a_r$};
	\draw [ thick, directed=0.5] (2,-6) to (2,-3);
	\draw [ thick, directed=0.5] (4,-6) to (4,-3);
	\draw [ thick, directed=0.7] (6,-8) to (6,-3);
	\draw [ thick, directed=0.7] (8,-8) to (8,-3);
	\node at (2,-2.25) {\scriptsize $b_{s+1}$};
	\node at (3,-4) {\,$\cdots$};
	\node at (3,-8.5) {\,$\cdots$};
	\node at (4,-2.25) {\scriptsize $b_u$};
	\node at (6,-2.25) {\scriptsize $a_{t}$};
	\node at (7,-4) {\,$\cdots$};
	\node at (7,-8.5) {\,$\cdots$};
	\node at (8,-2.25) {\scriptsize $a_{r+1}$};
\end{tikzpicture}
};
\endxy.
\]
The inverse is given by a similar map (thanks to relation \cref{straighten-zigzag}).

As in the proof of \cref{T:fullness}, whenever a morphism between $\Hom$-spaces in $\qwebsupdown$ is defined by applying some combination of compositions and tensor products of morphisms we can apply the functor $\Psi$ and obtain a corresponding morphism in the category of $\fq (n)$-modules.  In particular, commuting diagrams go to commuting diagrams and isomorphisms go to isomorphisms.

\subsection{Further Functors }\label{SS:FurtherFunctors}

\begin{theorem}\label{T:Thetafunctor}  There is a fully faithful functor of symmetric monoidal supercategories 
\[
\qwebs \to \qwebsupdown 
\] which takes objects and morphisms in $\qwebs$ to objects and morphisms of the same name in $\qwebsupdown$.

\end{theorem}
\begin{proof}  By construction the defining relations on morphisms of $\qwebs$ hold in $\qwebsupdown$.  Consequently, there is a well-defined functor of monoidal supercategories.  As discussed in \cref{SS:BraidingonOrientedWebs}, the relations listed in \cref{L:braidingforups} hold for all possible orientations of the strands in a diagram.  Now if $\ob{a}, \ob{b}$ are objects of $\qwebs$ and if $d$ is a web diagram which lies in $\Hom_{\qwebsupdown}(\ob{a}, \ob{b})$, then considering the end behavior of strands we see that we can use these relations, planar isotopies, and \cref{straighten-zigzag} to eliminate all cups, caps, and rightward crossings.  Furthermore, by \cref{delete-bubble} and \cref{L:downward relations} there are no bubbles in $d$.   Taken together, this implies $d$ is equal to a web diagram which lies in $\Hom_{\qwebs}(\ob{a}, \ob{b})$.  Fullness and faithfulness follows.
\end{proof}

We will identify $\qwebs$ as a full subcategory of $\qwebsupdown$ via this functor.  Let $\modsupdown$ denote the full monoidal subsupercategory of $\fq (n)$-supermodules generated by the objects 
\[
\left\{S^{p}\left( V_{n}\right), S^{p} \left(V_{n}  \right)^{*} \mid p\geq 1 \right\}.
\]
  We next show the functor $\Psi_n : \qwebs \to \mods$ can be extended to $\Psi_n: \qwebsupdown\to \modsupdown$.  Set the notation 
\begin{align}\label{}
\ev_{k}: S^{k}(V_{n})^{*}\otimes S^{k}(V_{n}) \to \C,\\
\coev_{k}: \C \to S^{k}(V_{n})\otimes S^{k}(V_{n})^{*},
\end{align} for the evaluation and coevaluation maps defined in \cref{SS: monoidal supercats}.

\begin{theorem}\label{T:Psiupdown} There is an essentially surjective, full functor of symmetric monoidal supercategories
\[
\Psi_n : \qwebsupdown \to \modsupdown
\] given on generating objects by 
\begin{align*}
\Psi_n \left(\up_{k} \right) &= S^{k}(V_{n}),\\
 \Psi_n \left(\down_{k} \right) &= S^{k}(V_{n})^{*}
\end{align*}
and on generating morphisms by defining it on the dot, merge, and split as in \cref{psi-functor} and
\iffalse
\[
\Psi_{n}\left(
\xy
(0,0)*{
\begin{tikzpicture}[scale=1.25, color=\clr ]
	\draw[thick, directed=1] (1,0) to (1,1);
	\node at (1,-0.15) {\scriptsize $k$};
	\node at (1,1.15) {\scriptsize $k$};
	\draw (1,0.5) \wdot ;
\end{tikzpicture}
};
\endxy\right)=\Phi_{1}(\htilde_{\overline{1}}1_{(k)}),\quad\quad\Psi_{n}\left(
\xy
(0,0)*{
\begin{tikzpicture}[scale=.35, color=\clr ]
	\draw [ thick, directed=1] (0, .75) to (0,2);
	\draw [ thick, directed=.65] (1,-1) to [out=90,in=330] (0,.75);
	\draw [ thick, directed=.65] (-1,-1) to [out=90,in=210] (0,.75);
	\node at (0, 2.5) {\scriptsize $k\! +\! l$};
	\node at (-1,-1.5) {\scriptsize $k$};
	\node at (1,-1.5) {\scriptsize $l$};
\end{tikzpicture}
};
\endxy\right)=\Phi_{2}(\etilde_1^{(l)}1_{(k,l)}),\quad\quad\Psi_{n}\left(
\xy
(0,0)*{
\begin{tikzpicture}[scale=.35, color=\clr ]
	\draw [ thick, directed=0.65] (0,-0.5) to (0,.75);
	\draw [ thick, directed=1] (0,.75) to [out=30,in=270] (1,2.5);
	\draw [ thick, directed=1] (0,.75) to [out=150,in=270] (-1,2.5); 
	\node at (0, -1) {\scriptsize $k\! +\! l$};
	\node at (-1,3) {\scriptsize $k$};
	\node at (1,3) {\scriptsize $l$};
\end{tikzpicture}
};
\endxy\right)=\Phi_{2}(\ftilde_1^{(l)}1_{(k+l,0)}),
\]
\fi
\[
\Psi_n\left( 
\xy
(0,0)*{
\bt[scale=.35, color=\clr]
	\draw [ thick, looseness=2, directed=0.99] (1,2.5) to [out=270,in=270] (-1,2.5);
	\node at (-1,3) {\scriptsize $k$};
	\node at (1,3) {\scriptsize $k$};
\et
};
\endxy \right)=\coev_k,\quad\quad
\Psi_n\left(
\xy
(0,0)*{
\bt[scale=.35, color=\clr]
	\draw [ thick, looseness=2, directed=0.99] (1,2.5) to [out=90,in=90] (-1,2.5);
	\node at (-1,2) {\scriptsize $k$};
	\node at (1,2) {\scriptsize $k$};
\et
};
\endxy
\right)=\ev_k.
\iffalse
\quad\quad
\Psi_n\left(\xy
(0,0)*{
\bt[color=\clr, scale=1.25]
	\draw[thick, directed=1] (0,0) to (0.5,0.5);
	\draw[thick, rdirected=0.1] (0.5,0) to (0,0.5);
	\node at (0,-0.15) {\scriptsize $k$};
	\node at (0,0.65) {\scriptsize $l$};
	\node at (0.5,-0.15) {\scriptsize $l$};
	\node at (0.5,0.65) {\scriptsize $k$};
\et
};
\endxy\right)=X_{k,l*}
\fi
\]
%where $X_{k,l*}\colon\s^k(V_n)\otimes\s^l(V_n)^*\to\s^l(V_n)^*\otimes\s^k(V_n)$ comes from the symmetry on $\modsupdown$.
\end{theorem}

\begin{proof} Since $\Psi_{n}$ is already known to preserve \cref{associativity,digon-removal,dot-collision,dots-past-merges,dumbbell-relation,square-switch,square-switch-dots,double-rungs-1,double-rungs-2} of $\qwebs$, it suffices to verify relations \cref{straighten-zigzag,delete-bubble} and the invertibility of \cref{E:leftwardcrossing}.  This follows from a direct calculation.
\end{proof}

\begin{remark}\label{R:Upsilonfunctor} The oriented Brauer-Clifford supercategory $\OBC$ was introduced in \cite{CK}.   It is given as a monoidal supercategories with generating objects $\up$ and $\down$, three even generating morphisms $\lcup:\unit\to\up\down$, $\lcap: \down\up\to\unit$, $\swap: \up\up\to\up\up$, and one odd generating morphism $\cldot:\up\to\up$. These morphisms are subject to certain relations which we omit.  There is a fully faithful functor of symmetric monoidal supercategories 
\[
\Upsilon: \OBC \to \qwebsupdown. 
\] On objects one has $\Upsilon(\up )= \up_{1}$ and $\Upsilon(\down ) = \down_{1}$.  The functor $\Upsilon$ sends the generating morphisms of $\OBC$ to the similarly drawn morphisms of $\qwebsupdown$ which have all edges labelled by $1$.  

As explained in \cite{CK}, there is a functor of monoidal supercategories 
\[
\Omega_{n}: \OBC \to  \fq (n)\text{-modules}.
\] Furthermore, $\Omega_{n} = \Psi_{n} \circ \Upsilon$.
\end{remark}

The functor $\Upsilon$ is expected to induce an equivalence between the Karoubi envelopes of $\OBC$ and $\qwebsupdown$ and, hence, they will have the same decategorifications.  It would be interesting to determine their decategorifications.

\section{Main Theorems}

\subsection{Equivalences of Categories}\label{SS:EquivalencesofCategories}
For $n \geq 1$, set $k=(n+1)(n+2)/2$ and recall the quasi-idempotent 
\[
e_{\lambda(n)} \in\Ser_{k} \cong \End_{\qwebs}\left(\up_{1}^{k} \right)= \End_{\qwebsupdown}\left(\up_{1}^{k} \right)
\]
defined in \cref{SS:SergeevAlgebra}.
 
\begin{definition}\label{D:qnWebs} Define $\webs$ (resp.\ $\qwebsupdown$) to be the monoidal supercategory given by the same generators and relations as $\qwebs$ (resp.\ $\qwebsupdown$) along with the relation $e_{\lambda(n)}=0$.
\end{definition}

We are now prepared to state and prove the main theorem of the paper.  

\begin{theorem}\label{T:maintheorem} The functor $\Psi_{n}: \qwebsupdown \to \modsupdown$ induces functors 
\begin{align*}
\Psi_{n}: & \webs \to \mods, \\
\Psi_{n}: & \websupdown \to \modsupdown.
\end{align*}
These functors are equivalences of symmetric monoidal supercategories.
\end{theorem}

\begin{proof}  Since in both cases $\Psi_{n}$ is a functor of monoidal categories, and $\Psi_{n}(e_{\lambda(n)})=0$ by \cref{P:sergeev-isomorphism,sergeev-kernel}, it follows from the discussion in \cref{SS:MonoidalSupercatsandIdeals} that $\Psi_{n}$ induces a functor of monoidal categories which we call by the same name.  By construction it is essentially surjective.  It remains to show that it is full and faithful in both cases.  By the discussion in \cref{SS:Useful isomorphisms} we may assume without loss of generality that both $\lambda$ and $\mu$ consist of only up arrows.  Combining \cref{T:Thetafunctor,T:fullness} it follows that $\Psi_{n}$ is full and, without loss of generality, that we are in the case of $\websupdown$.

Let $\mathcal{I}$ be the tensor ideal of $\qwebsupdown$ generated by $e_{\lambda(n)}$.  Then $\websupdown$ is canonically isomorphic to $\qwebsupdown /\mathcal{I}$.  To show faithfulness amounts to showing that if $f \in \Hom_{\qwebsupdown}(\lambda, \mu)$ and $\Psi_{n}(f)=0$, then $f \in \mathcal{I}(\lambda, \mu)$. Just as in the proof of \cref{T:fullness}, we may assume $|\lambda| = |\mu|$ since otherwise $\Psi_{n}$ is trivially faithful. Set $k=|\lambda|=|\mu |$.  In what follows we identify $\Ser_{k}$ and $\Hom_{\websupdown  }\left( \up_{1}^{k}, \up_{1}^{k}\right)$ via the isomorphism given in \cref{sergeev-kernel}.

As in the proof of \cref{T:fullness}, we can compose with merges and splits and induce the following commutative diagram of morphisms:

\begin{equation}\label{E:PsiBox2}
\begin{tikzcd}
\Hom_{\qwebsupdown   }(\lambda, \mu )  \arrow[r, hook, "\alpha"] \arrow[d,  "\Psi_{n}"] & \Hom_{\qwebsupdown   }\left( \up_{1}^{k}, \up_{1}^{k}\right) \arrow[d,  "\Psi_{n}"]   \\
\Hom_{\fq (n)}(S^{\lambda}(V_{n}),S^{\mu}(V_{n}))\arrow[r, hook, "\tilde{\alpha}"] & \Hom_{\fq (n)}\left( V_{n}^{\otimes k}, V_{n}^{\otimes k}\right).
\end{tikzcd}
\end{equation} 

Consequently, for $f \in \Hom_{\qwebsupdown   }(\lambda, \mu ) $ we have $\Psi_{n}(f) =0$ if and only if $\Psi_{n}(\alpha(f)) =0$.  We claim $\alpha(f) \in \mathcal{I}\left(\up_{1}^{k}, \up_{1}^{k} \right)$.    It follows from \cref{P:sergeev-isomorphism} that $\alpha(f) = \sum_{i=1}^{t} g_{i}e_{\gamma_{i}}h_{i}$ for $g_{i},h_{i} \in \End_{\qwebsupdown}(\uparrow_{1}^{k})$, where $e_{\gamma_{i}}$ are Sergeev quasi-idempotents labeled by strict partitions of $k$ with $\ell(\gamma_{i}) > n$. 

We next claim if $\gamma$ is a strict partition of $k$ with $\ell(\gamma) > n$, then $e_{\gamma} \in \mathcal{I}\left(\uparrow_{1}^{k} , \uparrow_{1}^{k} \right)$.  Since $\mathcal{I}$ is closed under composition and linear combinations, it will follow that $\alpha(f) \in \mathcal{I}\left(\uparrow_{1}^{k} , \uparrow_{1}^{k} \right)$.  Let $m >0$ be fixed and large enough so that $\psi: \Ser_{k} \to \End_{\fq (m)}(V_{m}^{\otimes k})$ is an isomorphism.  Furthermore, recall this map is compatible with the isomorphism given in \cref{P:sergeev-isomorphism}.  We identify the Sergeev algebra and these endomorphism algebras via these isomorphisms.  

In showing the claim we make use of the fact that, through Schur-Weyl-Sergeev duality, idempotents of $\Ser_{r}$ correspond to projection onto direct summands of $V_{m}^{\otimes r}$ as a $\fq (m)$-module.  In particular, for any strict partition of $r$, $\mu$, the Sergeev quasi-idempotent $e_{\mu} \in \Ser_{r}$ projects onto a summand of the $L_{m}(\mu)$-isotopy component of $V_{m}^{\otimes r}$.  By \cref{sum-multiplicity}, since $\gamma$ is a strict partition of $k$ with $\ell(\gamma) > n$, $L_{m}(\gamma)$ is a direct summand of $L_{m}(\lambda(n)) \otimes V_{m}^{\otimes (k-|\lambda(n)|)}$ in $V_{m}^{\otimes k}$.  That is, there are elements $a,b \in \Ser_{k}$ so that $x_{\gamma}:=a  \left(  e_{\lambda(n)} \otimes \Id_{V_{m}}^{\otimes (k-|\lambda(n)|)} \right)  b$ is a nonzero idempotent of $\Ser_{k}$ which projects onto a summand of $V_{m}^{\otimes k}$ isomorphic to $L_{m}(\gamma)$ and, hence, lies in the direct summand of \cref{E:ArtinWedderburn} labelled by $\gamma$.  However, since $e_{\lambda(n)}$ is an element of the tensor ideal $\mathcal{I}$, it follows that $x_{\gamma} \in \mathcal{I}(\uparrow_{1}^{k}, \uparrow_{1}^{k})$.  Furthermore, since the direct summand of \cref{E:ArtinWedderburn} labelled by $\gamma$ is a simple ideal of $\End_{\qwebsupdown}\left(\uparrow_{1}^{k} \right) = \Ser_{k}$, this in turn implies that the entire summand lies in $\mathcal{I}(\uparrow_{1}^{k}, \uparrow_{1}^{k})$.  In particular, $e_{\gamma}$ lies in $\mathcal{I}(\uparrow_{1}^{k}, \uparrow_{1}^{k})$ as originally claimed and, hence, $\alpha(f) \in \mathcal{I}(\uparrow_{1}^{k}, \uparrow_{1}^{k})$. 

Finally, recall that $\alpha$ has a left inverse given by composing by merges and splits.  Since $\mathcal{I}$ is a tensor ideal, $\alpha^{-1}$ takes elements of $\mathcal{I}(\uparrow_{1}^{k}, \uparrow_{1}^{k})$ to $\mathcal{I}(\lambda, \mu)$.  In particular, $f = \alpha^{-1}(\alpha(f)) \in \mathcal{I}(\lambda, \mu)$.  As explained in the second paragraph, this implies $\Psi_{n}$ is faithful. 
\end{proof}

\begin{remark}\label{R:OBC} Let $\modt$ denote the full monoidal subcategory of all $\fq (n)$-supermodules generated by $V_{n}$ and $V_{n}^{*}$.  Set $\OBC(n)$ to be the monoidal supercategory given by imposing the relation $e_{\lambda(n)}=0$ on $\OBC$.  Since $\OBC$ can be identified as a full subcategory of $\qwebsupdown$ using $\Upsilon$, the above theorem implies there is an equivalence of monoidal supercategories between $\OBC (n)$ and $\modt$.
\end{remark}

\section{Appendix on shifted tableaux}\label{appendix}

\subsection{The shifted Littlewood-Richardson rule}\label{SS:shiftedLRrule}
Let $\A$ denote the ordered alphabet $\A=\{1'<1<2'<2<\cdots\}$. We say the letters $1',2',3',\dots$ are \emph{marked}, and use the notation $\ul{a}$ to denote the unmarked version of any $a\in\A$. 

\begin{definition}\label{shifted-tableau}
Given a strict partition $\lambda$, a \emph{shifted tableau of shape} $\lambda$ is a filling of the boxes of the shifted frame $[\lambda]$ with elements of $\A$ in such a way that
\bi
\item the entries in each row are nondecreasing,
\item the entries in each column are nondecreasing,
\item each row has at most one $a'$ for $a=1,2,3,\dots$, and
\item each column has at most one $a$ for $a=1,2,3,\dots$.
\ei
\end{definition}

An example of a shifted tableau of shape $(4,3,1)$ is \[\begin{ytableau}1&2'&3&3\\ \none&2'&4'&4\\ \none&\none&5'\end{ytableau} \ .\] 

Given a shifted tableau $T$ of shape $\lambda$ and $i=1,2,3,\dots$, let $\nu_i$ be the number of entries $a$ in $T$ such that $\ul{a}=i$. The \emph{content} of $T$ is then defined to be $\nu=(\nu_1,\nu_2,\dots)$. When writing contents we often choose to surpress the trailing zeros. For example, the content of the shifted tableau above is $(1,2,2,2,1)$.

If $\lambda\subseteq\mu$ are strict partitions, then the \emph{skew shifted frame} $[\mu/\lambda]$ is the array of boxes obtained by removing $[\lambda]$ from $[\mu]$. For example, if $\mu=(4,3,1)$ and $\lambda=(3,1)$, then we have $$[\mu/\lambda]=\ydiagram{1+1,2,1} \ .$$ A \emph{shifted tableau of shape} $\mu/\lambda$ is a filling of the boxes of $[\mu/\lambda]$ with elements of $\A$ in such a way that the four conditions of Definition \ref{shifted-tableau} are satisfied.

The \emph{word} $w=w(T)=w_1w_2\cdots$ associated to a (possibly skew) shifted tableau $T$ is the sequence of elements of $\A$ obtained by reading the rows of $T$ from left to right, starting with the bottom row and working up. For example, the word of the shifted tableau above is $5'2'4'412'33$. 

Given a word $w=w_1\cdots w_n$ in the alphabet $\A$, define a series of statistics $m_i(j)$ for $i\in\A$ and $j=1, \dotsc , 2n$ as follows:
\begin{itemize}
\item $m_i(j)=$ multiplicity of $i$ among $w_{n-j+1},\dots,w_n$, for $0\leq j\leq n$, and
\item $m_i(n+j)=m_i(n)$ + multiplicity of $i'$ among $w_1,\dots,w_j$, for $0<j\leq n.$
\end{itemize}
In particular, $m_i(0)$ will be zero for all $i\in\A$.

\begin{definition}\label{lattice-property}
We say a word $w=w_1\cdots w_n$ has the \emph{lattice property} if whenever $m_i(j)=m_{i-1}(j)$ we have 
\begin{itemize}
\item[(1)] $w_{n-j}\neq i,i'$ if $0\leq j<n$, and
\item[(2)] $w_{j-n+1}\neq i-1,i'$ if $n\leq j<2n$.
\end{itemize}
\end{definition}

Let $\ul{w}:=\ul{w_1}\cdots\ul{w_n}$ denote the unmarked version of $w$. We are now ready to state the shifted Littlewood-Richardson rule.  For a strict partition $\lambda$, let $P_{\lambda}$ denote the Schur $P$-function labelled by $\lambda$ (e.g.\ see \cite{St}).  Define $f_{\mu,\nu}^{\lambda}$ by 
\[
P_{\lambda}P_{\nu} =\sum_{\mu} f_{\lambda,\nu}^{\mu} P_{\mu}.
\] 
Stembridge provides the following combinatorial rule for computing these structure constants.
\begin{theorem}\cite[Theorem 8.3]{St}
The coefficient $f_{\lambda,\nu}^\mu$ is the number of shifted tableaux $T$ of shape $\mu/\lambda$ and content $\nu$ such that
\begin{itemize}
\item[(a)] the word $w=w(T)$ satisfies the lattice property, and
\item[(b)] the leftmost $i$ of $\ul{w}$ is unmarked in $w$ for $1\leq i\leq l(\nu)$.
\end{itemize}
We call a shifted tableau $T$ satisfying $(a)$ and $(b)$ a \emph{shifted Littlewood-Richardson tableau}.  %In particular, $L_{n}(\lambda)$ is a direct summand of $L_{n}(\mu) \otimes L_{n}(\nu)$ if and only if there exist shifted Littlewood-Richardson tableau of shape $\mu/\lambda$ and content $\nu$.
\end{theorem}
	
Recall the strict ``staircase'' partition $\lambda(n)=(n+1,n,\dots,2,1)$ as defined in \cref{SS:SergeevAlgebra}.  It is easy to verify every $\mu\in\SP$ with $\ell(\mu)>n$ has $\mu\supseteq\lambda(n)$.

\begin{proposition}\label{LR-tableau}
For every strict partition $\mu$ with $l(\mu)>n$, there exists a strict partition $\nu$ of shape $\mu/\lambda(n)$ such that the shifted Littlewood-Richardson coefficient $f_{\lambda(n),\nu}^\mu$ is nonzero.
\end{proposition}

\begin{proof}  For every such $\mu$, we construct a shifted Littlewood-Richardson tableaux $T_{\mu,n}$ of shape $\mu/\lambda(n)$ whose content $\nu$ is also a strict partition, proving the proposition.

First, we define a \emph{hook} to be a left-justified array of boxes in which only the first row may have more than one box:
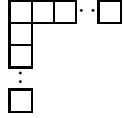
\[
\begin{ytableau}
 \ & \ & \ & \none[\cdot\,\cdot] & \ \\ 
 \\
 \\
 \none[:]\\
 \\
\end{ytableau}
\]
Given the shape of $\lambda(n)$, the skew shape $[\mu/\lambda(n)]$ is an ordinary partition and, hence, can be thought of as consisting of a series of hooks wedged inside each other with the corner of each hook lying on the diagonal of $[\mu/\lambda(n)]$. We number the hooks of $[\mu/\lambda(n)]$ from the upper left to the lower right.

We define $T_{\mu,n}$ to be the shifted tableau of shape $[\mu/\lambda(n)]$ whose $i$-th hook has the form
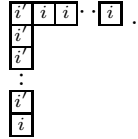
\[
\begin{ytableau}
i' & i & i & \none[\cdot\,\cdot] & i \\ 
i' \\
i' \\
 \none[:]\\
 i' \\
 i \\
\end{ytableau}
 \ .\]
The unmarked $i$ at the very bottom takes priority over all of the $i'$, so that if the $i$-th hook has only one row then every entry will be $i$. For example, if  $\mu=(8,5,4,2)$ and $n=2$, then we have
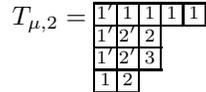
\[
T_{\mu,2}=
\begin{ytableau}
 1' & 1 & 1 & 1 & 1 \\
 1' & 2' & 2 \\
 1' & 2' & 3 \\
 1 & 2 \\
\end{ytableau}
\]
whose corresponding word is $w(T_{\mu,2})=121'2'31'2'21'1111$.

Clearly $T_{\mu,n}$ satisfies the conditions of Definition \ref{shifted-tableau} and is a shifted tableau. Since $\nu_i$ is just the number of boxes in the $i$-th hook, the content $\nu=(\nu_1,\nu_2,\dots)$ of $T_{\mu,n}$ is a strict partition. Indeed, at its largest, the $(i+1)$-st hook extends as far to the right and one box farther down than the $i$-th hook, which is precisely when $\nu_{i+1}=\nu_i-1$; otherwise $\nu_{i+1}<\nu_i-1$. And, as was previously observed, property (b) of the shifted Littlewood-Richardson rule is satisfied.

It remains to show that $T_{\mu,n}$ has the lattice property. Let $i\geq 2$ and let $l$ denote the length of the word $w=w(T_{\mu,n}).$ Suppose $m_i(j)=m_{i-1}(j)$ for $0\leq j<l$. If $m_i(j)=m_{i-1}(j)=0$, then $w_{l-j}\in\{1,1',2,2',\dots,i-1,(i-1)'\}$ so $w_{l-j}\neq i,i'.$ The case of $m_i(j)=m_{i-1}(j)>0$ occurs precisely if the width of the $(i-1)$-st hook exceeds that of the $i$-th hook by one box (e.g. the 2nd and 3rd hooks in the example). In that case $w_{l-j}$ lies in the $(i-1)$-st hook, so $w_{l-j}\in\{i-1,(i-1)'\}$ and $w_{l-j}\neq i,i'.$ Thus $T_{\mu,n}$ satisfies condition (1) of the lattice property. Since $m_i(l)\leq m_{i-1}(l)-1$, and between every pair of $i'$ in $w$ is at least one $(i-1)'$, we have $m_i(j)\neq m_{i-1}(j)$ for $l\leq j\leq 2l.$ Thus $T_{\mu,n}$ satisfies condition (2) of the lattice property, completing the proof.

\end{proof}

Recall  the character of the simple $\fq (m)$-supermodule labelled by the strict partition $\lambda$, $L_{m}(\lambda)$, has character given by a power of two multiple of the Schur P-function $P_{\lambda}(x_{1}, \dotsc , x_{m})$ (e.g.\ see \cite[Theorem 3.48]{CW2}).  Furthermore, tensor products of $\fq (m)$-supermodules corresponds to the multiplication of functions.  Consequently, $f_{\lambda,\nu}^{\mu} \neq 0$ if and only $L_{m}(\mu)$ is a direct summand of $L_{m}(\lambda) \otimes L_{m}(\nu)$.

\begin{corollary}\label{sum-multiplicity}
Fix a strict partition $\mu$ with $n<\ell(\mu)$ and fix a $m \geq \ell(\mu)$. Then there is a strict partition $\nu$ such that the simple $\q(m)$-supermodule $L_m(\mu)$ is a direct summand of the tensor product $L_m(\lambda(n))\otimes L_{\mu} (\nu)$ and, hence, is a direct summand of $L_m(\lambda(n))\otimes V_{m}^{|\nu|}$.
\end{corollary}

\begin{proof}  By \cref{LR-tableau} there exists a strict partition $\nu$ with  $|\nu|= |\mu| - |\lambda(n)|$, of shape $\mu/\lambda(n)$, and with $f_{\lambda(n), \nu}^{\mu} \neq 0$. This implies $L_{m}(\mu)$ is a direct summand of $L_{m}(\lambda(n)) \otimes L_{m}(\nu)$, as asserted.  The final statement follows from the fact that $L_{m}(\nu)$ is a direct summand of $V_{m}^{|\nu|}$ by \cref{T:sergeev-duality}.
\end{proof}

\end{document}